\documentclass[letterpaper,final,leqno,onefignum,onetabnum]{siamltex1213}

\usepackage{amsmath,amssymb,amscd,fancybox}
\usepackage{ifthen,float}
\usepackage{afterpage}
\usepackage{graphics}

\usepackage{subcaption}
\begingroup
    \makeatletter
    \@for\theoremstyle:=definition,remark,plain\do{
        \expandafter\g@addto@macro\csname th@\theoremstyle\endcsname{
            \addtolength\thm@preskip\parskip\onehalfspacing}}
\endgroup

\usepackage{color}
\usepackage{parskip}
\usepackage{comment}

\newcommand\cV{\mathcal V}

\newcommand\cP{\mathcal P}

\newcommand\bnu{\boldsymbol \nu}
\newcommand\bc{\boldsymbol c}

\newcommand\bj{\boldsymbol j}

\newcommand\tPhi{\widetilde \Phi}

\newcommand\hpsi{\widehat \psi}

\newcommand\br{\boldsymbol r}

\newcommand\bx{\boldsymbol x}
\newcommand\by{\boldsymbol y}

\newcommand\bi{\boldsymbol i}
\newcommand\bd{\boldsymbol d}

\newcommand\cF{\mathcal F}
\newcommand\cE{\mathcal E}

\newcommand\cC{\mathcal{C}}
\newcommand\cL{\mathcal{L}}

\renewcommand\Re{\operatorname{Re}}
\renewcommand\Im{\operatorname{Im}}

\newcommand\bbR{\mathbb R}

\newcommand\pa{\partial}

\newcommand\restrictedto{\upharpoonright}

\title{Smoothed corners and scattered waves}

\author{Charles L. Epstein\thanks{Departments of Mathematics and
    Radiology, University of Pennsylvania, Philadelphia, PA
    (\email{cle@math.upenn.edu}). Research supported in part by the
    NSF under grants DMS09-35165, DMS12-05851, and DMS-1507396, the Army
    Research Office under grant W911NF-12-1-0552, and by the Office of
    the Assistant Secretary of Defense for Research and Engineering
    and AFOSR under NSSEFF Program Award FA9550-10-1-0180.}  \and
  Michael O'Neil\thanks{Department of Mathematics, Courant Institute
    and School of Engineering, New York University, New York, NY
    (\email{oneil@cims.nyu.edu}). Questions, comments, or corrections
    to this document may be directed to that email address. Research
    supported in part by the Office of the Assistant Secretary of
    Defense for Research and Engineering and AFOSR under NSSEFF
    Program Award FA9550-10-1-0180.} }

\begin{document}
\maketitle
\slugger{sisc}{xxxx}{xx}{x}{x--x}

\begin{abstract}
  We introduce an arbitrary order, computationally efficient method to
  smooth corners on curves in the plane, as well as edges and vertices
  on surfaces in $\bbR^3.$ The method is local, only modifying the
  original surface in a neighborhood of the geometric singularity, and
  preserves desirable features like convexity and symmetry.  The
  smoothness of the final surface is an explicit parameter in the
  method, and the bandlimit of the smoothed surface is proportional to
  its smoothness. Several numerical examples are provided in the
  context of acoustic scattering. In particular, we compare scattered
  fields from smoothed geometries in two dimensions with those from
  polygonal domains. We observe that significant reductions in
  computational cost can be obtained if merely approximate solutions
  are desired in the near- or far-field. Provided that the smoothing is
  sub-wavelength, the error of the scattered field is proportional to
  the size of the geometry that is modified.
\end{abstract}



\pagestyle{myheadings}
\thispagestyle{plain}
\markboth{CORNERS AND WAVES}{C. L. EPSTEIN AND M. O'NEIL}

\numberwithin{equation}{section}

\section{Introduction}

In the numerical solution of boundary value problems for partial
differential equations an especially difficult case arises when the
boundary of the domain has corners (in two dimensions) or edges and
vertices (in three dimensions). Several groups have devoted resources
to solving this problem and have made serious inroads towards
addressing these issues in the context of the classical integral
equations of mathematical physics (acoustic and electromagnetic
scattering, elasticity, etc.) \cite{bremer_2013, bremer_2012,
  bremer_2010, helsing, helsing_2011, helsing_2011b, bruno_2009,
  bremer_2010b, kress_1990, serkh_2016}.  
The resulting numerical schemes often
involve the use of specially designed quadratures which handle not
only singular or weakly-singular integrals but also singular layer
potential densities.  These methods are based on several standard
ideas in modern numerical analysis, namely low-rank
approximations~\cite{cheng_2005, bremer_2012b}, generalized Gaussian
quadratures and adaptive refinement~\cite{yarvin_1998,bremer_2010b},
and (semi-) analytic product integration formulae~\cite{helsing_2014,
  helsing_2014b, hao_2014}. 
More recently, explicit 
exact forms of the solutions to actual layer potential densities
were derived in~\cite{serkh_2016}.
All the numerical tools just mentioned are
now well-developed and require minimal sophistication to use, but can
still be time-consuming, or too special-purpose to
implement. An approach that has not been investigated thoroughly at
this time in the literature is that of solving an analogous scattering
problem from a smoother geometry that is \emph{close} to the original one. In
what follows, by \emph{close} we mean different only in a (small)
neighborhood of the geometric singularities (e.g. corners in
two dimensions).

Lacking, up to this time, is a reliable, systematic, computationally
simple method for smoothing such irregularities that also retains
desirable geometric features, such as convexity or local symmetries.
In this note we discuss several methods for doing this, including a
simple convolution method as well as introduce a new geometric
method, particularly useful for regularizing surfaces in
three dimensions.
Our
methods are tailored for use with polygons in two dimensions and
polyhedra in three dimensions. 
However, because the modifications are done
locally, this approach can be applied to more general shapes (namely
curves which intersect at their endpoints) through composition with
diffeomorphisms. Indeed, our method already employs such compositions
in the three dimensional case.

Corner and edge rounding methods are useful
for two reasons. 
First, in the context of the solution of scattering problems
via integral equations, smoothing geometric singularities on a
sub-wavelength scale provides a means by which to apply standard
numerical quadratures~\cite{alpert, klockner_2013, helsing_2011} for
weakly-singular integrals along smooth boundaries, instead of the more
complicated schemes required in the neighborhood of corners.
In two dimensions, our numerical examples show that 
convergence is roughly first-order in the scattered field,
 both in the near- and far-fields.  
In three dimensions, reducing the number of discretization nodes is
particularly useful because of the relative cost of even the fastest
solvers. 
State of the art, high-order accurate 
solvers in three dimensions include those by
Bremer, Gillman, Gimbutas, and Martinsson~\cite{bremer_2012c,bremer_2015} 
and Bruno~\cite{bruno_2012}.

Second, since the schemes to be presented only change the
geometry locally, they may lead to a new class of algorithms which can
be incorporated into modern computer-aided design (CAD) and
engineering (CAE) software packages.  The regularity of the smoothed
surface can be precisely controlled in the neighborhood of the
singularity. Applications in fine-grained polishing of machined
mechanical parts are straightforward.  This paper  investigates the
advantages and disadvantages of solving a scattering problem from a
nearby smoothed geometry instead of the original non-smooth one.

We organize the remainder of the paper as follows.
Section~\ref{sec-review} reviews standard integral equation
formulations of acoustic scattering phenomena, as well as both the
analytic regularity results of scattering from geometries with corners
and the numerical techniques that have been developed to compute
them.  Section~\ref{sect2d} describes a straightforward and systematic
way to smooth the corners of polygons in two dimensions. The method
can be extended to regions with piecewise smooth boundaries via the
application of a diffeomorphism.  Several numerical experiments are
presented to illustrate the heuristics of the approach.
Section~\ref{sec-3d} details several
methods for smoothing polyhedra in three dimensions. The methods of
Section~\ref{sect2d} are extended to three dimensions and a new
geometric method is introduced which is applicable in most cases.
Section~\ref{sec-general} puts all the previous sections together and
gives a recipe for smoothing a general polyhedron in three dimensions.
Section~\ref{sect4} reviews some analytical methods that can be used
to construct the diffeomorphisms required by the three-dimensional
methods of Section~\ref{sec-3d}.  Lastly, the conclusions in
Section~\ref{sec-conclusions} discuss drawbacks and difficulties
with our method, as well as points to future areas of research and
applications.  Numerical experiments are included throughout the paper
to demonstrate the application of scattering from smoothed geometries,
as well as to visually describe the results of the smoothing
techniques.

\section{Scattering in singular geometries} \label{sec-review}

There are  two questions that require answers when studying
scattering (acoustic, electrostatic and electromagnetic, etc.) in
singular geometries using integral equations. First, in the
neighborhood of a corner or edge, {\em what regularity can we expect
  in the solution for data with a given smoothness?} And second, {\em
  if a solution exists, which can be represented in terms of a
  layer-potential density, is the density continuous and how can it be
  numerically calculated?} The first question has been studied in
detail by Dauge, etc \cite{costabel_2000, dauge_1988}. The latter
question is mainly an exercise in numerical integration, and has been
thoroughly studied by Bremer, Bruno, Helsing, etc. See
\cite{bremer_2013, bruno_2009, helsing} for more details.  Often, the
numerical solution is a combination of sophisticated quadrature
schemes coupled with an adaptive discretization of the geometry (in
order to correctly resolve complicated layer potential densities).  We
now give a very brief review of some results in both of these areas.

\subsection{An integral equation approach}

Almost all of the classical partial differential equations of
mathematical physics can be reformulated in an equivalent integral
equation form~\cite{guenther_lee}. The integral equation form has many
advantages, namely the direct handling of unbounded domains in the
case where the solution of a PDE reduces to a boundary integral
equation.  Furthermore, when the integral equation is Fredholm of the
second kind, as is often the case, provable bounds exist on the
accuracy of the solution which are directly related to the order of
the quadrature rule used in the discretization~\cite{anselone_1964,
  anselone_1971}. In this section,  
we summarize a basic Nystr\"om-type discretization
of an integral equation for the Helmholtz equation
 that can be used to solve an exterior acoustic scattering problem.

Time-harmonic acoustic wave propagation in
homogeneous free-space (we address the two-dimensional version here)
is governed by the Helmholtz equation,
\begin{equation}\label{eq-helmholtz}
\left( \Delta + k^2 \right) u(\bx) = 0 \qquad \text{in } \bbR^2,
\end{equation}
where $u$ is related to the acoustic pressure and $k$ is related to
the wavenumber of the field, namely $k = \omega/c$, where $\omega$ is
the angular velocity and $c$ is the speed of sound in the medium.  In
particular, often one is interested in the solution to a {\em
  scattering} problem in the presence of some inclusion $\Omega$,
where the total pressure field $u^{tot}$ is the sum of an incoming
field $u^{inc}$ and a scattered field $u$.  If the boundary of the
inclusion is given by $\Gamma$, then {\em sound-hard} scattering
phenomena can be formulated as the following boundary value problem:
\begin{equation}\label{eq-hard}
  \begin{aligned}
    \left( \Delta + k^2 \right) u^{tot}(\bx) &= 0 &\qquad &\text{in } \bbR^2
    \setminus \Omega, \\
    \frac{\partial u^{tot}(\bx)}{\partial n} &= 0 &\qquad &\text{on }
    \Gamma,
  \end{aligned}
\end{equation}
where $\partial/\partial n$ represents the derivative with respect
to the outward normal to $\Gamma$. This boundary value problem is
also known as the Neumann scattering problem. Dirichlet boundary
conditions $u^{tot} = 0$ along $\Gamma$ correspond to sound-soft
scattering problems.
The solution
to~\eqref{eq-hard} is unique under a suitable decay condition, known
as a Sommerfeld radiation 
condition, on the scattered field $u$. In particular, in two
dimensions, the scattered field $u$ must satisfy:
\begin{equation}
\lim_{|\bx| \to \infty} \sqrt{|\bx|} \left( \frac{\partial }{\partial r}
u(\bx) - ik u(\bx) \right)=0, 
\end{equation}
and $\partial/\partial r$ is understood to be differentiation in the
radial direction. It is well-known that the Green's function
for~\eqref{eq-helmholtz} is given in terms of the zeroth order Hankel
function of the first kind, $H_0^{(1)}$, and is normalized as:
\begin{equation}
  g_k(\bx)  = \frac{i}{4} H_0^{(1)}\left( k |\bx| \right).
\end{equation}
Using this Green's function, a solution to~\eqref{eq-hard} can be
expressed in terms of a single-layer potential
\begin{equation}
  \begin{aligned}
    u(\bx) &= \mathcal S_k[\sigma](\bx) \\
      &=\int_\Gamma g_k(|\bx - \by|) \, \sigma(\by) \, ds(\by),
  \end{aligned}
\end{equation}
where $s$ is arclength along $\Gamma$. After taking the proper limit
as $\bx \to \Gamma$ from the exterior, this representation results
in the second-kind integral equation for the density
$\sigma$:
\begin{equation}
\frac{1}{2} \sigma + \mathcal S_k'[\sigma] = - 
  \frac{\partial }{\partial n} u^{inc} \qquad \text{on } \Gamma,
\end{equation}
or more explicitly,
\begin{equation}\label{eq-cont}
\frac{1}{2}\sigma(\bx) + \int_\Gamma \left[ \frac{\partial}{\partial
n_x} g_k(\bx, \by) \right] \sigma(\by) \, ds(\by) = -
\frac{\partial }{\partial n} u^{inc}(\bx) \qquad \text{for } \bx \in
\Gamma.
\end{equation}
The operator $\mathcal S_k'$ represents the normal derivative of a
single-layer potential.  If $\Gamma$ is $\cC^1,$ then the integral
in~\eqref{eq-cont} is weakly-singular and can be evaluated using
specially designed quadrature rules~\cite{hao_2014}.  There are
several approaches to discretizing the continuous integral
equation~\eqref{eq-cont}, namely Galerkin, collocation, qualocation,
and Nystr\"om discretizations ~\cite{dahlquist, atkinson_2009}. The
methods of this paper apply to all of these approaches (under suitable
small changes); we briefly describe the Nystr\"om method for its
simplicity.

The Nystr\"om discretization of~\eqref{eq-cont} replaces continuous
functions and integrals by samples and sums of samples. Namely, for a
given quadrature rule consisting of nodes and weights
$\{\bx_j,w_{j\ell} \}$  for the
integral appearing in~\eqref{eq-cont}, we approximate the solution
$\sigma(\bx_j) \approx \sigma_j$ at each node $\bx_j$ as the solution
to the system of equations:
\begin{equation}\label{eq-discrete}
\frac{1}{2} \sigma_j + \sum_{\ell} w_{j\ell} \, \frac{\partial}{\partial
n_{x_j}} g_k(\bx_j, \bx_\ell) \, \sigma_\ell= -u^{inc}(\bx_j),
\end{equation}
for all $j$.  Here we have explicitly shown that the quadrature
weights can be a function of the {\em outgoing} node $\bx_j$. As the
number of discretization points or order of quadrature increase,
$\sigma_j$ approaches the exact solution $\sigma(\bx_j)$.  The
previous linear system can be solved directly if the resulting linear
system is small enough, or for larger systems using iterative (fast
multipole methods and GMRES, etc.)~\cite{wideband2d, saad_1986} or
fast direct solvers \cite{gillman_2012, ho_2012}.

It should be noted that integral equation~\eqref{eq-cont} fails to be
uniquely solvable at a discrete set of $k$'s, known as
spurious resonances. This is not a failure of the uniqueness
properties of the PDE, but rather a failure
in the particular choice of integral representation. Choosing what is known as
a {\em combined-field} representation can result in a uniquely
solvable integral equation, albeit at the cost of a slightly more
complicated formulation~\cite{colton_kress, bruno_2012}. One possible
combined-field (or regularized) representation
of this type is of the form:
\begin{equation}\label{eq_neu_cfie}
  u = \mathcal S_k[\sigma] + \alpha \, \mathcal D_k \mathcal S_0
  [\sigma],
\end{equation}
where $\alpha$ is a user-chosen complex-valued parameter, 
$\mathcal D$ is known as the double-layer potential, given by
\begin{equation}
  \mathcal D_k[\sigma](\bx) = \int_\Gamma \left[ \frac{\partial}{\partial
n_y} g_k(\bx, \by) \right] \sigma(\by) \, ds(\by),
\end{equation}
and $\mathcal S_0$ is a single-layer potential corresponding to the
Green's function for Laplace's equation:
\begin{equation}
  \mathcal S_0[\sigma](\bx) = \int_\Gamma \frac{1}{2\pi}
  \ln\frac{1}{|\bx-\by|}
  \, \sigma(\by) \, ds(\by).
\end{equation}
There are many other regularizations that one may use, and this is the
subject of ongoing research (especially in the large-$k$ regime).  We
make a point to explicitly state the form of the integral
representation for numerical experiments appearing later in the paper.

\subsection{Analytic results in singular geometries}

In the previous section we discussed the process by which the
Helmholtz boundary value problem~\eqref{eq-hard} for the field $u$ is
reformulated as a boundary integral equation for a separate unknown
layer potential density $\sigma$. We have not, however, discussed the
effect that the geometry has on the solution $\sigma$ (assuming that
the data $u^{inc}$ is smooth). The regularity of the solution $\sigma$
to the integral equation is strongly affected by the
presence of corners on the boundary $\Gamma$, the boundary data, and 
details of the local geometry, e.g. whether the corners are
re-entrant, acute, obtuse, etc.

On smooth domains, the layer potential operators $\mathcal S_k$,
$\mathcal S'_k$, and $\mathcal D_k$ are compact, classical
pseudodifferential operators and therefore the invertibility of the
associated second-kind integral equation follows from the Fredholm
alternative~\cite{fabes_1977, fabes_1978}. The mapping properties on
Sobolev and H\"older spaces are well-known and essentially optimal.
However, when the domain is merely continuous and not everywhere
differentiable,
these operators cease to be compact.  While canonical PDE results have
existed for some time, it is a relatively recent result in functional
analysis that the classical integral equation corresponding to the
interior Dirichlet problem for Laplace's equation, namely
\begin{equation}
\frac{1}{2} \rho(\bx) + \int_\Gamma \left[ \frac{\partial}{\partial
n_y} \frac{1}{2\pi} \ln\frac{1}{| \bx - \by |}
\right] \rho(\by) \, ds(\by) = f(\bx), \qquad \text{for } \bx \in
\Gamma,
\end{equation}
where $\Gamma$ bounds some Lipschitz domain $D$, is invertible on $\mathcal
L_2$~\cite{verchota_1984}.  Similar results exist for the Neumann problem as
well, and, with some work, extend to the analogous integral equations in the
Helmholtz case~\cite{costabel_1988, jerison_1981}.

Classically, representations for solutions to the Helmholtz
equation can be obtained in the exterior of a wedge or corner by using
fractional Bessel function expansions, as in~\cite{keller_1951}. An
expansion of this type, however, does not immediately yield similar
statements concerning the density, $\sigma$.
Very recently, however, expansions of the actual density (at least in
the Laplace case) were
derived that allow for the construction of very efficient, most likely
optimal, solvers~\cite{serkh_2016}.
The topic has been further studied by many in the finite element,
asymptotics, and analysis communities including, but certainly not
limited to, Buffa, Ciarlet, Costabel, Dauge, and others~\cite{dauge_1988, 
costabel_2000, buffa_2002}. This classical work  addresses solutions to the
Helmholtz equation and Maxwell's equations, as well as Stokes flow in
fluid dynamics and elasticity.

\subsection{Numerical methods for Lipschitz domains}

While the results of the previous section are  interesting from
a mathematical standpoint, and certainly offer insights on how to properly
construct finite element methods that have desirable properties in singular
geometries, they offer little help in the construction of numerical quadrature
schemes that can be used efficiently in the Nystr\"om method solution for the
associated boundary integral equation.  Recently there have been several papers
addressing the question of constructing (mostly brute force)
discretization schemes for boundary integral equations on polyhedral domains or
domains with corners.  As mentioned before, these schemes are often a
combination of adaptive refinement of the geometry near the singular set, the
design of specialized quadratures, and proper re-weighting of the unknown
density.

Adaptive or dyadic refinement of the geometries and density near
geometric singularities has been commonplace for some
time, but it was only recently detailed 
how to embed the Nystr\"om discretization into
the proper continuous function space in order for the spectrum of the
finite-dimensional approximation to converge to the spectrum of the
continuous integral equation~\cite{bremer_2012}. We omit a discussion
of the dyadic refinement methods since they are well-known
and~\cite{helsing} offers a nice review. However, we  briefly
mention the $\cL_2$ norm-preserving scheme discussed by Bremer.

First, it should be pointed out that the unknowns in the discrete
system~\eqref{eq-discrete} are {\em point values} of the continuous density
$\sigma$. Much of the theory developed for integral equations makes use of the
$\mathcal L_2$ properties of the data and solution, but this is at odds with
the system~\eqref{eq-discrete}. As a higly non-uniform mesh is refined, the
$\ell_2$-norm of the vector $\boldsymbol \sigma = (\sigma_1 \cdots \sigma_n)^t$
becomes increasingly incomparable to the $\mathcal L_2$ norm of the solution to the
continuous integral equation~\eqref{eq-cont}. For a set of quadrature weights
$\{ h_j\}$ which accurately integrate $\sigma$ and $\sigma^2$, the proper
discrete unknown should therefore be $\tilde\sigma_j = \sqrt{h_j} \, \sigma_j$
so that
\begin{equation}
  \begin{aligned}
    \lVert \tilde{\boldsymbol \sigma} \rVert_{\ell_2} &= 
      \sum_j \tilde\sigma_j^2 \\
      &= \sum _j \left( \sqrt{h_j} \, \sigma_j
    \right)^2\\
    &= \sum_j h_j \, \sigma_j^2 \\
    &\approx \int_\Gamma \sigma^2(\bx) \, ds(\bx) 
    = \lVert \sigma \rVert_{\mathcal L_2}.
  \end{aligned}
\end{equation}
Intuitively, this embedding properly scales the unknown $\sigma_j$
according to the clustering of the discretization along $\Gamma$.
This re-weighting enables us to replace the discrete system
in~\eqref{eq-discrete} with:
\begin{equation}
\frac{1}{2} \sqrt{h_j} \, \sigma_j + 
\sum_\ell \frac{\sqrt{h_j} \, w_{j\ell}}{\sqrt{h_\ell}} \, 
\frac{\partial}{\partial
n_{x_j}} g_k(\bx_j, \bx_\ell) \, \sqrt{h_\ell} \, \sigma_\ell
= -\sqrt{h_j} \, u^{inc}(\bx_j),
\end{equation}
and declare $\tilde{\boldsymbol \sigma}$ to be the new unknown.
There is no reason to assume that the $h_j$'s and the $w_{j\ell}$'s
are the same, however in practice they are very similar except near the
singularity of $g_k$. 

Under this re-weighting, the condition number of the discrete system
converges to the condition number of the continuous problem as the
mesh size tends to zero. If the curve $\Gamma$ has corners, then,
under refinement, the condition number of the original
system~\eqref{eq-discrete} will usually diverge. For a thorough
discussion and many results concerning this idea,
see~\cite{bremer_2012}. This norm-preserving embedding is one of the
main tools used to construct high-order accurate boundary integral
equation codes in complicated and singular geometries for both the
Dirichlet and Neumann problems.  Similar ideas
with regard to $\mathcal L_1$-embedding have recently been used for
divergence-form differential equations with high-contrast background
media~\cite{askham}.  Often these re-weighting techniques alleviate
the need for designing specialized {\em corner quadratures} that are
able to integrate singular Green's functions multiplied by singular
densities which diverge in the corner~\cite{bremer_2010c, bremer_2010,
  bremer_2010b, kolm}.

It is with the previous section in mind that we begin to investigate
the relationship between the solution of a scattering problem on a
polygonal domain with that of a {\em nearby} smooth domain.  In the
next section we describe a simple convolution-based method to smooth
polygons, and then report on the relationship between the numerical
solutions to scattering problems in the smoothed and singular
geometries.

\section{Smoothing polygons in 2-dimensions} \label{sect2d}

An obvious approach to smoothing polygons is to locally represent the
polygon as a graph and convolve with a smooth, compactly-supported
even function with some specified order of differentiability. However
obvious, this technique seems not to have been analyzed or reported in
the literature. We use \emph{smooth} to mean that the function is
band-limited to some specified order. We  restrict our attention
to closed domains in two dimensions because of the emphasis on
applications to scattering problems. Scattering from open surfaces
requires several other numerical and analytical
tools~\cite{bruno_2012b, jiang_2003, jiang_2004, lintner_2015}.  
Convolutional
smoothing is an effective method in two dimensions due to the
following elementary lemma:
\begin{lemma}
  Let $\varphi(x)$ be an even, integrable function, with compact
  support and total integral $1.$ For any $a,b\in\bbR$ we have
  \begin{equation}
    \int\limits_{-\infty}^{\infty}\varphi(y) \left( a(x-y)+b
    \right)
    \, dy = ax + b.
  \end{equation}
\end{lemma}
\begin{proof} This follows immediately from the observation that
  \begin{equation}
    \int\limits_{-\infty}^{\infty}\varphi(y) \, y \, dy=0.
  \end{equation}
\end{proof}

More importantly, 
this theorem remains true in $n$ dimensions. If $\varphi$ is a now a
even function of $n$ variables, with total integral $1$, then a simple
application of Fubini's theorem shows that convolving $\varphi$ with a
linear function simply reproduces that function.

In what follows, let a polygon $\mathcal P \in \bbR^2$ be described by
an ordered set of $n+1$ vertices $\{ v_j \}$ and $n$ edges $\{ e_j \}$
such that $v_1 = v_{n+1}$.  Each edge $e_j$ is defined by the set $\{
v_j, \, v_{j+1}\}$.  In a sufficiently small neighborhood of a
particular vertex $v$, the polygon can be represented as an
\emph{even} graph of some function $f_v$ over a support line at
$v$. We can normalize coordinates so that $x=0$ corresponds to the
vertex, with $f_v(0)=0$.  Then, for some $\delta>0$, the function
$f_v$ is linear on intervals $[-\delta,0]$ and $[0,\delta]$. See
Figure~\ref{fig_fv} for a plot of this configuration.  Suppose that
our convolution kernel $\varphi$ is supported on $[-1,1]$, then for
some $0<h<\delta/2$ let
\begin{equation}
  \varphi_h(x)=\frac{1}{h}\varphi\left(\frac xh\right),
\end{equation}
and set
\begin{equation}
  f_v^h(x)=\int\limits_{-h}^{h} \varphi_h(y) \, f_v(x-y) \, dy.
\end{equation}
From the lemma, it is clear that
\begin{equation}
  f_v^h(x)=f_v(x) \quad \text{ if } |x|\geq h.
\end{equation}
Hence the graph of $f_v^h$ defines a smooth (with band-limit dependent
on that of $\varphi$) curve that agrees with the graph of $f_v$
outside an neighborhood of the vertex of size $h$. See
Figure~\ref{fig_fv2} for a depiction.

\begin{figure}[t]
    \centering
    \begin{subfigure}[t]{.4\linewidth}
        \centering
        \includegraphics[width=.95\linewidth]{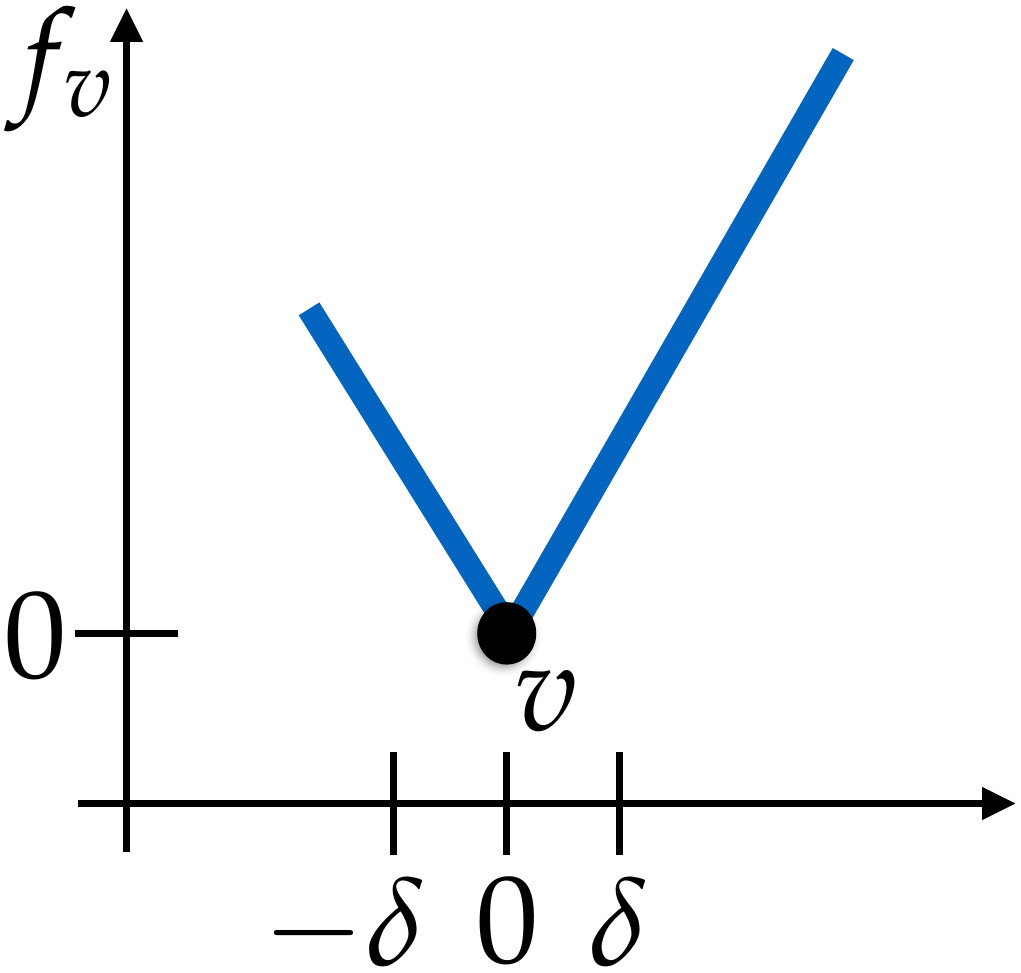}
        \caption{The rotated and translated graph of the neighborhood of a
          vertex of a polygon.}\label{fig_fv}
    \end{subfigure}
    \qquad
    \begin{subfigure}[t]{.4\linewidth}
        \centering
        \includegraphics[width=.95\linewidth]{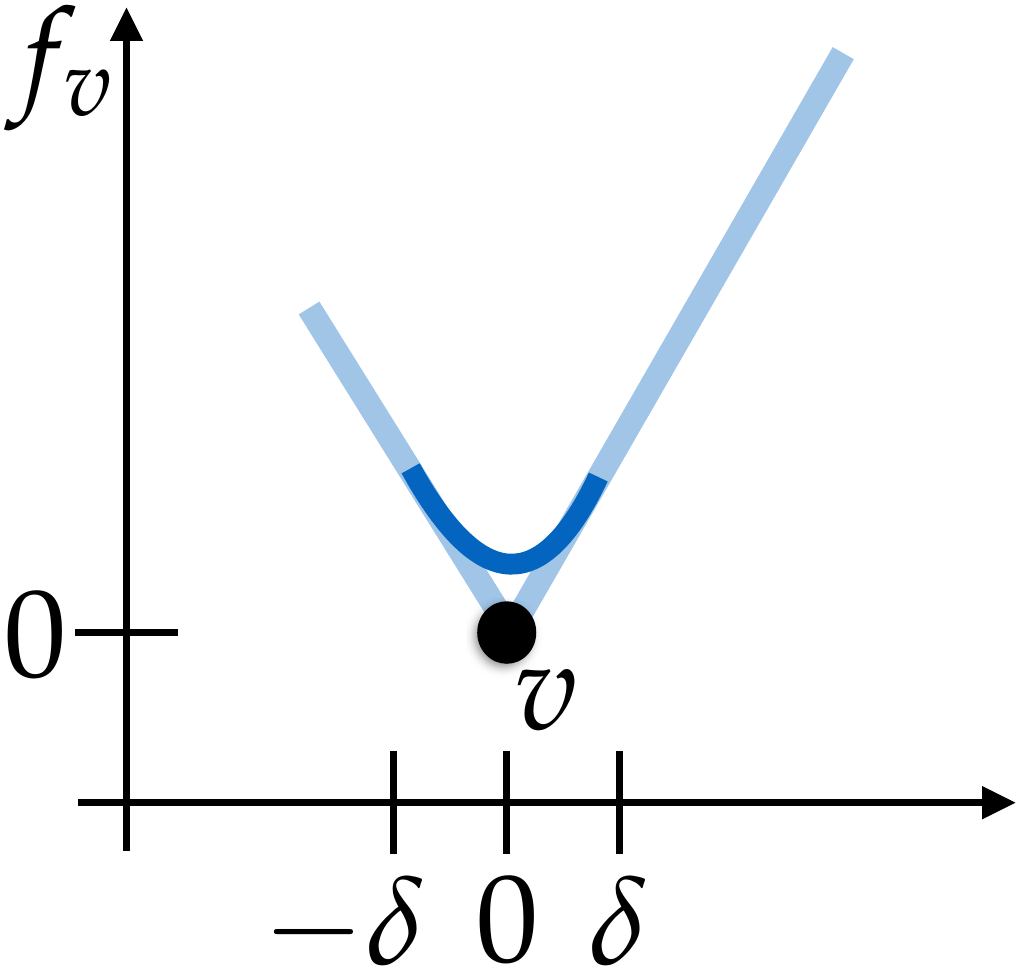}
        \caption{The smoothed vertex of the graph $f_v$.}\label{fig_fv2}        
    \end{subfigure}
    \caption{The basic configuration for smoothing around a vertex.}
\end{figure}

This gives an effective means to smooth the vertices of the polygon
$\mathcal P$; since only a neighborhood of each vertex is changed,
they can be smoothed locally and then glued together along the remaining
straight edges. If the interior angle at a vertex is less than $\pi,$
then the smoothed vertex lies inside of the original polygon, whereas
if it is larger than $\pi,$ then the smoothed vertex lies in the
exterior.

The following simple algorithm can be used to uniformly smooth the
polygon $\mathcal P$ with a given smooth, even function $\varphi$, with
support in $[-1,1]$.

\begin{quote}
\begin{center}
{\bf Algorithm for polygonal smoothing via convolution}
\line(1,0){350}
\end{center}
\begin{enumerate}
\item[{\bf Step 0:}] Choose a smoothing parameter $h>0$, smaller \\ than
$\frac{1}{2}\min\{|v_j-v_{j+1}|:\: j=1,\dots,n\}$.
\item[{\bf Step 1:}] For each $j$, represent a neighborhood of the
  vertex $v_j$ as the graph of an even piecewise linear function $f_j$
  over a support line to $\cP$ at $v_j.$
\item[{\bf Step 2:}] Convolve the functions $f_j$ with $\varphi_h,$ to
  obtain $f_j^h$.
\item[{\bf Step 3:}] Replace a neighborhood of $v_j$ with part of the
  graph of $f_j^h$ by gluing along the linear parts of the graph of
  $f_j^h$, which agree with the graph of $f_j$.
\end{enumerate}
\begin{center}\line(1,0){350}\end{center}
\end{quote}
{\em Remark}. The reason to use an \emph{even} linear function in Step
1 is to insure that the smoothed polygon has the same discrete
symmetries as $\cP.$

The convolution can be done efficiently via either closed-form
analytic expressions (depending on the choice of kernel $\varphi$) or
by high-order numerical integration using an adaptive discretization
scheme of the polygon and kernel as discussed in more detail in
Section~\ref{sec_scattering}.  Furthermore, an adaptive smoothing
algorithm can be constructed by which the width parameter $h$ is
allowed to depend on the pairwise vertex spacing $|v_j - v_{j+1}|$.

\begin{figure}[t]
\centering
\includegraphics[width=.55\linewidth]{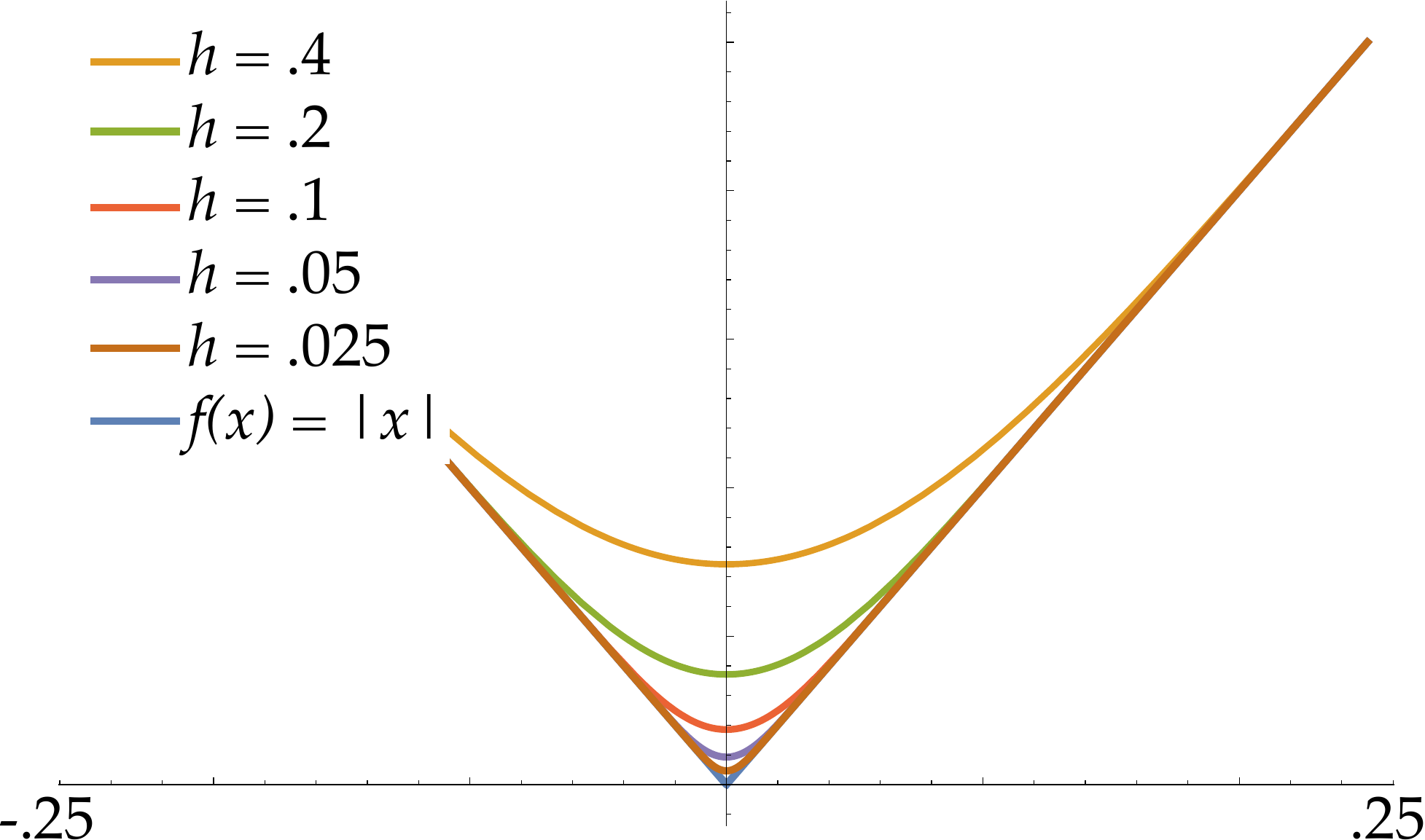}
\caption{A range of smoothings of a $\pi/2$ corner done by convolving
  a local representation with $\psi_k^h,$ with $k=8$ and $h=0.025$,
  $0.05$, $0.1$, $0.2$, $0.4$.}\label{fig20}
\end{figure}

\subsection{Selection of smoothing kernels}
To make this an effective method requires the choice of a good family
of smoothing kernels. We briefly discuss the details concerning 
two such kernels, one compactly supported and the other 
\emph{numerically} compactly supported. Let us first examine the family
of functions $\psi_k(x) \in\cC^{k-1}(\bbR)$,
\begin{equation}\label{eq_kerncompact}
  \psi_k(x)=c_k \, (1-x^2)^k \, \chi_{[-1,1]}(x),
\end{equation}
where $\chi_{[a,b]}$ is the indicator function on the interval
$[a,b]$.  These functions should be familiar from undergraduate
analysis, and are well-suited to convolutional  smoothing.  Here $c_k$
is chosen so that $\psi_k$ has total integral $1$. 
In fact,
\begin{equation}
  \psi_k(x)= \Gamma\left( k+\frac{3}{2} \right) 
\frac{ (1-x^2)^k }{\sqrt{\pi} \, \Gamma(k+1)} 
 \, \chi_{[-1,1]}(x).
\end{equation}
 An example of
smoothing a right-angled vertex using this kernel is shown in
Figure~\ref{fig20}.

\begin{figure}[t]
  \centering
  \begin{subfigure}[t]{.45\linewidth}
    \centering
    \includegraphics[width=.95\linewidth]{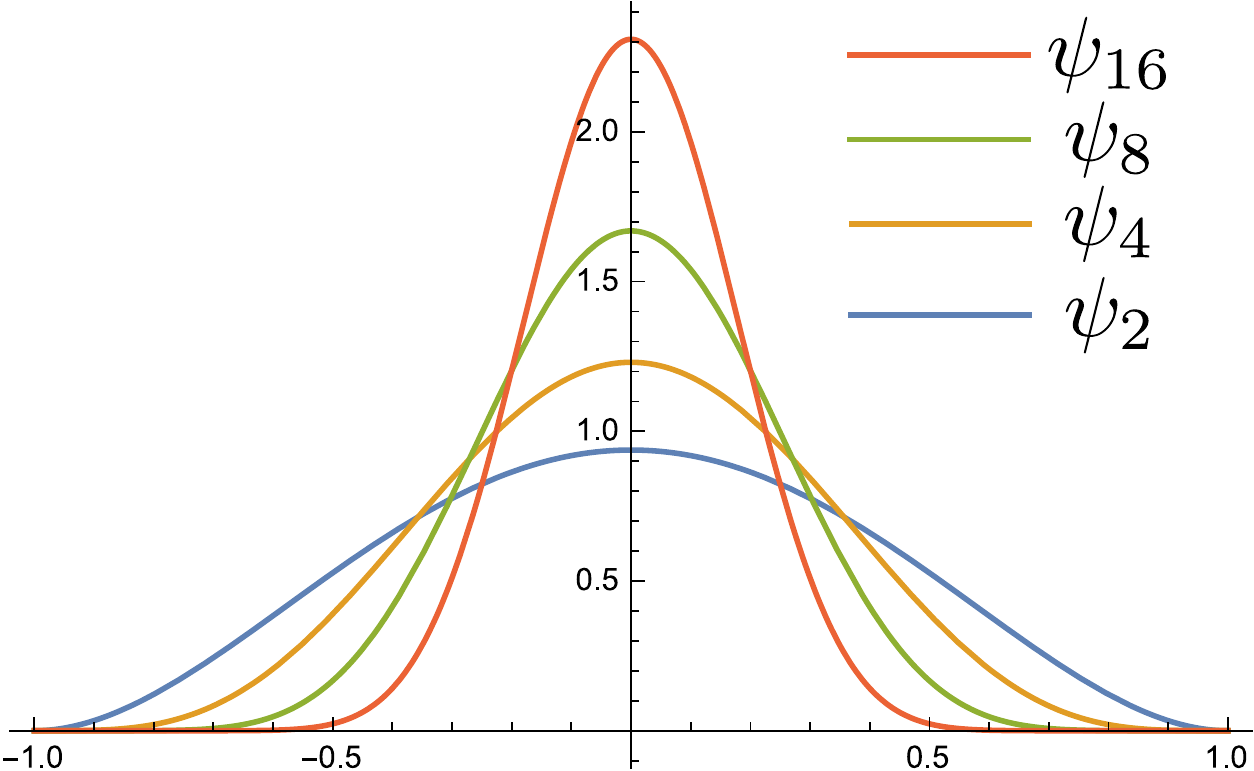}
    \caption{Plots of $\psi_k$ for $k=2$, $4$, $8$, $16$.}
  \end{subfigure}
  \quad
  \begin{subfigure}[t]{.45\linewidth}
    \centering
    \includegraphics[width=.95\linewidth]{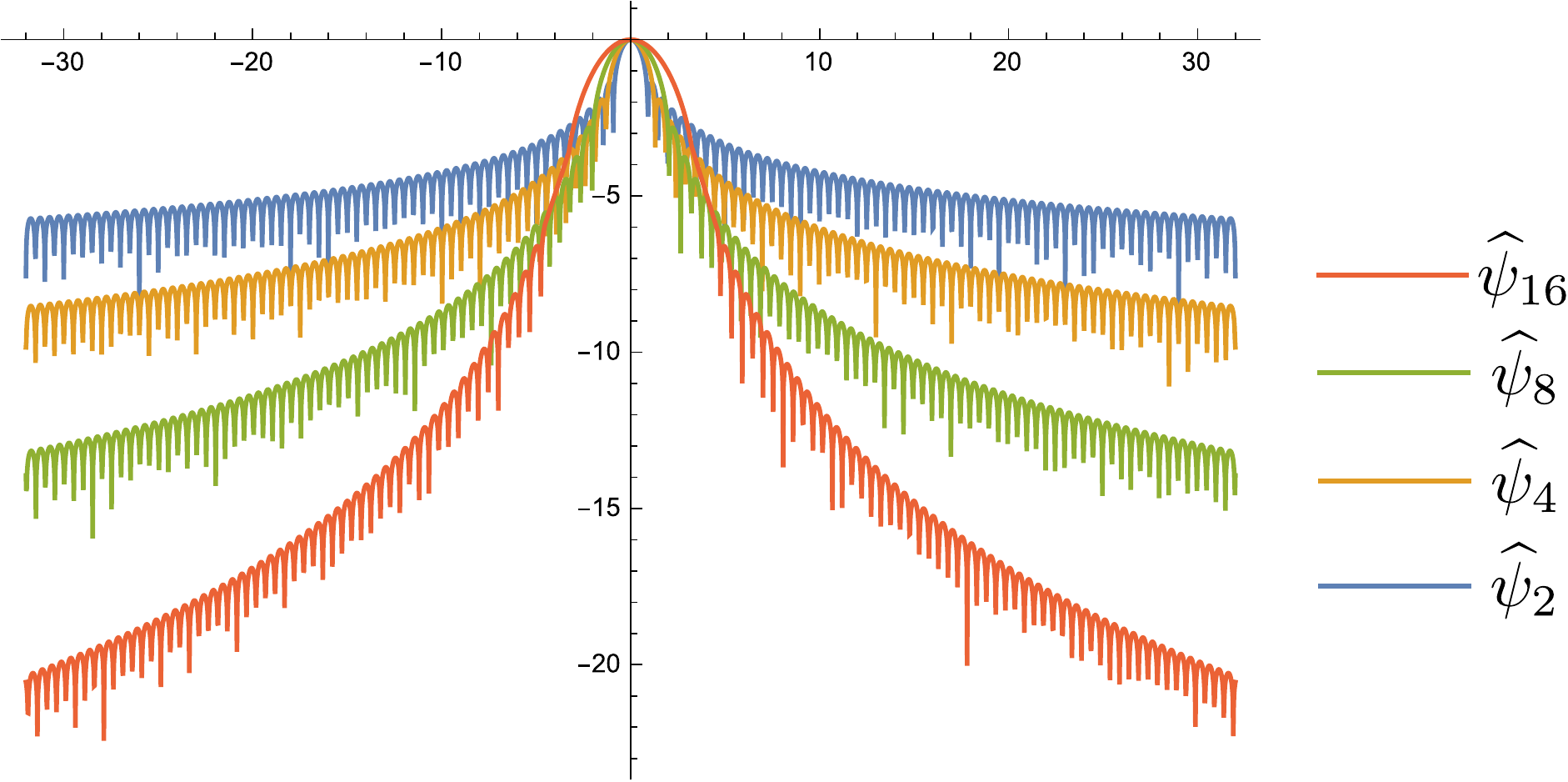}
    \caption{$\log_{10}$ of the absolute value of the Fourier
      transform of the kernels $\psi_k$.}
  \end{subfigure}
  \caption{Examples of the convolution kernels $\psi_k$ and their
  log-power spectra.}\label{Fig01}
\end{figure}

When choosing a kernel with which to perform this convolutional
smoothing, it is important to choose one which is localized in both
physical space \emph{and} Fourier space. Post-convolution, the
resulting smooth curve will then have a band-limit proportional to the
product of the band-limits of the straight edges and the kernel. The
lower the resulting band-limit, the more accurately the curve can be
discretized with a fixed number of degrees of freedom (discretization
points).  The Fourier transform of the function $\psi_k$ is given
analytically as
\begin{equation}
\begin{split}
\mathcal F\left[ \psi_k \right](\xi) &= \hpsi_k(\xi) \\
&=\Gamma\left(k+\frac 32\right)\left(\frac{1}{\pi\xi}\right)^{k+\frac 12}
J_{k+\frac 12}(2\pi\xi),
\end{split}
\end{equation}
where $J_n$ is the Bessel function of the first kind of order $n$
and we have chosen the convention
\begin{equation}
\mathcal F\left[ f \right](\xi) = 
\int_{-\infty}^\infty f(x) \, e^{-2\pi i \xi x} \, dx.
\end{equation}
It is clear that $|\hpsi_k(\xi)|\leq \hpsi_k(0)=1$, and asymptotically
for large $\xi$ these behave like:
\begin{equation}
  |\hpsi_k(\xi)|\approx
  \frac{e\sqrt{\pi}}{k}\left(\frac{2k}{e|\xi|}\right)^{k+1}.
\end{equation}
This shows that once $|\xi| > 2k/e$, the Fourier transform of $\psi_k$
decays quite
rapidly. The Fourier transform of the scaled function satisfies
\begin{equation}
  \mathcal F \left[ \frac{1}{h} \psi_{k} \left( \frac{x}{h}\right)
    \right](\xi) = \mathcal F \left[ \psi_{k,h} (x) \right](\xi) =
  \hpsi_k(h\xi),
\end{equation}
from which it follows that using frequencies a bit larger than
$\mathcal O (2k/eh)$ should suffice. Graphs of the Fourier transforms
of $\{\psi_4, \, \psi_8, \, \psi_{12}, \, \psi_{16}\}$ are shown in
Figure~\ref{Fig01}. Figure~\ref{fig-regpolys} shows multiple smoothings
of regular polygons convolved with the kernel $\psi^h_{k}$ for
various values of $h$.
Note that the smoothings are
nested inside one another for various values of $h$, with the more
interior smoothings corresponding to larger values of $h$.

\begin{figure}[t]
  \centering
  \begin{subfigure}[t]{.45\linewidth}
    \centering
    \includegraphics[width=.95\linewidth]{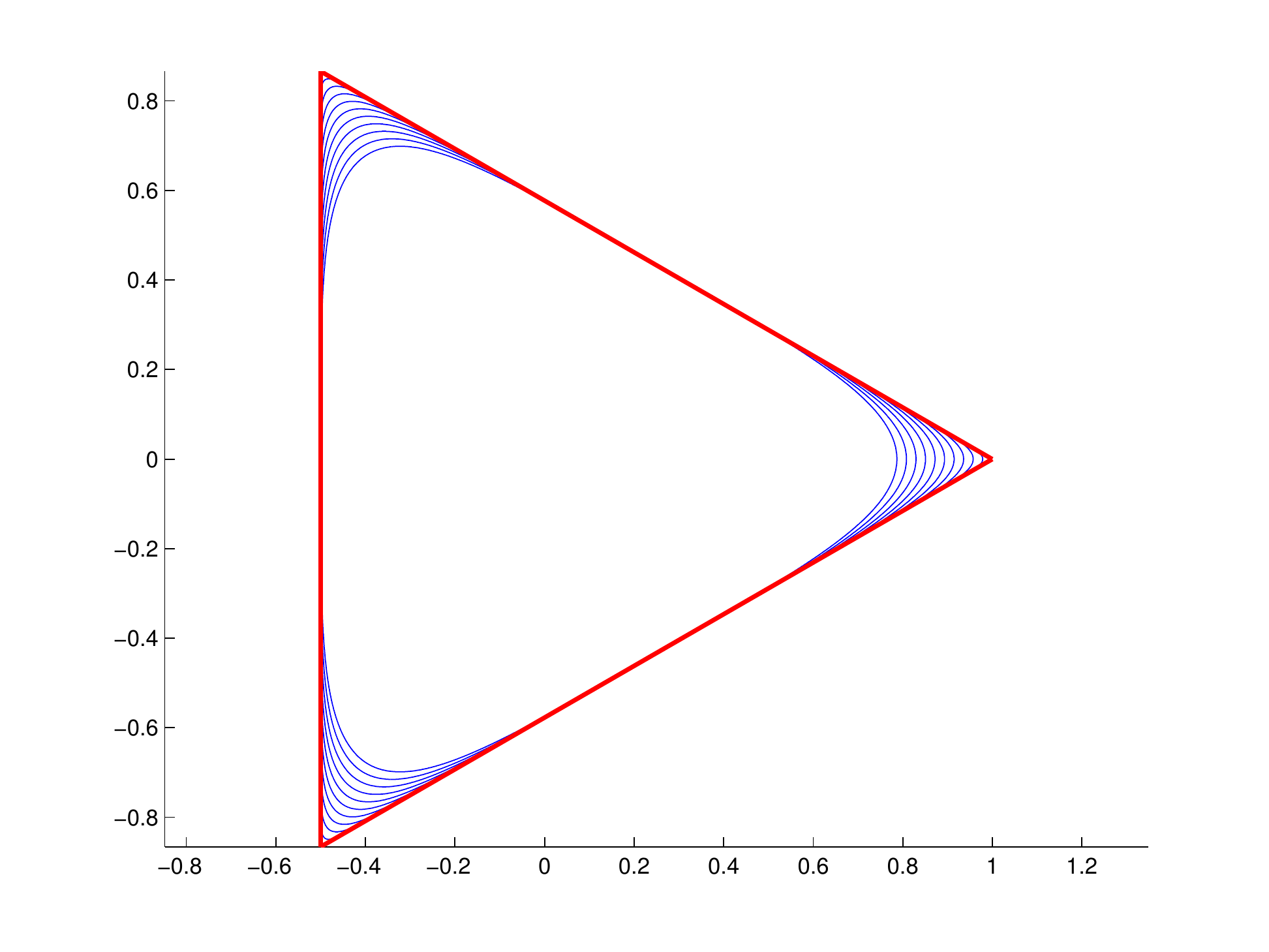}
    \caption{Several smoothings of a triangle.}
  \end{subfigure}
  \quad
  \begin{subfigure}[t]{.45\linewidth}
    \centering
    \includegraphics[width=.95\linewidth]{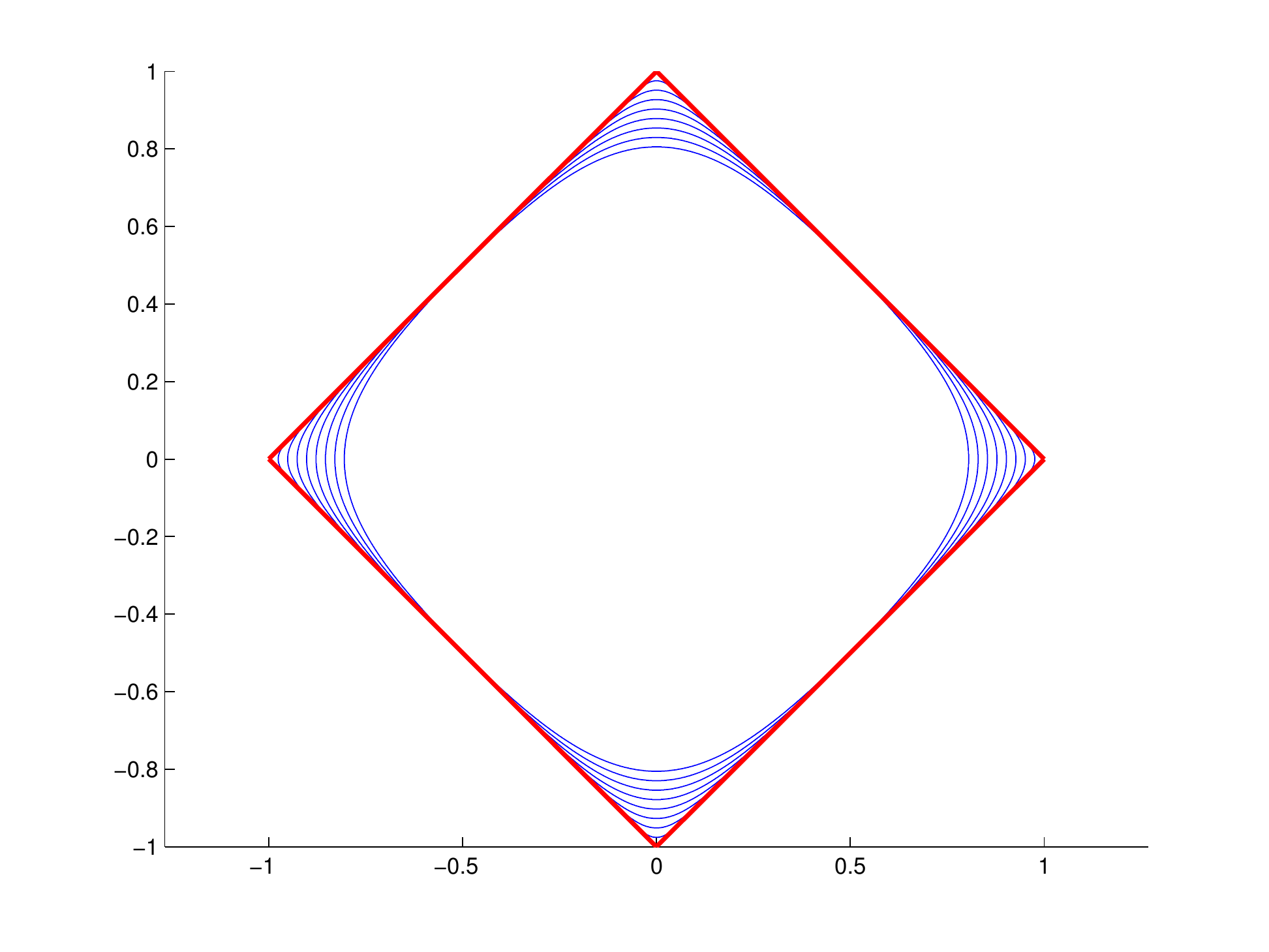}
    \caption{Several smoothings of a square.}
  \end{subfigure} \\
  \begin{subfigure}[t]{.45\linewidth}
    \centering
    \includegraphics[width=.95\linewidth]{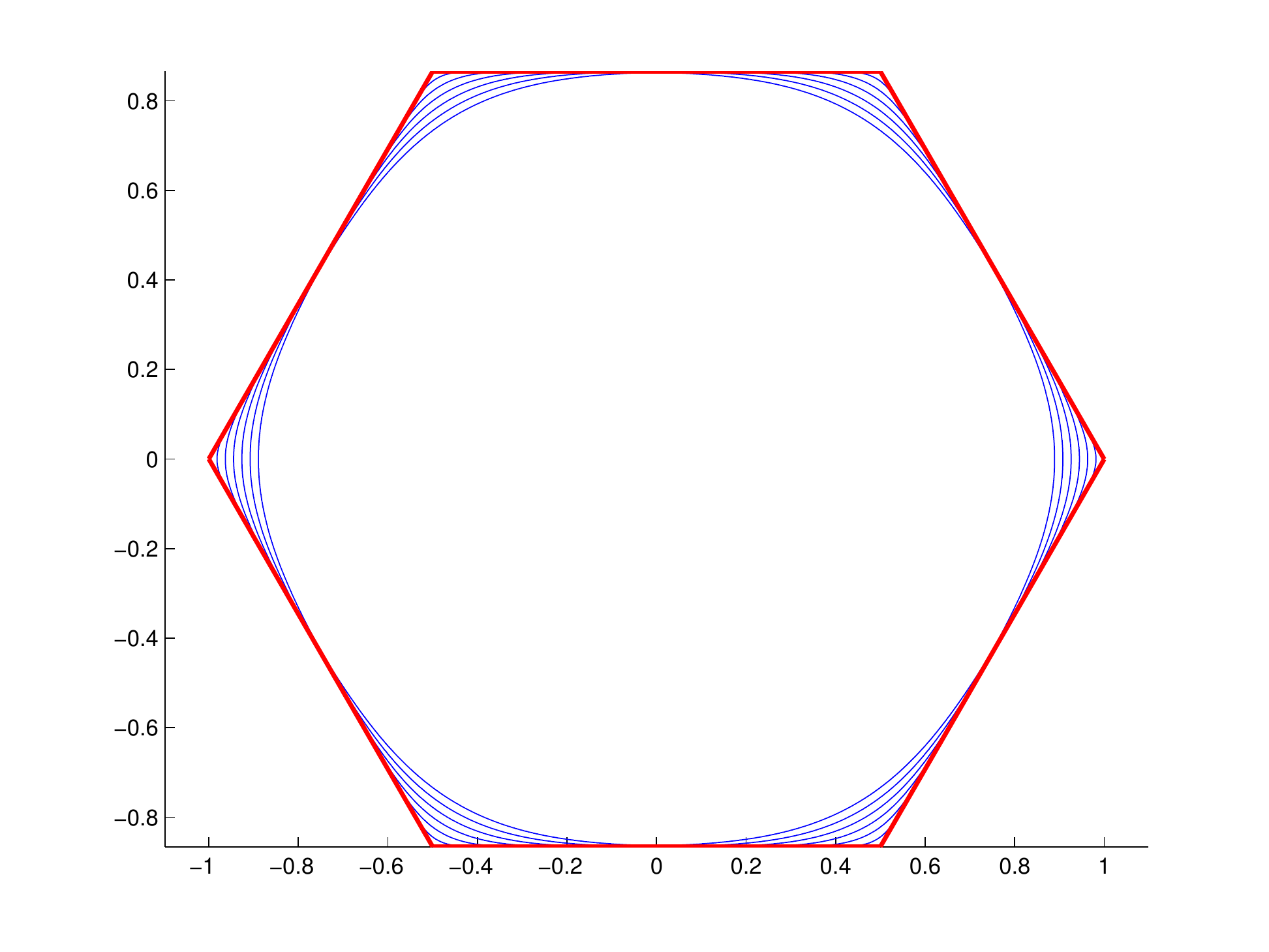}
    \caption{Several smoothings of a hexagon.}
  \end{subfigure}
  \quad
  \begin{subfigure}[t]{.45\linewidth}
    \centering
    \includegraphics[width=.95\linewidth]{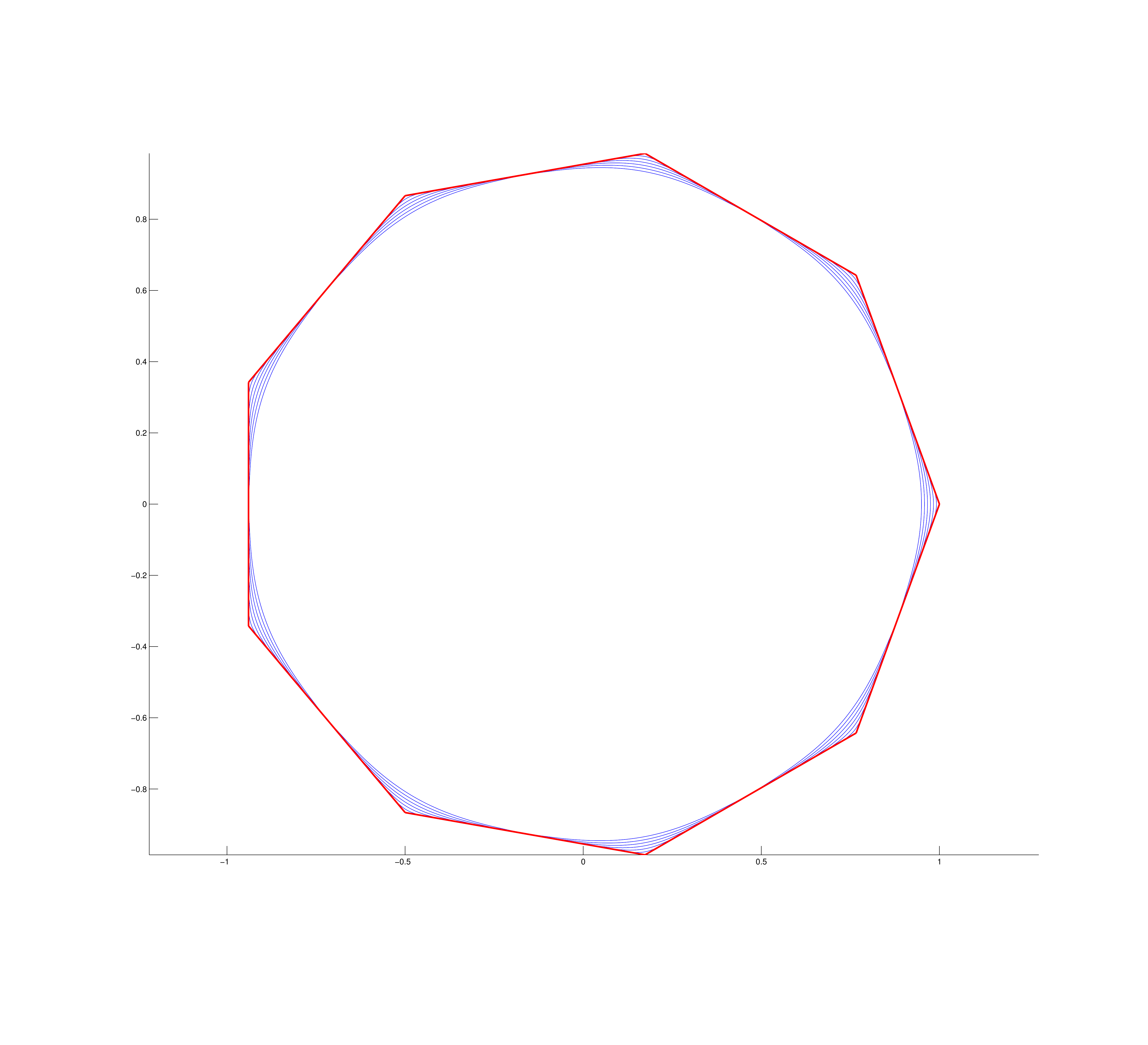}
    \caption{Several smoothings of a nonagon.}
  \end{subfigure}
  \caption{Convolutional smoothings of regular polygons.}\label{fig-regpolys}
\end{figure}

The kernel $\psi_k$ in equation~\eqref{eq_kerncompact} is convenient
to use for our purposes because of its explicit compactness. However,
if we are concerned with the support in the Fourier domain of
$\hpsi_k$ (i.e. the band-limit of $\psi_k$, and therefore the
band-limit of the smoothed geometry), we may wish to choose a kernel
with somewhat more \emph{optimal} uncertainty properties, the
Gaussian:
\begin{equation}
\begin{split}
\phi(x) &= \frac{1}{\sqrt{2\pi}} e^{-x^2/2}, \\
\widehat\phi(\xi) &= e^{-2 \pi^2 x^2}.
\end{split}
\end{equation}

The kernel $\phi$ is \emph{not} analytically compactly supported,
however, it is \emph{numerically} compactly supported. By this we mean
that for any $\epsilon >0$ we can find a threshold $x_\epsilon > 0$
such that for any $|x| > x_\epsilon$, $\phi(x) < \epsilon$. This,
coupled with the integrability of $\phi$, allows us to choose a width
parameter $h$ such that outside of a neighborhood of a vertex, the
resulting smoothed geometry differs pointwise from a straight line
segment by at most  $\epsilon$. Furthermore, if the
neighborhood of a vertex is represented as the graph of a function
$f$, the convolution of $f$ with the Gaussian can be done
analytically. Indeed, a symmetric $f$ will be of the form $f(x) = a|x|
+ b$, for some parameters $a$, $b$, and if we denote a scaled version
of the Gaussian by $\phi_h$, then
\begin{equation}\label{eq_gconv}
\left[ \phi_h * f \right](x) = ax \, \text{erf} \left(
\frac{x}{\sqrt{2}h}\right) + b + \sqrt{\frac{2}{\pi}} \, a h
\, e^{-x^2/2h^2},
\end{equation}
where $\text{erf}$ is the error function. Clearly, for any $\epsilon >
0$, there is a sufficiently large $x_\epsilon$ such that $|\phi_h*f -
f| < \epsilon$ for all $|x| > x_\epsilon$.  In the following numerical
experiments, we set $\epsilon \approx 10^{-15}$ such that smoothing
calculations are done to nearly machine precision. It should be noted
that the choice of $\epsilon$ is {\em independent} of the choice of
$h$. The value of $\epsilon$ determines the size of $|\phi_h(h)|$.

\subsection{Discretization of the smoothing}\label{sec_discretization}

We first discretize a smoothed geometry with a specified value of $h$
(depending on the particular polygon) using polynomial panels
described by $16$ Gauss-Legendre interpolation nodes (samples of
values and derivatives are obtained numerically via adaptive
discretization). Each panel is {\em resolved} when the corresponding
Legendre polynomial coefficients (and those of the arclength function)
of an oversampled discretization are below some threshold, set to
$10^{-10}$ in all cases. Obtaining higher precision is
straightforward, and merely a matter of further refinement. We are
mainly concerned with rough convergence on sub-wavelength rounded
geometries.  See Figure~\ref{fig_disc} for a picture of the
discretization using Gauss-Legendre nodes on each panel, as well as a
diagram of the smoothing kernel and corner.

\begin{figure}[t]
  \centering
  \begin{subfigure}[t]{.3\linewidth}
    \centering
    \includegraphics[width=1\linewidth]{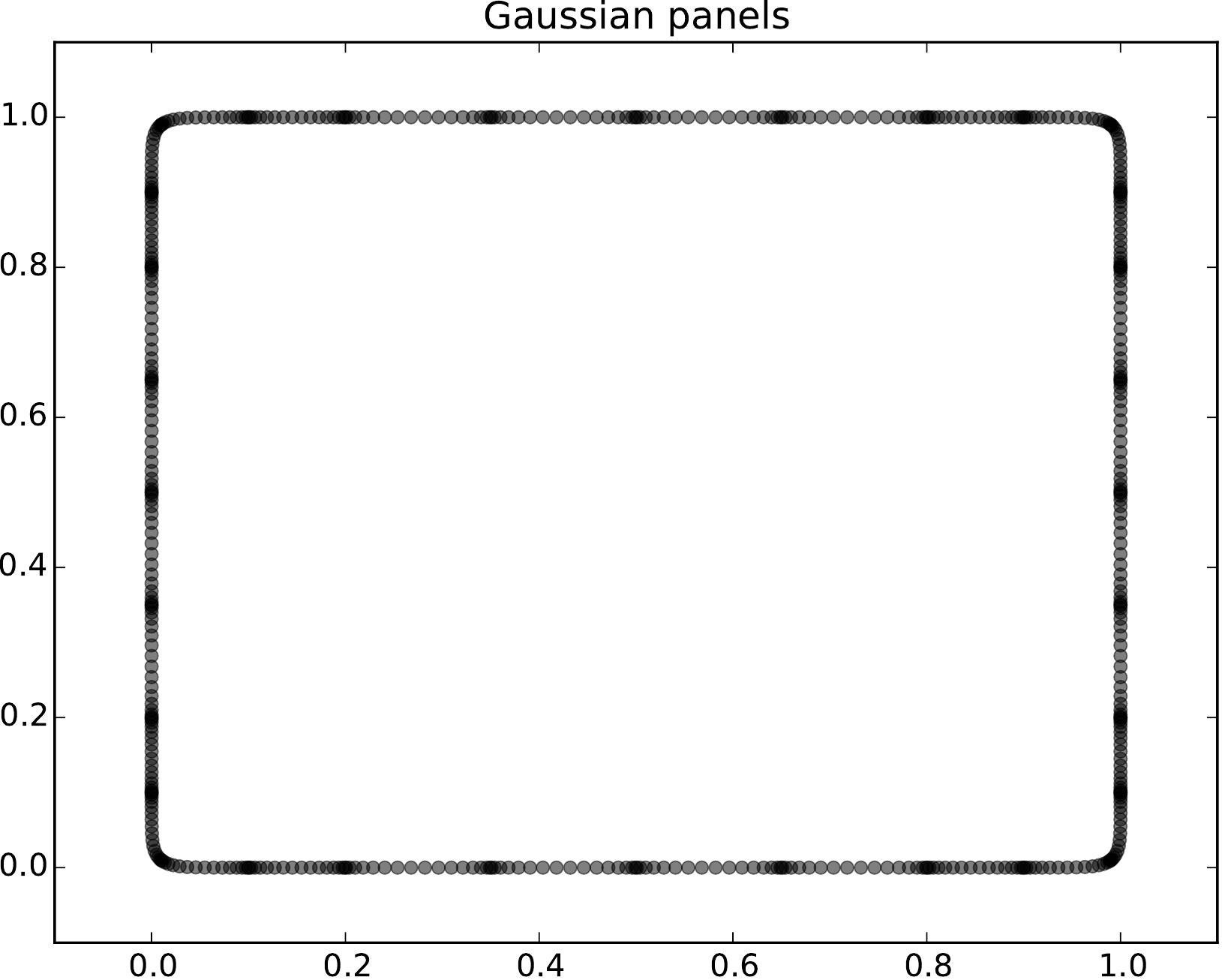}
    \caption{The total geometry.}
  \end{subfigure}
  \quad
  \begin{subfigure}[t]{.3\linewidth}
    \centering
    \includegraphics[width=1\linewidth]{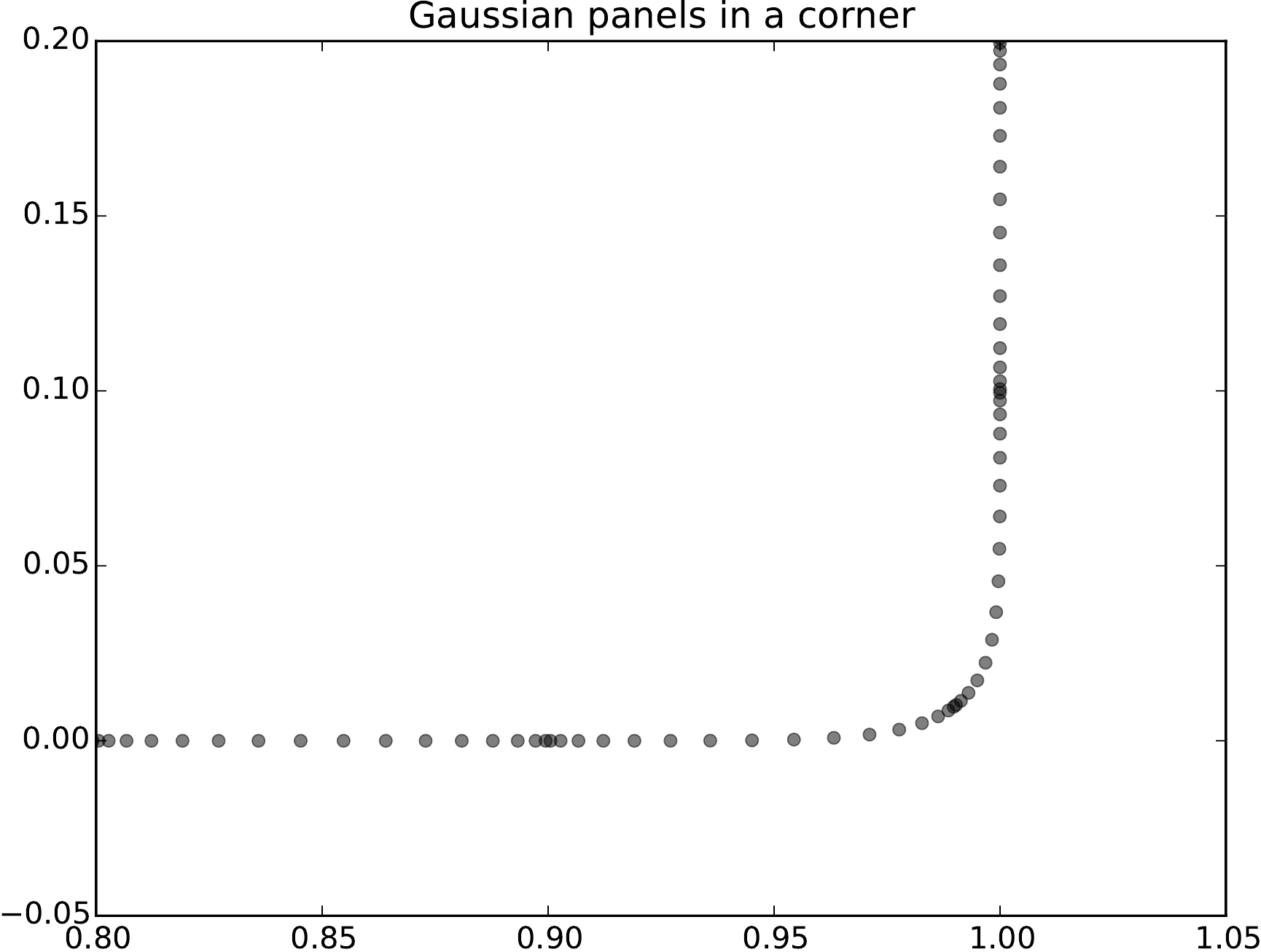}
    \caption{A corner of the geometry.}
  \end{subfigure}
  \quad
  \begin{subfigure}[t]{.3\linewidth}
    \centering
    \includegraphics[width=1\linewidth]{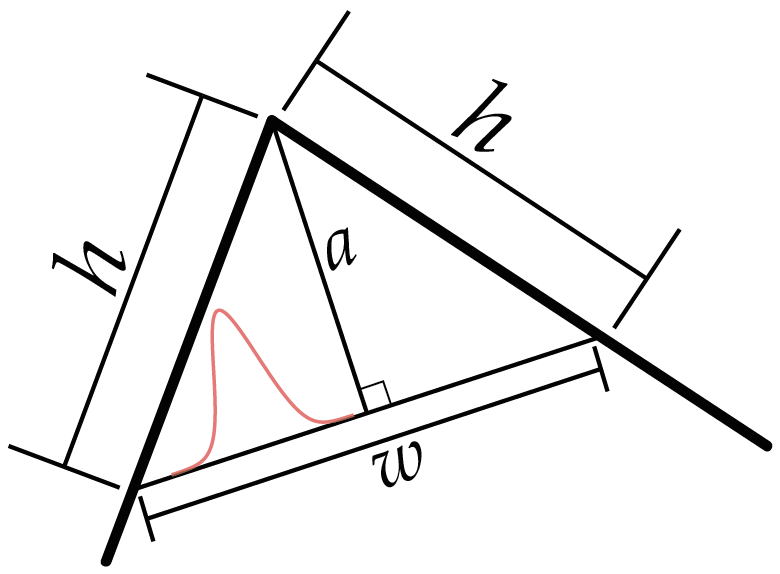}
    \caption{Kernel arrangement.}
  \end{subfigure}
  \caption{Smoothed polygon as sampled using Gauss-Legendre nodes.}
  \label{fig_disc}
\end{figure}

Outside of a distance $h$ from the corner along an edge, the boundary
contains straight edges which can be directly described using linear
polynomial parameterizations.  Inside a distance $h$ from the corner,
we insert (via translation and rotation) an adaptive panel-based
discretization of the rounded function:
\begin{equation}
  f_\delta(x) = \int_{-\infty}^{\infty} \phi_\delta(t) \,
  \left( a - \frac{a}{w/2} \vert t \vert \right) \, dt,
  \qquad \text{for } x \in (-w/2, w/2),
\end{equation}
where for $\epsilon >0$, $\delta = \delta(w)$ 
is chosen such that $f_\delta$
matches the original polygon to precision
$\epsilon$. Figure~\ref{fig_disc} depicts the lengths $a$, $w$, and
$h$.  It is the curve $f_\delta$ that is adaptively discretized so
that its value, first derivative, and arclength functions are
accurate to an absolute precision $\epsilon$ \cite{trefethen_2012}. In
all examples, $\phi_\delta$ is the Gaussian kernel, and the explicit
convolution is given in equation~\eqref{eq_gconv}.  In one final
pre-processing step of the geometry, further refinement takes place
until all neighboring panels differ in arclength by at most a factor
of two and no panel is larger than $2\lambda$, where $\lambda$ is the
wavelength inherent to the problem.  Using the resulting
discretization nodes $\{\bx_i\}$, we discretize the relevant integral
equation (as in the next section) using the $\cL_2$-weighted
Nystr\"om method.
This discretization scheme, used in conjunction with high-order
quadratures for weakly-singular kernels, ensures the convergence of
potentials for both the Dirichlet and Neumann scattering problems in
corner geometries.

We solve the linear system resulting from the Nystr\"om discretization
of the continuous integral equation directly using the LAPACK
implementation of $LU$-factorization. All numerical experiments are
implemented in Fortran 90 and run using the Intel Fortran Compiler
with MKL libraries. Entries in the discretized matrix corresponding to
source-target pairs that reside on the same panel or on neighboring
panels are determined using generalized Gaussian quadratures for
logarithmically singular kernels~\cite{bremer_2010b}. Entries corresponding
to source-target pairs that reside on non-neighboring panels are
obtained from the $16$-point Gaussian quadrature rule corresponding to
unit weight (the Legendre polynomial case).

\begin{figure}[t]
  \centering
  \begin{subfigure}[t]{.3\linewidth}
    \centering
    \includegraphics[width=1\linewidth]{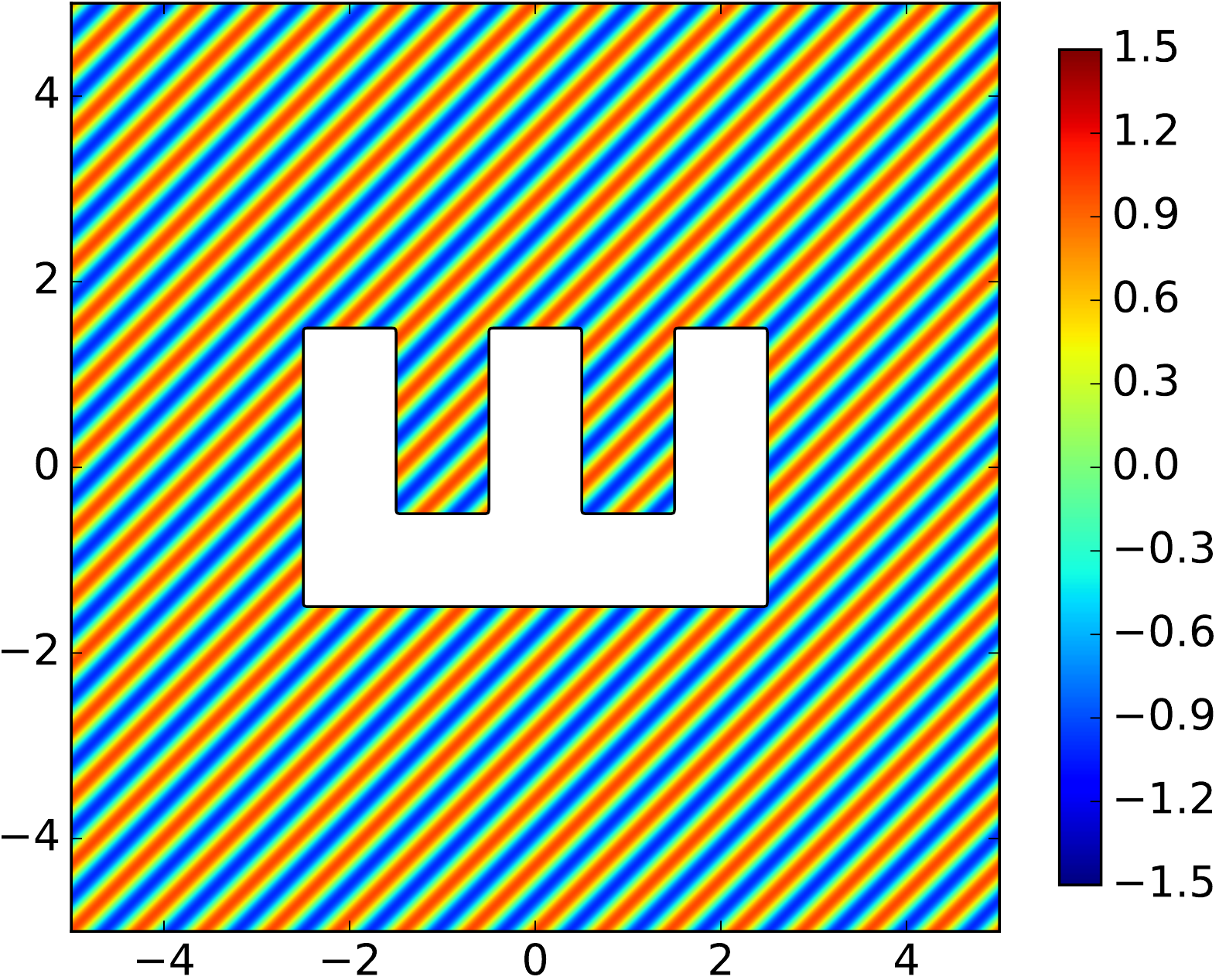}
    \caption{The incoming field.}
  \end{subfigure}
  \quad
  \begin{subfigure}[t]{.3\linewidth}
    \centering
    \includegraphics[width=1\linewidth]{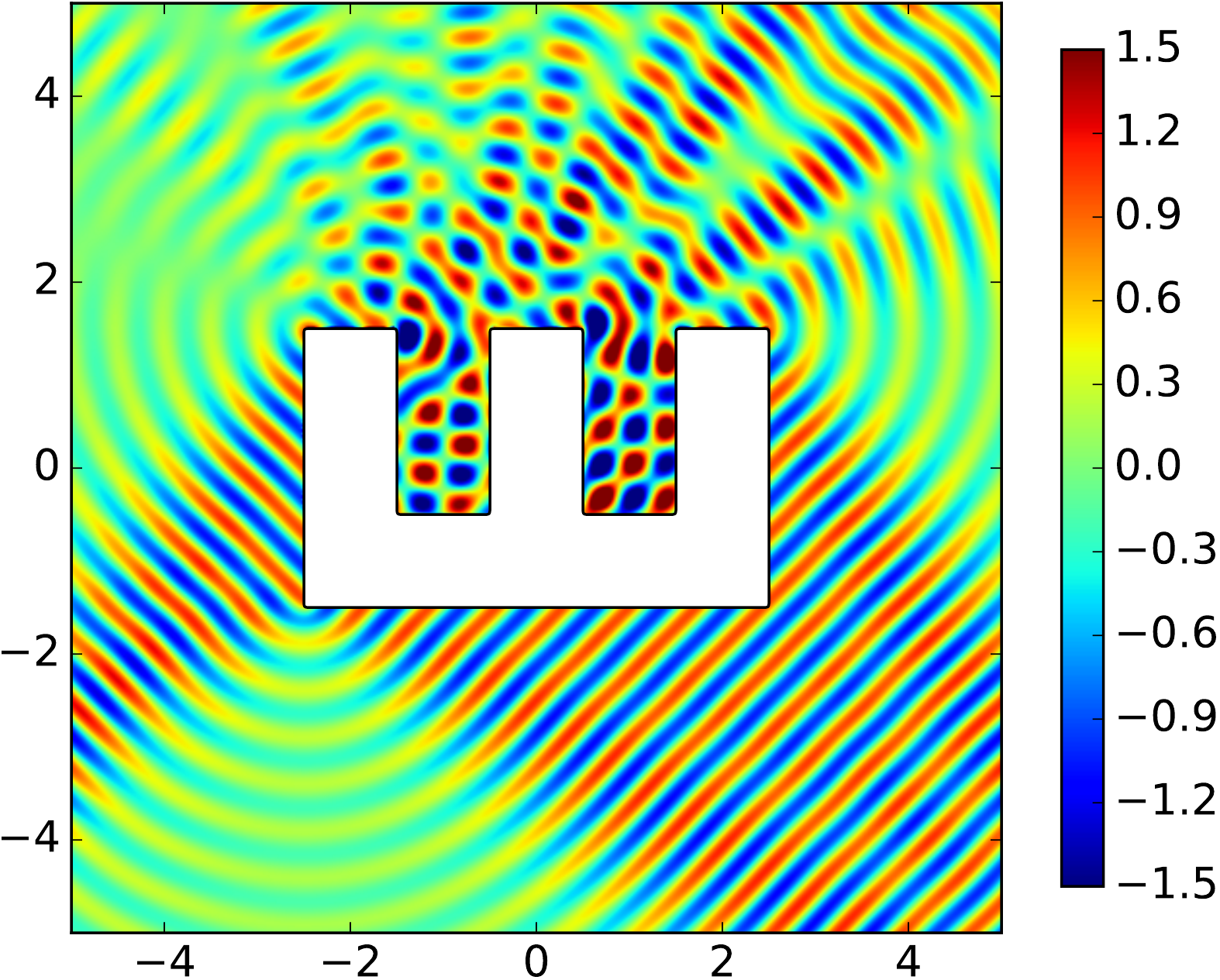}
    \caption{The scattered field.}
  \end{subfigure}
  \quad
  \begin{subfigure}[t]{.3\linewidth}
    \centering
    \includegraphics[width=1\linewidth]{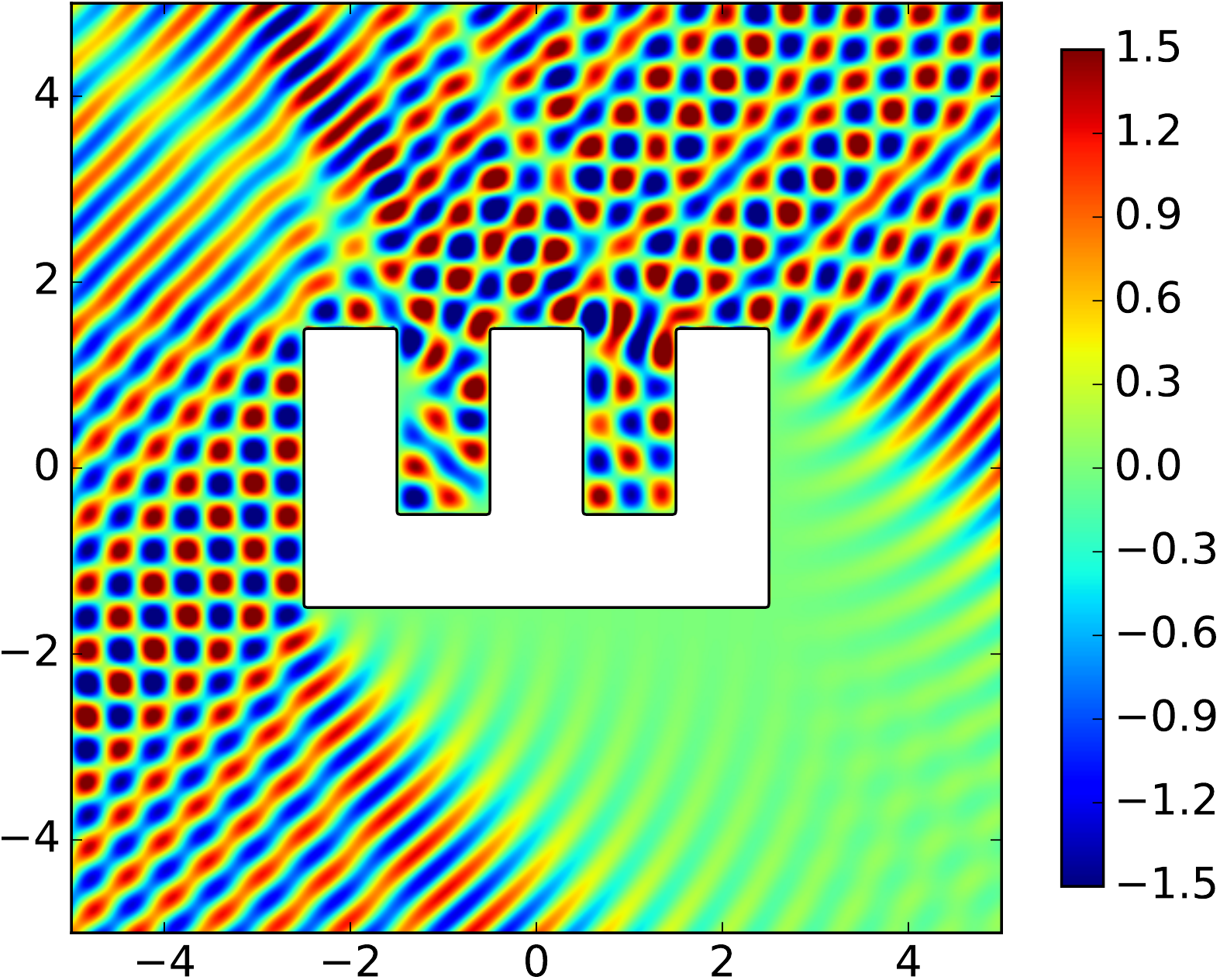}
    \caption{The total field.}
  \end{subfigure}
  \caption{Example exterior sound-soft (Dirichlet) 
    scattering problem. The real
    part of all fields is shown. The angle of incident plane wave is
    $\phi = -\pi/4$.}
  \label{fig_scattering}
\end{figure}

\subsection{Scattering from smoothed polygons: Sound-soft}
\label{sec_scattering}

We now turn our attention to numerical experiments pertaining to the
scattering of acoustic waves from smoothed polygons.  In this section,
we study exterior Helmholtz scattering problems for Dirichlet
boundary conditions; in the following section, we address the
analogous Neumann problem.  In the case of Dirichlet boundary
conditions (corresponding to the case of a sound-soft scatterer), we
have the following boundary value problem:
\begin{equation}
\begin{split}
( \Delta + k^2 ) u^{tot} &= 0  \qquad \text{in } \bbR^2\setminus\Omega,  \\
u^{tot} &= 0  \qquad \text{on } \Gamma = \partial\Omega,
\end{split}
\end{equation}
along with suitable radiation conditions at infinity.
Representing the scattered solution $u$ using a combined-field
potential \cite{epstein_2015},
\begin{equation}
  u = \left( \mathcal S_k + i \left( k\alpha +
  \beta\right)\mathcal D_k \right) \sigma,
\end{equation}
we have the following second-kind integral equation along $\Gamma$
for the density $\sigma$: 
\begin{equation}
  \frac{\sigma}{2}  + \left( \mathcal S_k + i \left( k\alpha +
  \beta\right)\mathcal D_k \right)
  \sigma = -u^{inc} \qquad \text{on } \Gamma,
\end{equation}
where $\mathcal S_k$ and $\mathcal D_k$ are interpreted in their
on-surface limiting sense.  We have set $\alpha = 1.2$ and $\beta =
0.8$ in all examples.  The scattered field is then calculated at all
exterior volume locations using standard Gaussian quadrature for
polynomials and the fast multipole method for the two-dimensional
Helmholtz equation~\cite{gg_software}. More accurate near-surface
evaluation could be obtained using the methods of \cite{helsing_2014b}
or \cite{klockner_2013, rachh_2016}.

The following simulations are obtained from driving the scattering
problem by setting $u^{inc}$ to be a two-dimensional plane-wave,
traveling in the direction of the angle $\phi$:
\begin{equation}
u_\phi^{inc}(\bx) = e^{ik(x\cos\phi + y\sin\phi)}.
\end{equation}
It is easy to see that $u^{inc}$ satisfies the free-space Helmholtz
equation, but {\em not} the Sommerfeld radiation condition.  See
Figure~\ref{fig_scattering} for depiction of an incoming plane wave
$u_{-\pi/4}^{inc}$, scattered field $u$, and total field $u^{tot}$
with Dirichlet boundary conditions. In this example, $k = 12.43 + i
10^{-5}$, corresponding to a wavelength of
$\lambda = 2\pi/\Re{k} \approx 0.505 $.
The accuracy of the integral equation solver is tested by calculating
the error in the potential when compared to a {\em known} solution
obtained from placing a fundamental source in the 
interior of the object. I.e., we solve a test problem:
\begin{equation}\label{eq_testsolution}
\begin{aligned}
( \Delta + k^2 ) u &= 0  &\qquad &\text{in } \bbR^2\setminus\Omega,  \\
u &= g_k(\cdot, \bx_0)  & &\text{on } \Gamma,
\end{aligned}
\end{equation}
where $\bx_0$ is placed near the center of the object. The potential
$u$ is then compared with the exact solution $g_k(\cdot,\bx_0)$ at
test points placed on a circle some distance away from the scatterer.

We study the effect of the corner rounding by examining what is
referred to as the {\em sonar cross section} (SCS) of the object $\Omega$.
Usually, this function is given in terms of the far-field behavior of
the scattered field based on large-$\bx$ asymptotics of $H^{(1)}_0$:
\begin{equation}
  u^{far}(\bx) = \sqrt{\frac{1}{8\pi k}}
    \frac{e^{ik\vert \bx \vert}}{\sqrt{\vert \bx\vert}} e^{i\pi/4} \int_\Gamma
  e^{-ik \hat\br \cdot \by} \, \sigma(\by) \, ds(\by),
\end{equation}
where $\hat\br = \bx/|\bx|$. The far-field signature
is often
used in inverse obstacle scattering problems where measurement noise
is frequently the dominant component anyway~\cite{colton_kress}.

However, in our case, we have direct access to the scattered field at
any observation point.  We can thereby evaluate near-field functions
at varying radii from the scatterer:
\begin{equation}
  \begin{aligned}
u^{near}_d(\theta) &= \int_\Gamma g_k(\bd, \by) \, \sigma(\by) \,
ds(\by), \\
\bd &= \bc + d \cos\theta \, \bi + d \sin\theta \, \bj,
  \end{aligned}
\end{equation}
where we denote the scattered field at a distance $d$ from the
centroid $\bc$ of $\Omega$. The vectors $\bi$, $\bj$ are the unit
vectors in the $x$, $y$ directions, respectively. There are two types
of cross sections that are usually computed: mono-static and
bi-static. Mono-static cross sections characterize the scatterer in
terms of the intensity of the backscatter in the {\em same} direction
as the incoming wave. In particular, we calculate $u^{near}_d$ at a
{\em single} value of $\theta$ corresponding to the opposite angle of
propagation of the incoming plane wave $u^{inc}_{-\theta}$. If the
mono-static cross section is sampled at $m$ angles, this requires
solving $m$ separate scattering problems.

On the other hand, the bi-static cross section contains intensities of
the scattered field for a {\em fixed} angle of incident plane wave.
 Figure~\ref{fig_cross} shows sample mono-static and bi-static
cross sections for the scattering problem depicted in
Figure~\ref{fig_scattering}, each captured at a distance of $d = 10
\approx 20\lambda$ from the origin. The angle of incidence for the
bi-static case was $\theta = -\pi/4$. In each case, the cross section
is plotted on a polar grid in decibels:
\begin{equation}
  \mathcal C(\theta) = 10 \log_{10}\left( \vert u(\theta) \vert \right).
\end{equation}

\begin{figure}[t]
  \centering
  \begin{subfigure}[t]{.4\linewidth}
    \centering
    \includegraphics[width=1\linewidth]{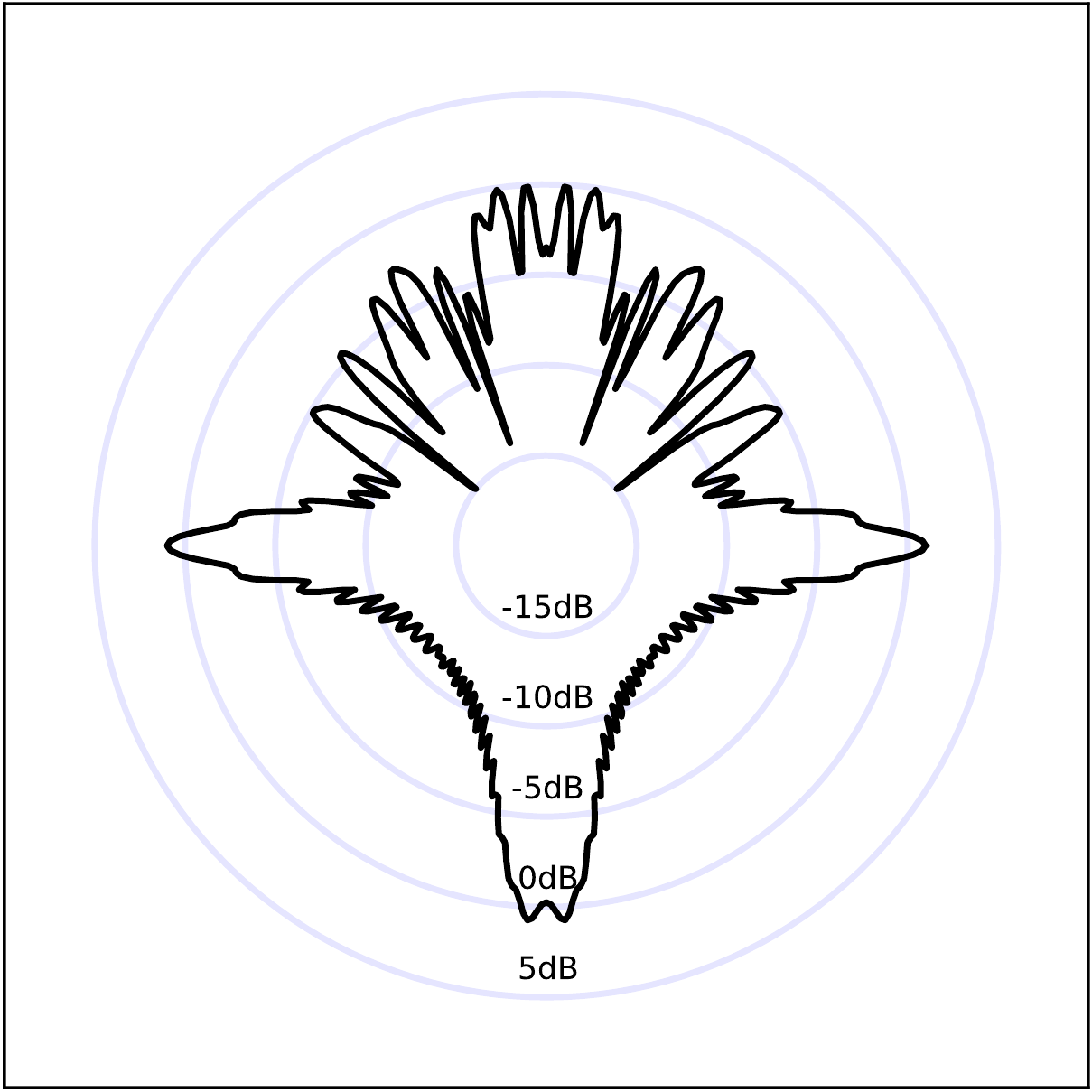}
    \caption{The mono-static cross section.}
  \end{subfigure}
  \qquad
  \begin{subfigure}[t]{.4\linewidth}
    \centering
    \includegraphics[width=1\linewidth]{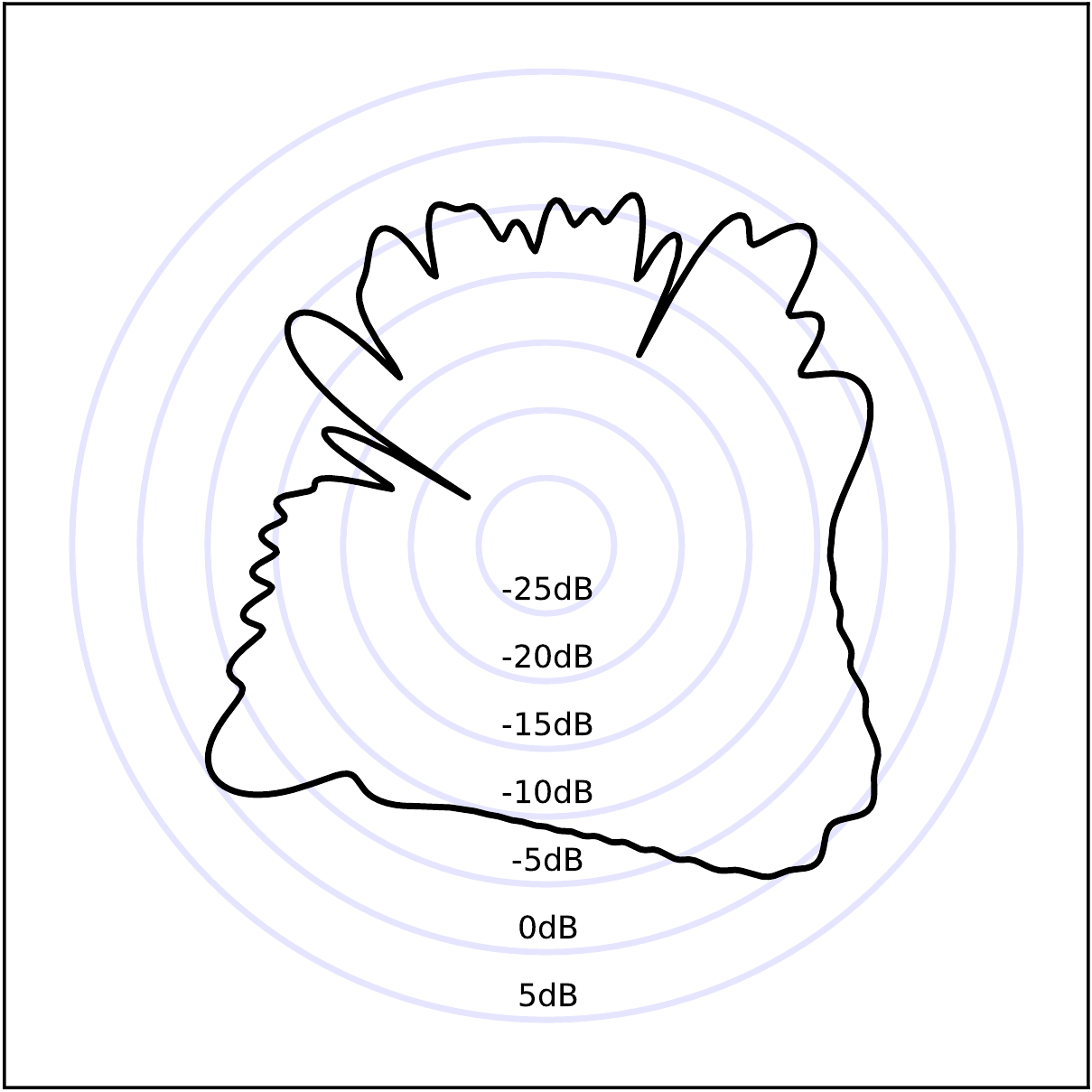}
    \caption{The bi-static cross section.}
  \end{subfigure}
  \caption{Example mono-static and bi-static cross sections for the
    Dirichlet problem 
    corresponding to the geometry in Figure~\ref{fig_scattering},
    captured at approximately $20\lambda$ from the origin.
    The absolute value of the scattered field is plotted on a
    $\log_{10}$ scale.}
  \label{fig_cross}
\end{figure}

As the size of the region that is rounded near the corners is
decreased, to below sub-wavelength, we see a convergence of the cross
sections.  Figure~\ref{fig_manyscs} shows a plot of several bi-static
and mono-static cross sections for the same object (that in
Figure~\ref{fig_scattering}). Here, we have increased the wave number
to $k = 54.32 + i 10^{-5}$ to allow for a larger dynamic range of
rounding widths. This value of $k$
corresponds to a wavelength of $\lambda \approx 0.12$.
The cross section is evaluated on a disc of radius $15
\approx 125\lambda$ centered at the origin. 

\begin{figure}[h!]
  \centering
  \begin{subfigure}[t]{.4\linewidth}
    \centering
    \includegraphics[width=1\linewidth]{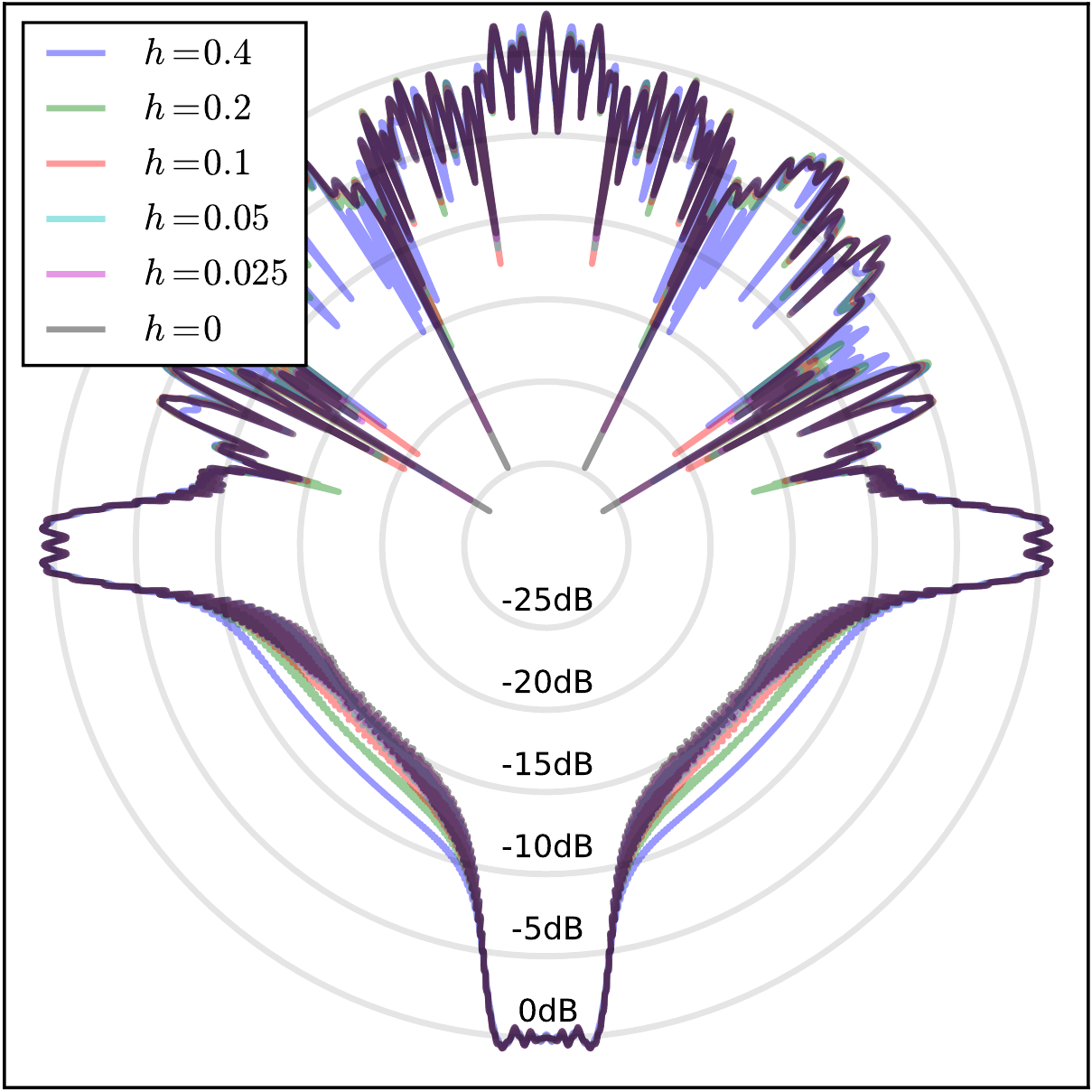}
    \caption{Several mono-static cross sections.}
  \end{subfigure}
  \qquad
  \begin{subfigure}[t]{.4\linewidth}
    \centering
    \includegraphics[width=1\linewidth]{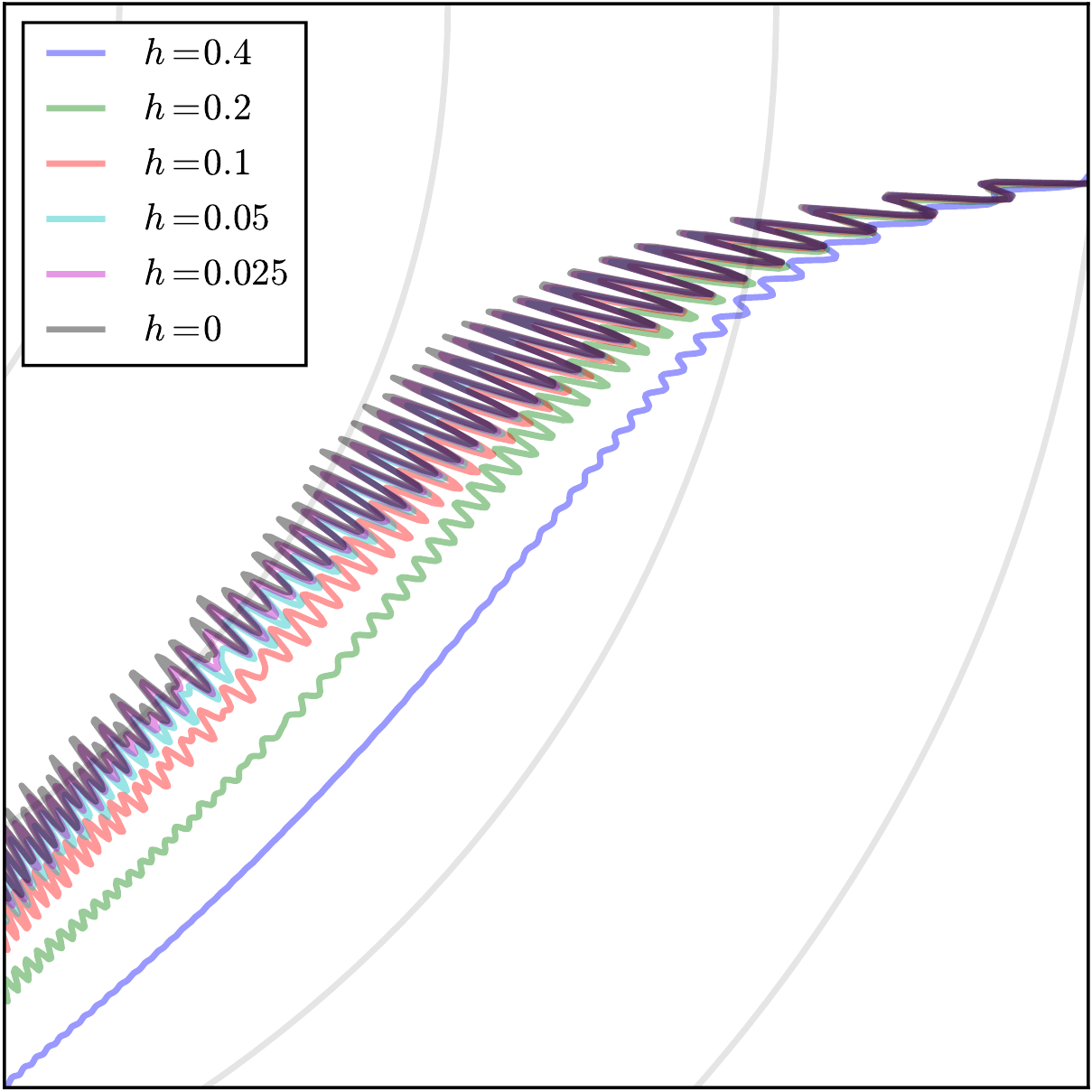}
    \caption{Details of several mono-static cross sections.}
  \end{subfigure}\\
  \begin{subfigure}[t]{.4\linewidth}
    \centering
    \includegraphics[width=1\linewidth]{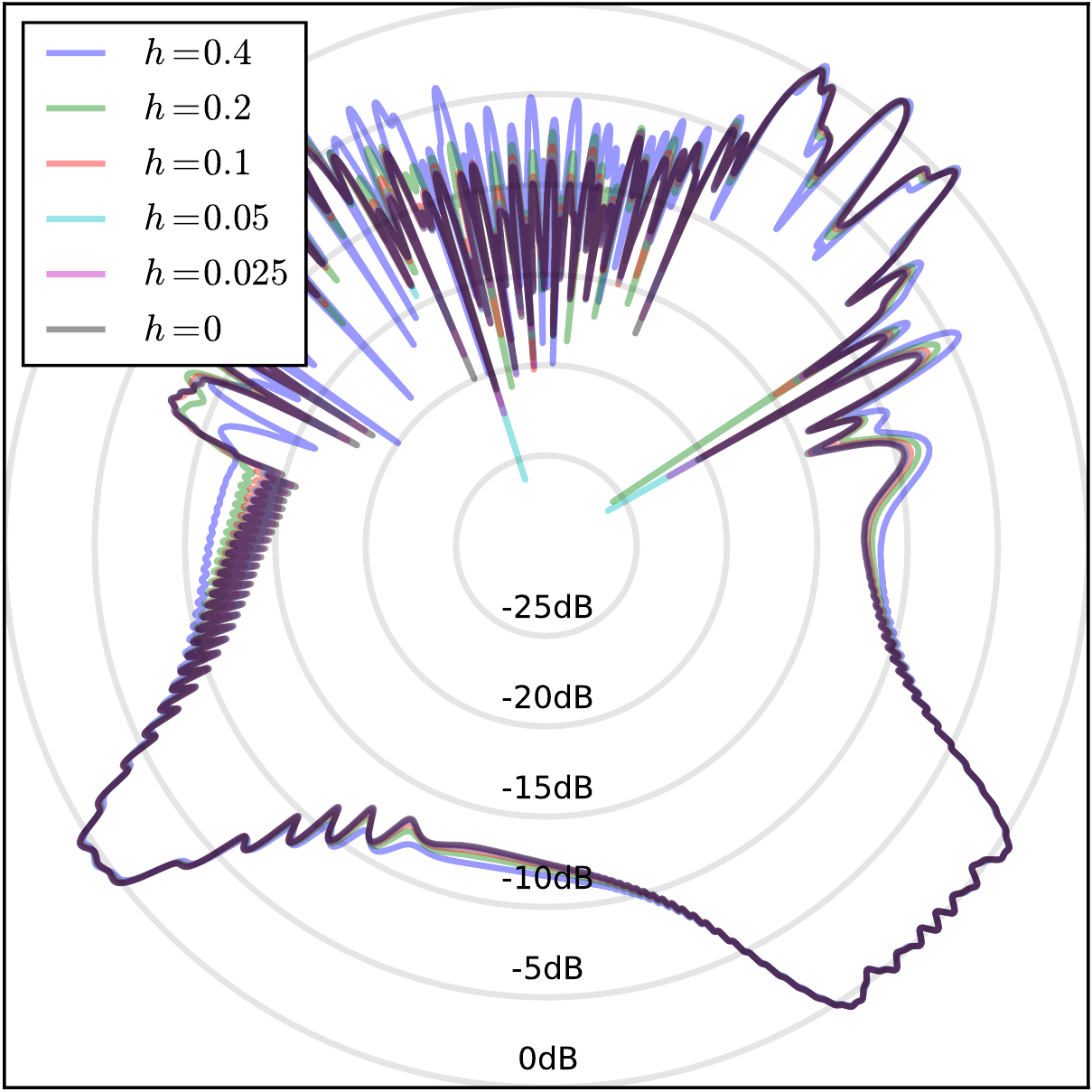}
    \caption{Several bi-static cross sections.}
  \end{subfigure}
  \qquad
  \begin{subfigure}[t]{.4\linewidth}
    \centering
    \includegraphics[width=1\linewidth]{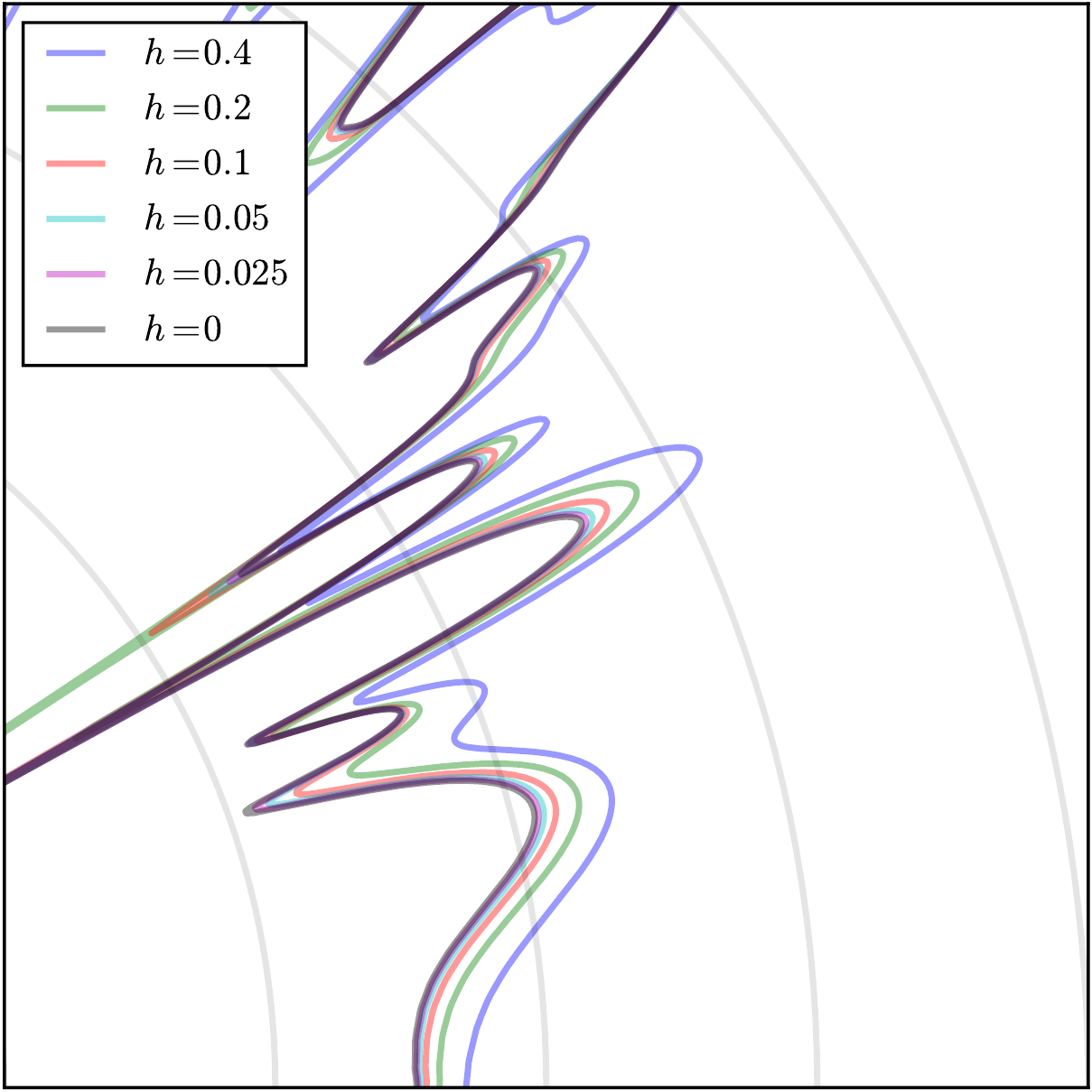}
    \caption{Details of several bi-static cross sections.}
  \end{subfigure}
  \caption{Example mono-static and bi-static cross sections for the
    Dirichlet problem 
    corresponding to several roundings of the geometry in
    Figure~\ref{fig_scattering}, captured on a disc of radius $15
    \approx 125\lambda$ from the origin.  The absolute value of the
    scattered field is plotted on a decibel $=10\log_{10}$ scale.}
  \label{fig_manyscs}
\end{figure}

The obvious question to ask is how close these solutions are to the
solution in the case of scattering from an exact polygon with corners.
Results of this experiment are shown in Figures~\ref{fig_regress},
\ref{fig_diffs}, \ref{fig_diffs_big}, and~\ref{fig_triangle}.
Convergence results of the far-field and
moderately near-field bi-static cross sections are reported in
Tables~\ref{tab_comb_54}, \ref{tab_comb_6}, \ref{tab_triangle52}
and~\ref{tab_triangle7}. Near-field convergence is given in
Table~\ref{tab_comb_near}.
In each case, the order of convergence of the
scattered field is commensurate with the scale of the rounding.

The errors in the value of the potential converge at a rate of roughly
first-order with respect to the rounding parameter. Slightly faster
convergence is actually observed, which may be due to the high
accuracy of the rounding and the smoothing effects of the layer
potential representation. We are currently investigating this
phenomena.  It is worth pointing out that in Figure~\ref{fig_diffs}
there are no correct digits in the solution (in a relative sense)
until the rounding is performed on a scale roughly equal to the
wavelength of the solution. The exact solution ($h = 0.0$) was
calculated by dyadic refinement of the edges of the polygon near the
corners to a scale of $10^{-10}$. The resulting integral equation was
solved using an $\mathcal L_2$ weighting scheme, as described
in~\cite{bremer_2012}.

\afterpage{
\begin{figure}[t!]
  \centering
  \begin{subfigure}[t]{.45\linewidth}
    \centering
    \includegraphics[width=1\linewidth]{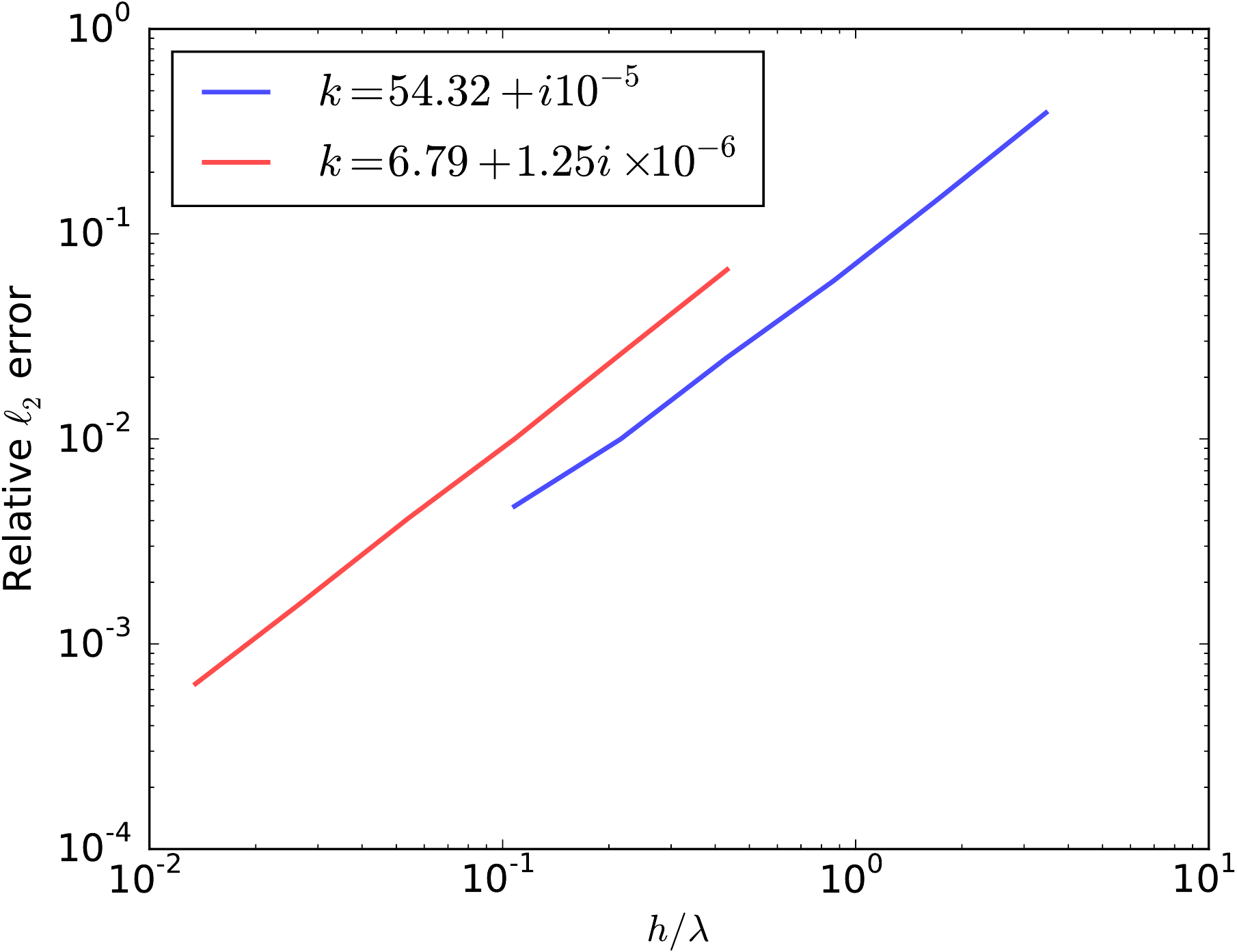}
    \caption{The empirical convergence for the comb shape with $k =
      54.32 + i \, 10^{-5}$ is $\mathcal O(h^{1.28})$ and that for $k
      = 6.79 + 1.25i \times 10^{-6}$ is $\mathcal O(h^{1.34})$.}
  \end{subfigure}
  \hfill
  \begin{subfigure}[t]{.45\linewidth}
    \centering
    \includegraphics[width=1\linewidth]{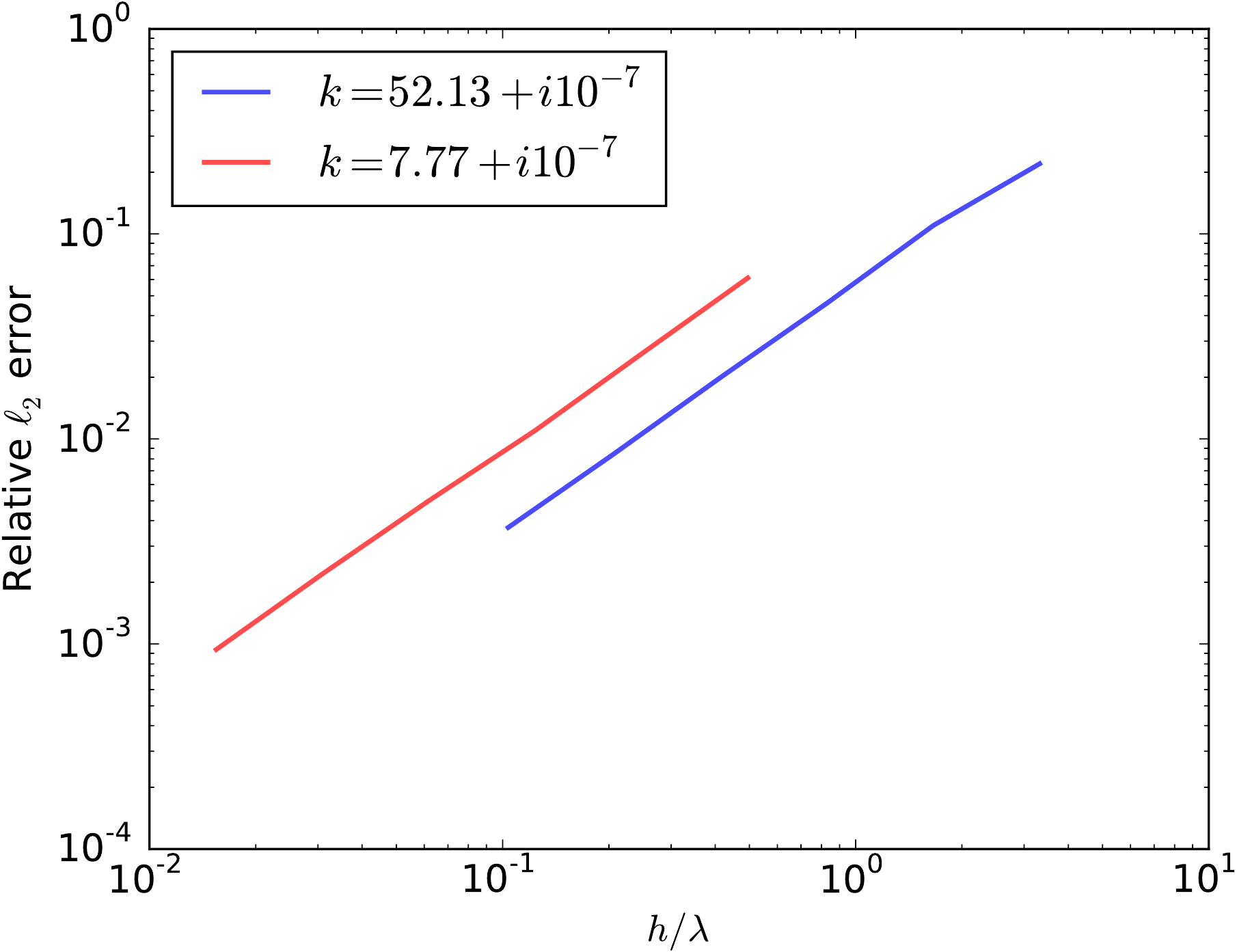}
    \caption{The empirical convergence for the triangle shape in
      Figure~\ref{fig_triangle} with $k =
      52.13 + i \, 10^{-7}$ is $\mathcal O(h^{1.19})$ and that for $k
      = 7.77 + i\, 10^{-6}$ is $\mathcal O(h^{1.21})$.}
  \end{subfigure}
  \caption{Plot of the relative $\ell_2$ error of the scattered
    potential for the
    Dirichlet problem versus rounding size in terms of wavelength. Both
    regimes exhibit commensurate convergence. }
  \label{fig_regress}
\end{figure}

\vspace{\baselineskip}

\begin{figure}[h!]
  \centering
  \begin{subfigure}[t]{.5\linewidth}
    \centering
    \includegraphics[width=1\linewidth]{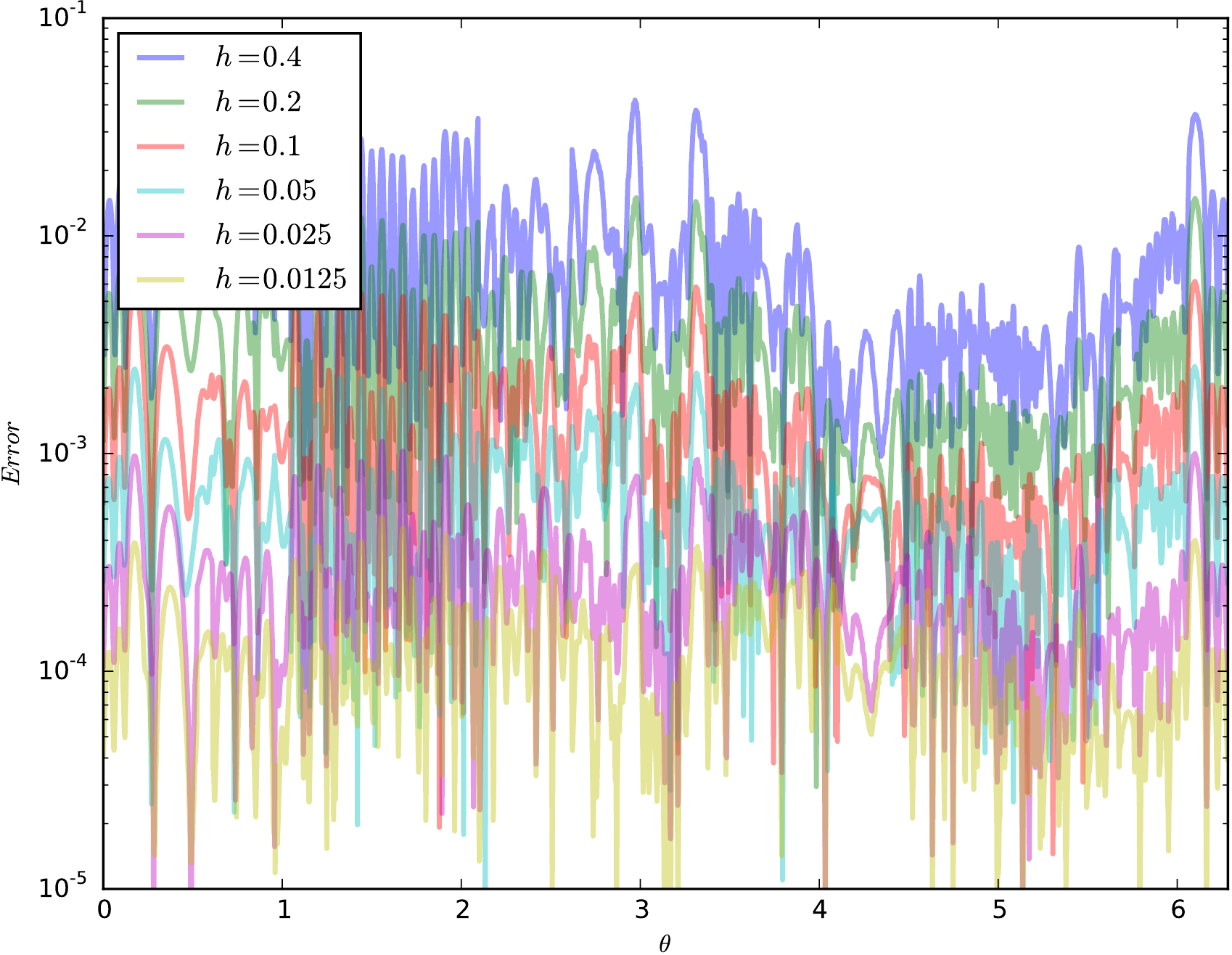}
    \caption{Errors in the real part of bi-static cross sections.}
  \end{subfigure}
  \hfill
  \begin{subtable}[b]{.45\linewidth}
    \centering
    \resizebox{\columnwidth}{!}{%
    \begin{tabular}{llll}\hline
      $h$ & $n$ & $RMSE$ error & Rel. $\ell_2$ error \\ \hline
      $0.4$    & $4608$ & $1.2 \times 10^{-2}$ & $3.9 \times 10^{-1}$ \\ 
      $0.2$    & $5312$ & $4.7 \times 10^{-3}$ & $1.5 \times 10^{-1}$ \\ 
      $0.1$    & $5824$ & $1.9 \times 10^{-3}$ & $5.9 \times 10^{-2}$ \\ 
      $0.05$   & $6752$ & $8.0 \times 10^{-4}$ & $2.5 \times 10^{-2}$ \\ 
      $0.025$  & $7296$ & $3.3 \times 10^{-4}$ & $1.0 \times 10^{-2}$ \\
      $0.0125$ & $7680$ & $1.5 \times 10^{-4}$ & $4.7 \times 10^{-3}$ \\
      $0.0$   & $14592$ & & \\
      \hline
    \end{tabular}  }
    \caption{Errors in the bi-static cross sections.}
    \label{tab_comb_54}
  \end{subtable}
  \caption{Errors in the complex valued bi-static cross section for
    the Dirichlet problem at a distance of $1000 \approx 8333\lambda$
    from the origin (as compared to the true corner scattering
    problem). The error converges approximately to first order in the
    rounding parameter $h$. In each case, the PDE was solved to
    roughly a precision of $10^{-9}$ in the $\mathcal L_\infty$ norm
    (as determined by testing against a known solution).}
  \label{fig_diffs}
\end{figure}
\clearpage
}

\afterpage{
\begin{figure}[h]
  \centering
  \begin{subfigure}[t]{.5\linewidth}
    \centering
    \includegraphics[width=1\linewidth]{plot73-crop.pdf}
    \caption{Total field for $k = 6.79 + 1.25i \times 10^{-6}$. The
      wavelength is approximately $\lambda \approx .93$.}
  \end{subfigure}
  \hfill
  \begin{subtable}[b]{.45\linewidth}
    \centering
    \resizebox{\columnwidth}{!}{%
    \begin{tabular}{llll}\hline
      $h$ & $n$ & $RMSE$ error & Rel. $\ell_2$ error \\ \hline
      $0.4$ & $3872$ & $2.8 \times 10^{-3}$ & $6.7 \times 10^{-2}$ \\ 
      $0.2$ & $4576$ & $1.1 \times 10^{-3}$ & $2.6 \times 10^{-2}$ \\ 
      $0.1$ & $5088$ & $4.4 \times 10^{-4}$ & $1.0 \times 10^{-2}$ \\ 
      $0.05$ & $5632$ & $1.7 \times 10^{-4}$ & $4.1 \times 10^{-3}$ \\ 
      $0.025$ & $6016$ & $6.8 \times 10^{-5}$ & $1.6 \times 10^{-3}$ \\
      $0.0125$ & $6400$ & $2.7 \times 10^{-5}$ & $6.4 \times 10^{-4}$ \\
      $0.0$   & $13312$ & & \\
      \hline
    \end{tabular}  }
    \caption{Errors in the bi-static cross sections for $k = 6.79 +
      1.25i \times 10^{-6}$ at a radius of $1000 \approx 1081\lambda$.}
    \label{tab_comb_6}
  \end{subtable} \\
  \vspace{\baselineskip}
  \begin{subfigure}[t]{.5\linewidth}
    \centering
    \includegraphics[width=1\linewidth]{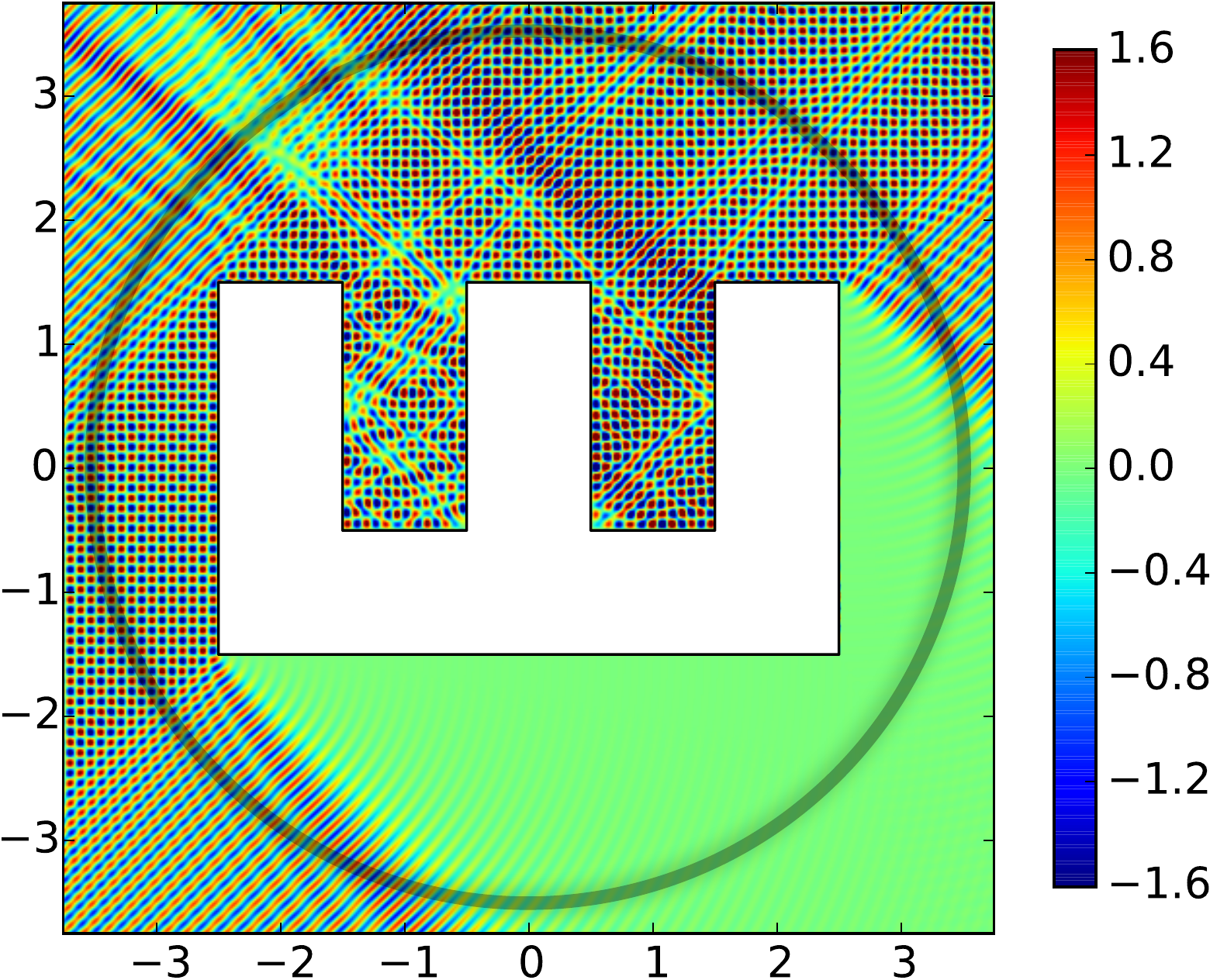}
    \caption{Total field for $k = 54.32 + i \times 10^{-5}$ along with
      testing curve for convergence of scattered field. The corners
      in this plot were rounded with a parameter $h=0.025$.}
  \end{subfigure}
  \hfill
  \begin{subtable}[b]{.45\linewidth}
    \centering
    \resizebox{\columnwidth}{!}{%
    \begin{tabular}{llll}\hline
      $h$ & $n$ & $RMSE$ error & Rel. $\ell_2$ error \\ \hline
      $0.4$    & $4608$  & $ 2.1 \times 10^{-1}$ & $ 2.7 \times 10^{-1}$ \\ 
      $0.2$    & $5312$  & $ 8.2 \times 10^{-2}$ & $ 1.1 \times 10^{-1}$ \\ 
      $0.1$    & $5824$  & $ 3.2 \times 10^{-2}$ & $ 4.1 \times 10^{-2}$ \\ 
      $0.05$   & $6752$  & $ 1.3 \times 10^{-2}$ & $ 1.6 \times 10^{-2}$ \\ 
      $0.025$  & $7296$  & $ 4.9 \times 10^{-3}$ & $ 6.4 \times 10^{-3}$ \\
      $0.0125$ & $7680$  & $ 2.0 \times 10^{-3}$ & $ 2.5 \times 10^{-3}$ \\
      $0.0$    & $14592$ & & \\
      \hline
    \end{tabular}  }
    \caption{Errors in the bi-static cross section for $k = 54.32 + i
      \times 10^{-5}$ in the near-field at a radius of $3.5 \approx
      30\lambda$.}
    \label{tab_comb_near}
  \end{subtable}  
  \caption{Errors in the complex bi-static cross section for the
    Dirichlet problem (as compared
    to the true corner scattering problem). The error converges
    approximately to first order in the rounding parameter $h$. In
    each case, the PDE was solved to roughly a precision of $10^{-9}$ in
    the $\mathcal L_\infty$ norm (as determined by testing against a
    known solution).}
  \label{fig_diffs_big}
\end{figure}
\clearpage
}

\afterpage{
  \begin{figure}[h]
    \centering
    \begin{subfigure}[t]{.45\linewidth}
      \centering
      \includegraphics[width=1\linewidth]{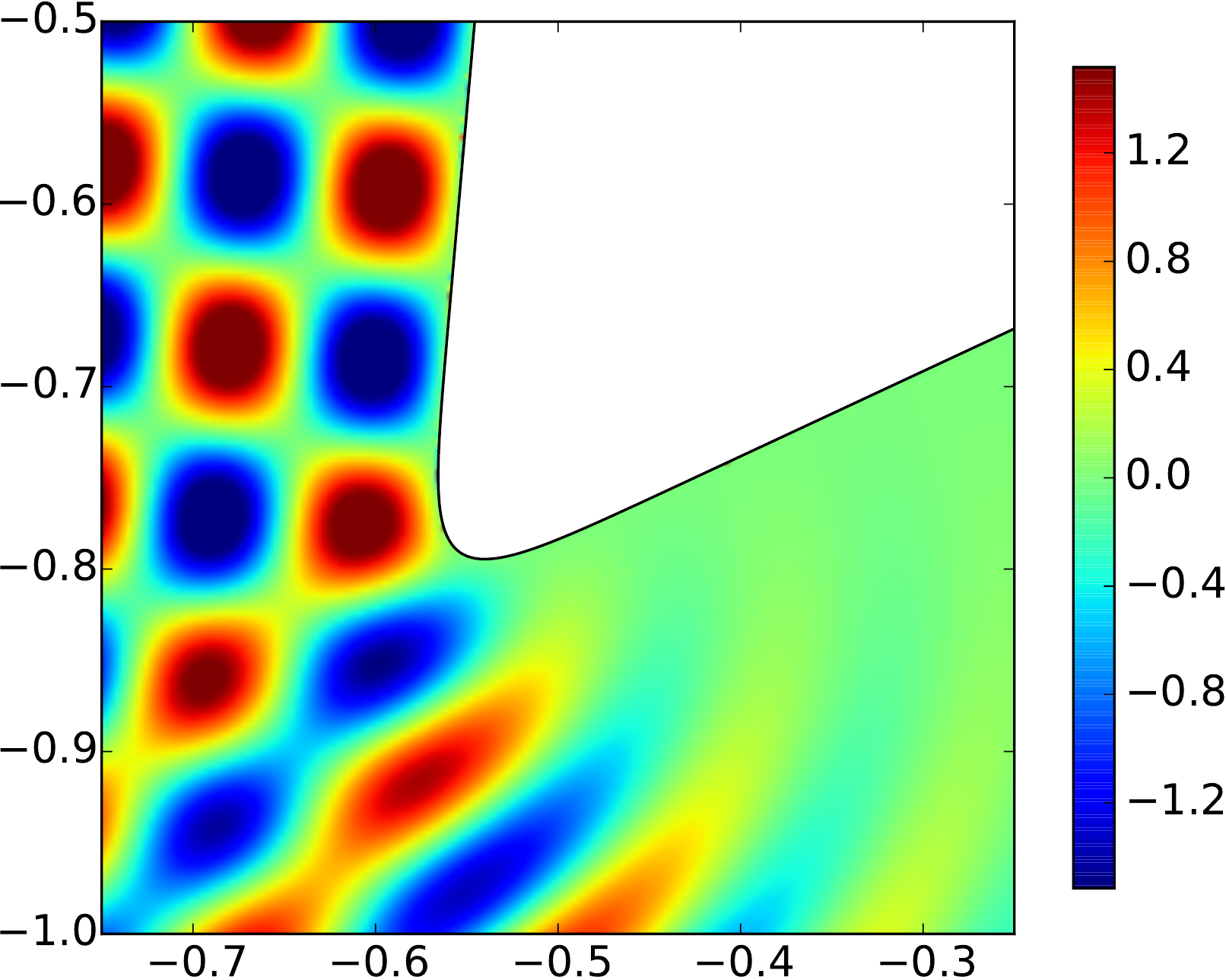}
      \caption{Rounding of $h = 0.4$, $k = 52.13 + i 10^{-7}$.}
    \end{subfigure} \hfill
    \begin{subfigure}[t]{.45\linewidth}
      \centering
      \includegraphics[width=1\linewidth]{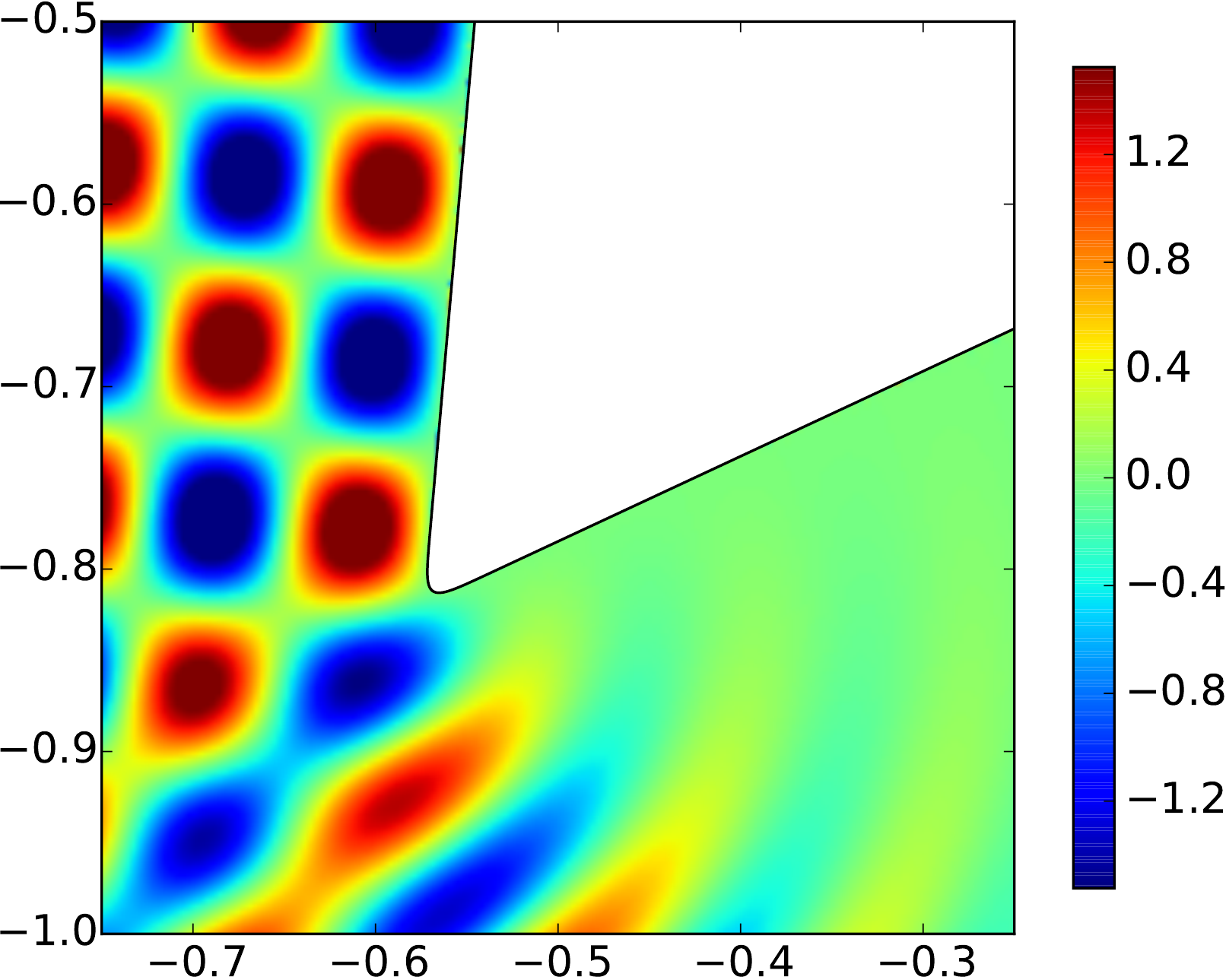}
      \caption{Rounding of $h = 0.1$, $k = 52.13 + i 10^{-7}$.}
    \end{subfigure} \\ \vspace{\baselineskip}
    \begin{subfigure}[t]{.45\linewidth}
      \centering
      \includegraphics[width=1\linewidth]{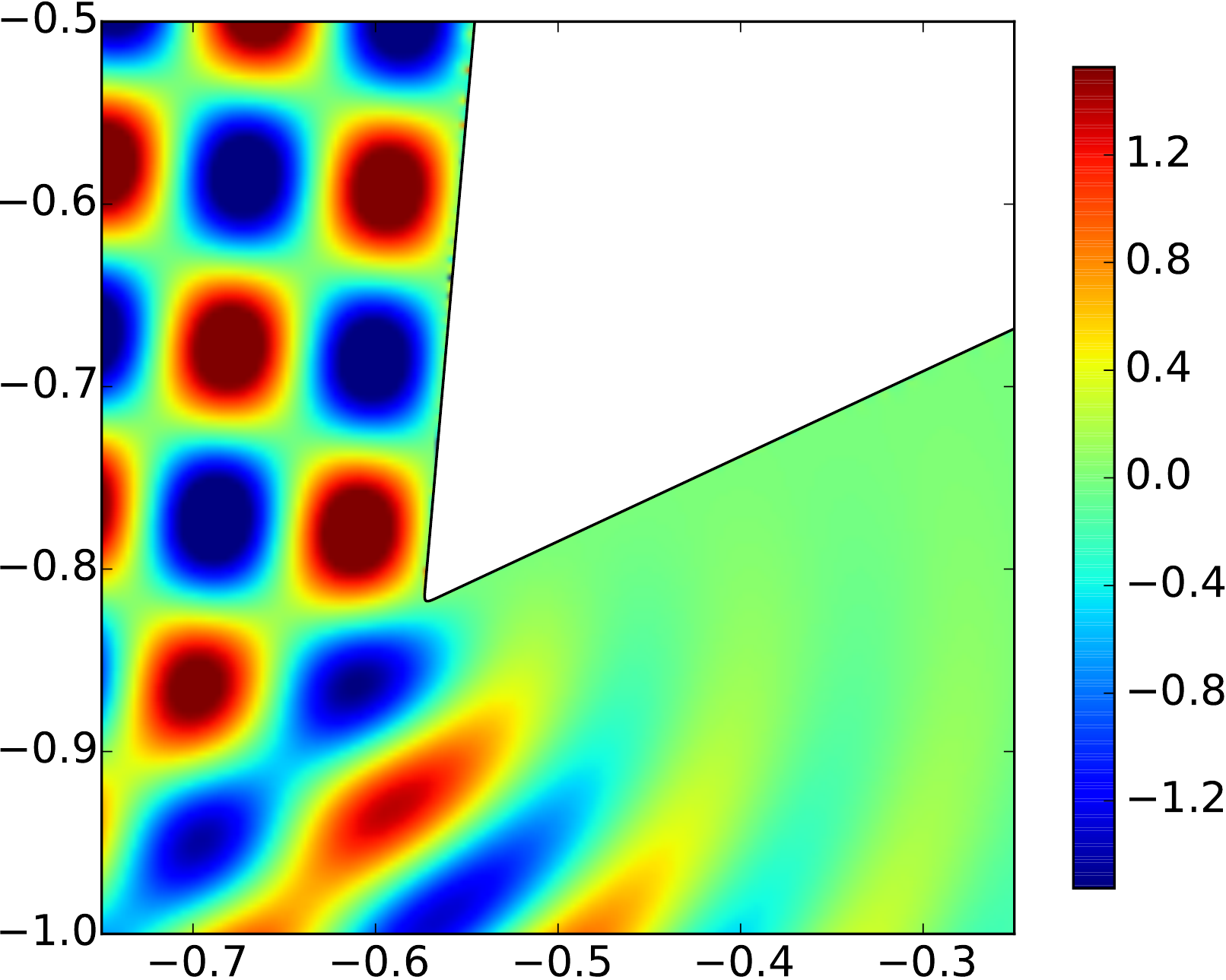}
      \caption{Rounding of $h = 0.025$, $k = 52.13 + i 10^{-7}$.}
    \end{subfigure} \hfill
    \begin{subfigure}[t]{.45\linewidth}
      \centering
      \includegraphics[width=1\linewidth]{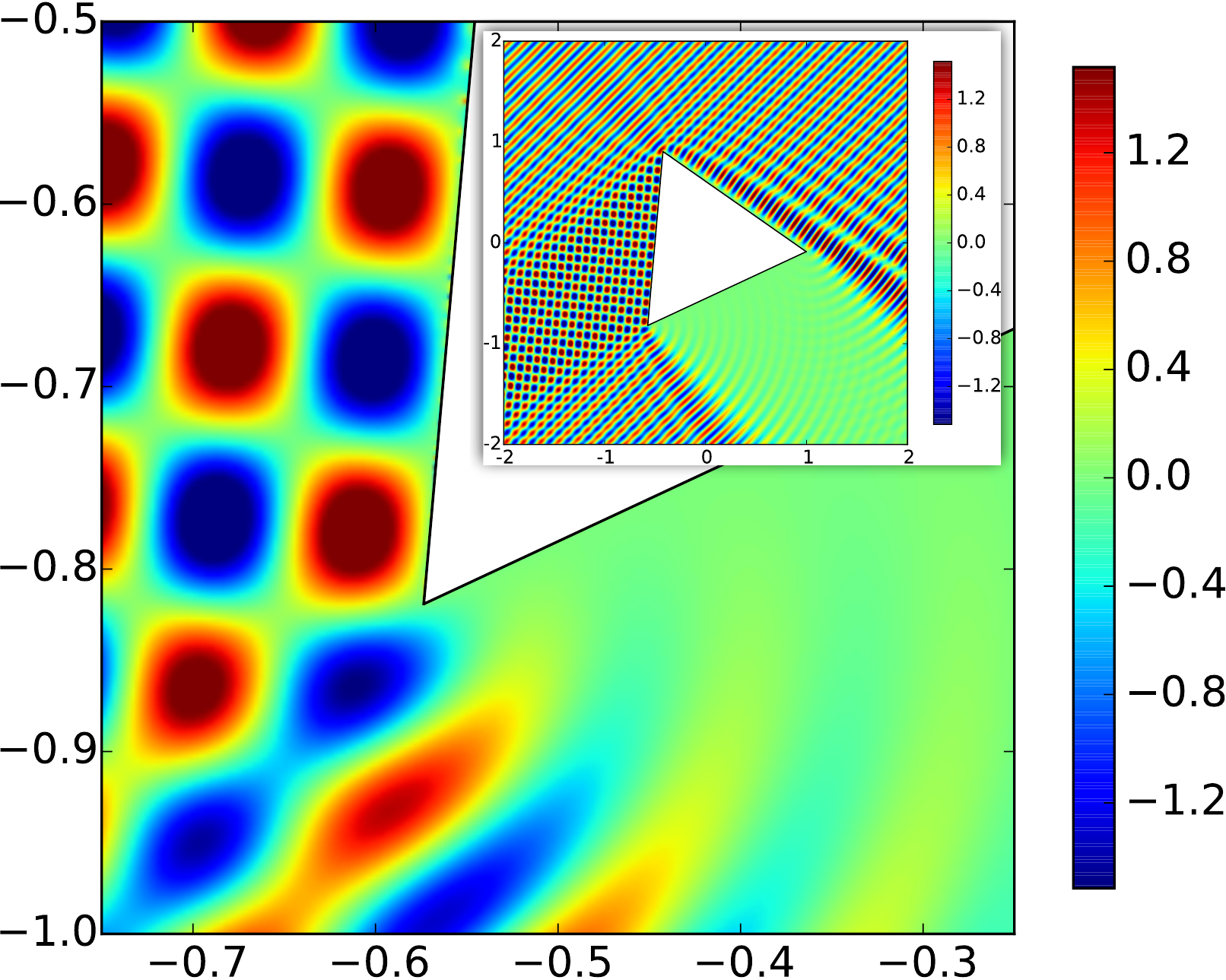}
      \caption{The corner problem, $k = 52.13 + i 10^{-7}$.}
      \label{fig_triangle1}
    \end{subfigure} \\ \vspace{\baselineskip}
  \begin{subtable}[b]{.45\linewidth}
    \centering
    \resizebox{\columnwidth}{!}{%
    \begin{tabular}{llll}\hline
      $h$ & $n$ & $RMSE$ error & Rel. $\ell_2$ error \\ \hline
      $0.4$    & $1056$ & $ 5.0 \times 10^{-2}$ & $ 2.2 \times 10^{-1}$ \\ 
      $0.2$    & $1248$ & $ 2.4 \times 10^{-2}$ & $ 1.1 \times 10^{-1}$ \\ 
      $0.1$    & $1344$ & $ 1.1 \times 10^{-2}$ & $ 4.6 \times 10^{-2}$ \\ 
      $0.05$   & $1440$ & $ 4.5 \times 10^{-3}$ & $ 2.0 \times 10^{-2}$ \\ 
      $0.025$  & $1632$ & $ 1.9 \times 10^{-3}$ & $ 8.5 \times 10^{-3}$ \\
      $0.0125$ & $1728$ & $ 8.4 \times 10^{-4}$ & $ 3.7 \times 10^{-3}$ \\
      $0.0$    & $3456$ & & \\
      \hline
    \end{tabular}   }
    \caption{Errors in the bi-static cross section 
      at $r = 10$ for $k
      = 52.13 + i 10^{-7}$.}
    \label{tab_triangle52}
  \end{subtable} \hfill
  \begin{subtable}[b]{.45\linewidth}
    \centering
    \resizebox{\columnwidth}{!}{%
    \begin{tabular}{llll}\hline
      $h$ & $n$ & $RMSE$ error & Rel. $\ell_2$ error \\ \hline
      $0.4$    & $1056$ & $ 1.5 \times 10^{-2}$ & $ 6.2 \times 10^{-2}$ \\ 
      $0.2$    & $1152$ & $ 6.3 \times 10^{-3}$ & $ 2.6 \times 10^{-2}$ \\ 
      $0.1$    & $1248$ & $ 2.7 \times 10^{-3}$ & $ 1.1 \times 10^{-2}$ \\ 
      $0.05$   & $1344$ & $ 1.2 \times 10^{-3}$ & $ 5.0 \times 10^{-3}$ \\ 
      $0.025$  & $1536$ & $ 5.1 \times 10^{-3}$ & $ 2.2 \times 10^{-3}$ \\
      $0.0125$ & $1632$ & $ 2.2 \times 10^{-4}$ & $ 9.4 \times 10^{-4}$ \\
      $0.0$    & $3360$ & & \\
      \hline
    \end{tabular}  }
    \caption{Errors in the bi-static cross section 
      at $r = 10$ for $k
      = 7.77 + i 10^{-6}$.}
    \label{tab_triangle7}    
  \end{subtable}
  \caption{A depiction of sound-soft (Dirichlet) scattering for
    various roundings, along with convergence of the bi-static cross
    section in the moderate near-field.  In each case, the PDE was
    solved to roughly a precision of $10^{-9}$ in the
    $\mathcal L_\infty$ norm (as determined by testing against a known
    solution).}
  \label{fig_triangle}
  \end{figure}
  \clearpage
}

\subsection{Scattering from smoothed polygons: Sound-hard}
\label{sec_neumann}

We now present results corresponding to the {\em sound-hard}
scattering problem, i.e. the exterior 
Neumann problem for the Helmholtz equation:
\begin{equation}
\begin{split}
( \Delta + k^2 ) u^{tot} &= 0  \qquad \text{in } \bbR^2\setminus\Omega,  \\
\frac{u^{tot}}{\partial n} &= 0  \qquad \text{on } \Gamma = \partial\Omega,
\end{split}
\end{equation}
along with suitable radiation conditions at infinity.
Representing the scattered solution $u$ using a single-layer 
potential:
\begin{equation}\label{eq_neu_rep}
  u = \mathcal S_k\sigma,
\end{equation}
we have the following second-kind integral equation along $\Gamma$
for the density $\sigma$: 
\begin{equation}\label{eq_inteq_neumann}
  -\frac{\sigma}{2}  + \mathcal S_k'  \sigma = -\frac{\partial
    u^{inc}}{\partial n} \qquad 
\text{on } \Gamma,
\end{equation}
where $\mathcal S_k' = \partial  \mathcal S_k / \partial n$ and is 
interpreted suitably as an on-surface limit.  
As before, our reference solver for the true corner problem follows
the method detailed in~\cite{bremer_2012}.

We also recall that using representation~\eqref{eq_neu_rep} may yield {\em
  spurious resonance} in the resulting integral equation for values of $k$
which correspond to eigenvalues of the interior Laplace Dirichlet problem. For
simplicity we have chosen $k$ to avoid these values.  Well-conditioned
combined-field representations exist which are invertible for all values of $k$
with $\Im{k}\geq 0$, but they involve the composition of layer potentials, as
in~\eqref{eq_neu_cfie}, not merely the summation~\cite{colton_kress}.  After
solving~\eqref{eq_inteq_neumann}, we evaluate the scattered field as in the
previous section, using the fast multipole method for the two-dimensional
Helmholtz equation and standard Gaussian quadrature.

\begin{figure}[t]
  \centering
  \begin{subfigure}[t]{.3\linewidth}
    \centering
    \includegraphics[width=1\linewidth]{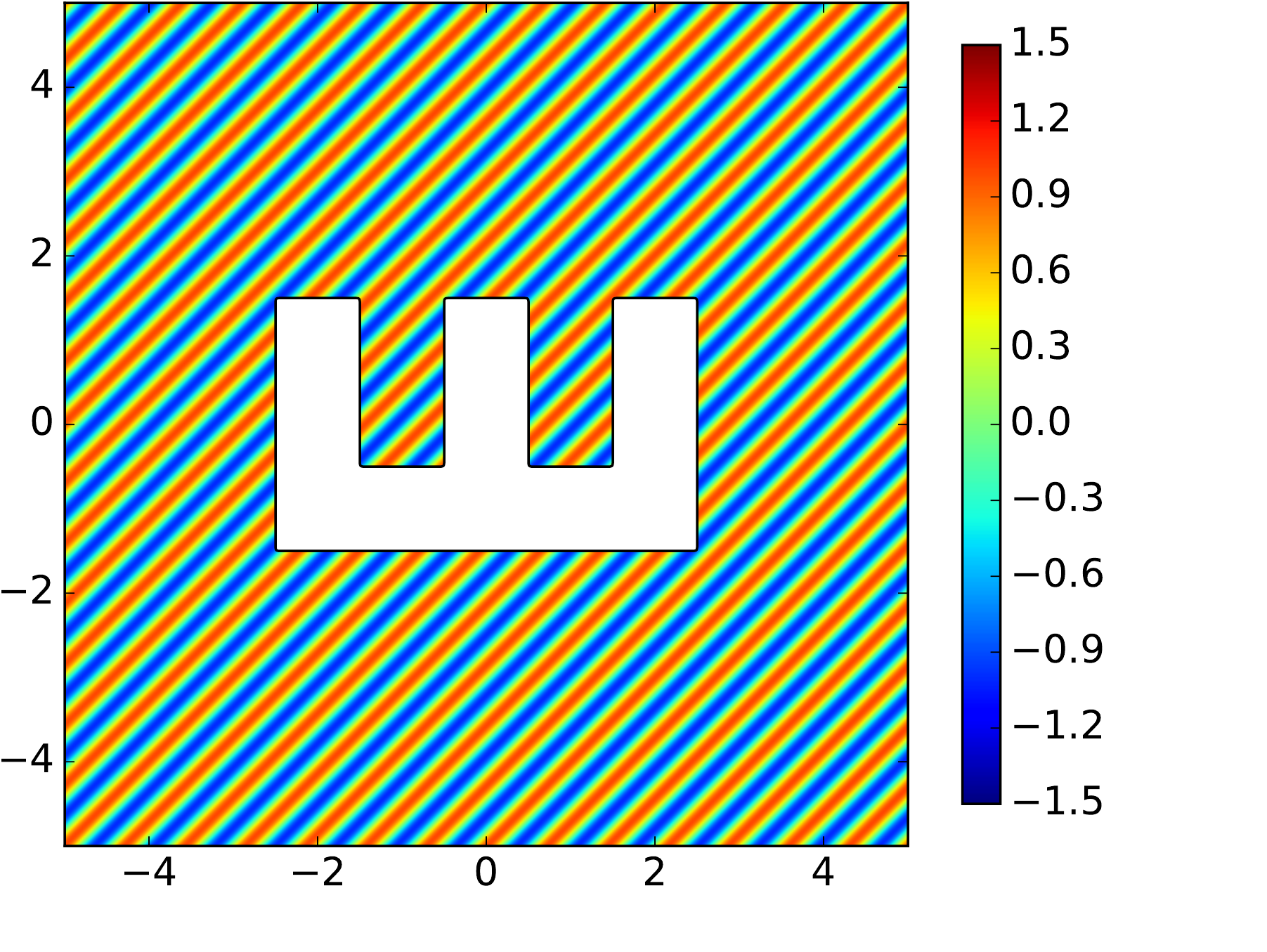}
    \caption{The incoming field.}
  \end{subfigure}
  \quad
  \begin{subfigure}[t]{.3\linewidth}
    \centering
    \includegraphics[width=1\linewidth]{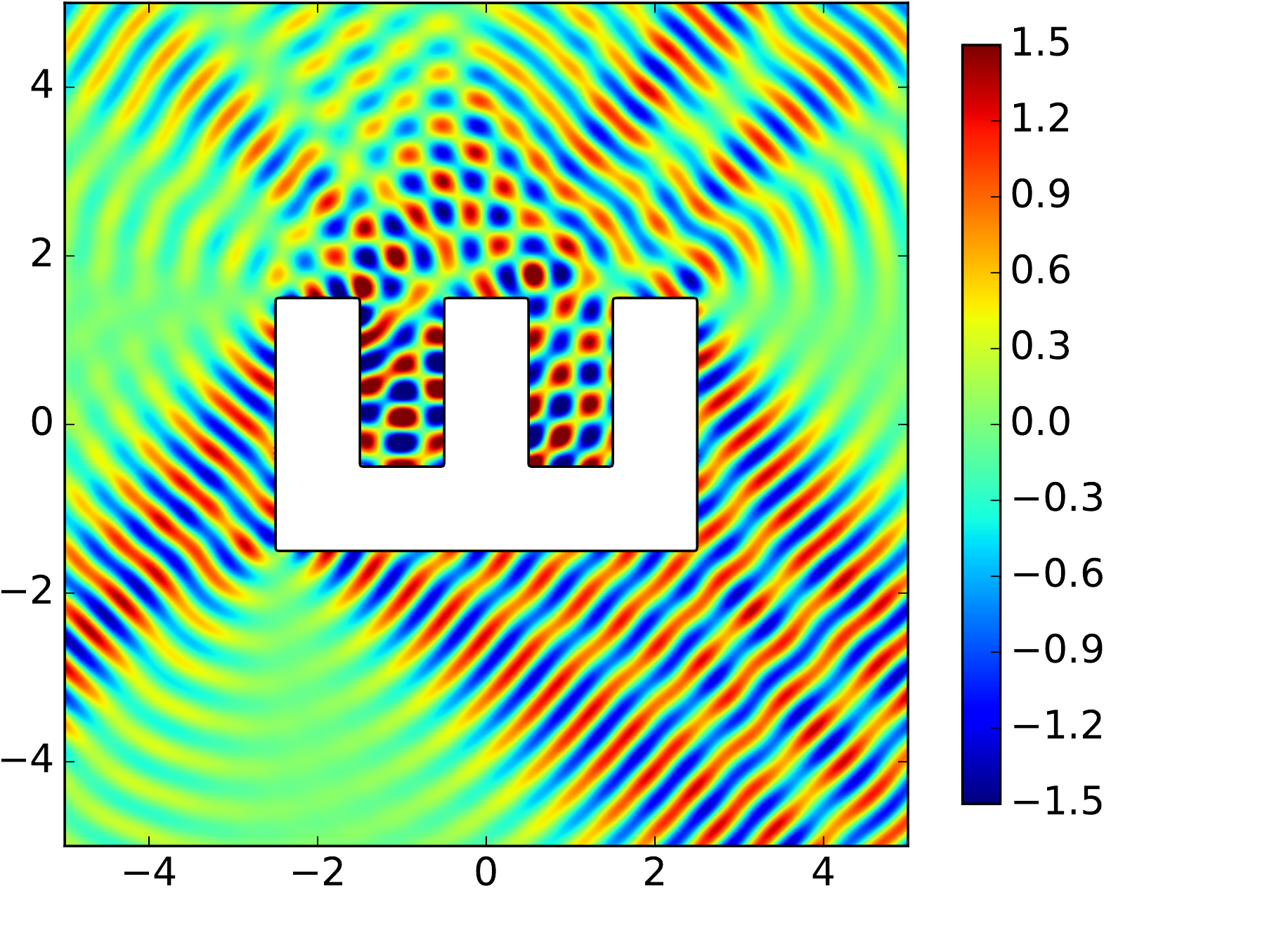}
    \caption{The scattered field.}
  \end{subfigure}
  \quad
  \begin{subfigure}[t]{.3\linewidth}
    \centering
    \includegraphics[width=1\linewidth]{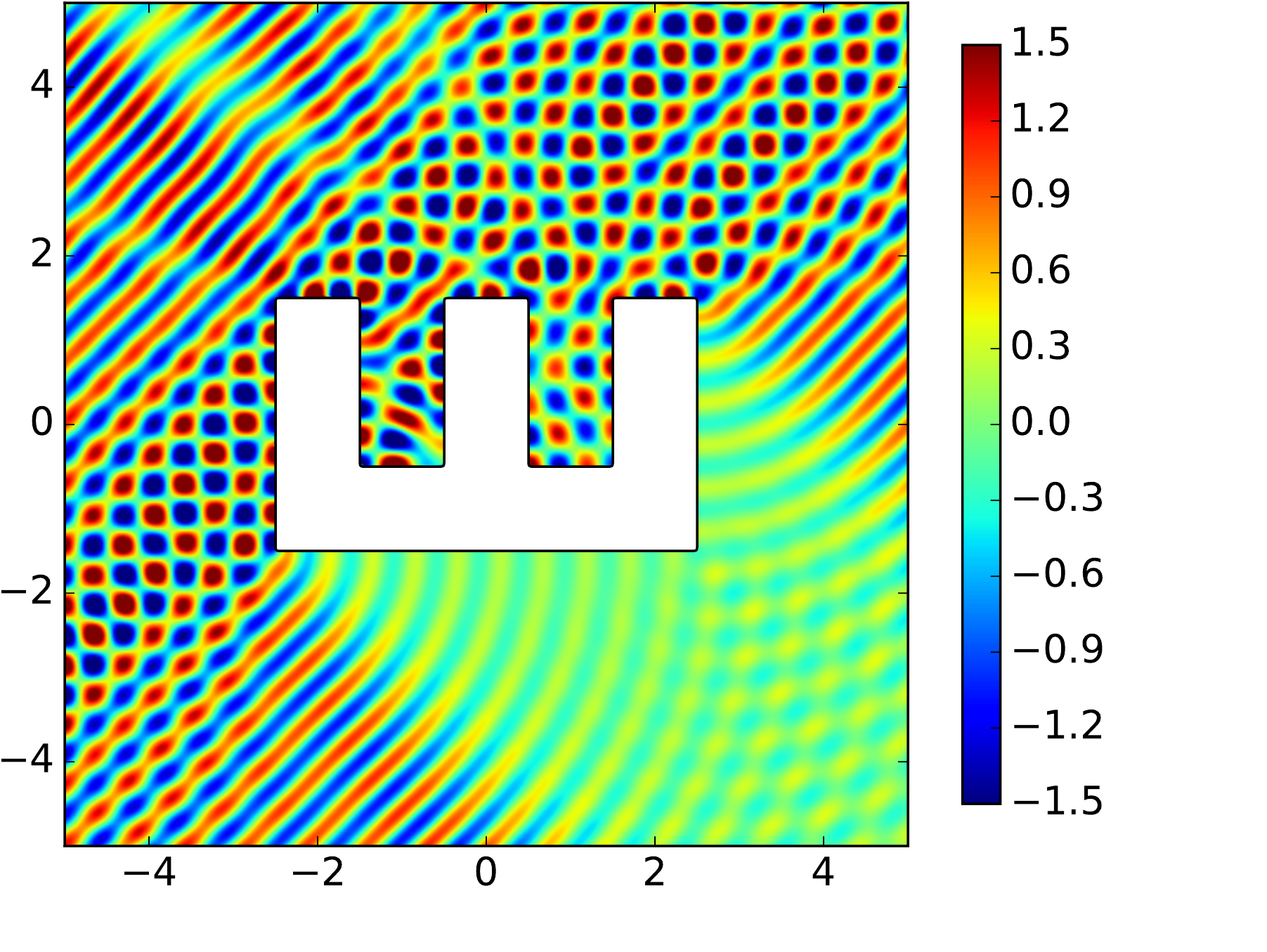}
    \caption{The total field.}
  \end{subfigure}
  \caption{Example exterior sound-hard (Neumann) 
    scattering problem for $k=12.43 + i 10^{-5}$. The real
    part of all fields are shown. The angle of the incident plane wave is
    $\phi = -\pi/4$.}
  \label{fig_neumann}
\end{figure}

The following simulations are obtained from driving the scattering
problem by setting $u^{inc}$ to be a two-dimensional plane-wave, as
before, traveling in the direction of the angle $\phi$:
\begin{equation}
u_\phi^{inc}(\bx) = e^{ik(x\cos\phi + y\sin\phi)}.
\end{equation}
See
Figure~\ref{fig_neumann} for a depiction of an incoming plane wave
$u_{-\pi/4}^{inc}$, scattered field $u$, and total field $u^{tot}$
with Neumann boundary conditions. In this example, $k = 12.43 + i
10^{-5}$, corresponding to a wavelength of
$\lambda = 2\pi/\Re{k} \approx 0.505 $.

The accuracy of the integral equation solver is tested, as before
in~\eqref{eq_testsolution}, 
by
comparison with a known test solution.
In order to study the effect of corner rounding for the Neumann
problem, we reproduce several of the experiments performed in the
Dirichlet case. In particular, we compare the bi-static SCS of the
true corner problem with that from successive roundings. See
Figures~\ref{fig_diffs_neumann}, \ref{fig_neu_triangle},
and~\ref{fig_neumann_regress} for plots of Neumann solutions and 
convergence results.

As in the Dirichlet case, as the size of the region that 
is rounded near the corners is
decreased, to below sub-wavelength, we see a convergence of the
bi-static cross section of roughly first-order.  We simulated
the Neumann problem at the same frequencies as in the Dirichlet case
for comparison.

\afterpage{
\begin{figure}[h]
  \centering
  \begin{subfigure}[t]{.5\linewidth}
    \centering
    \includegraphics[width=1\linewidth]{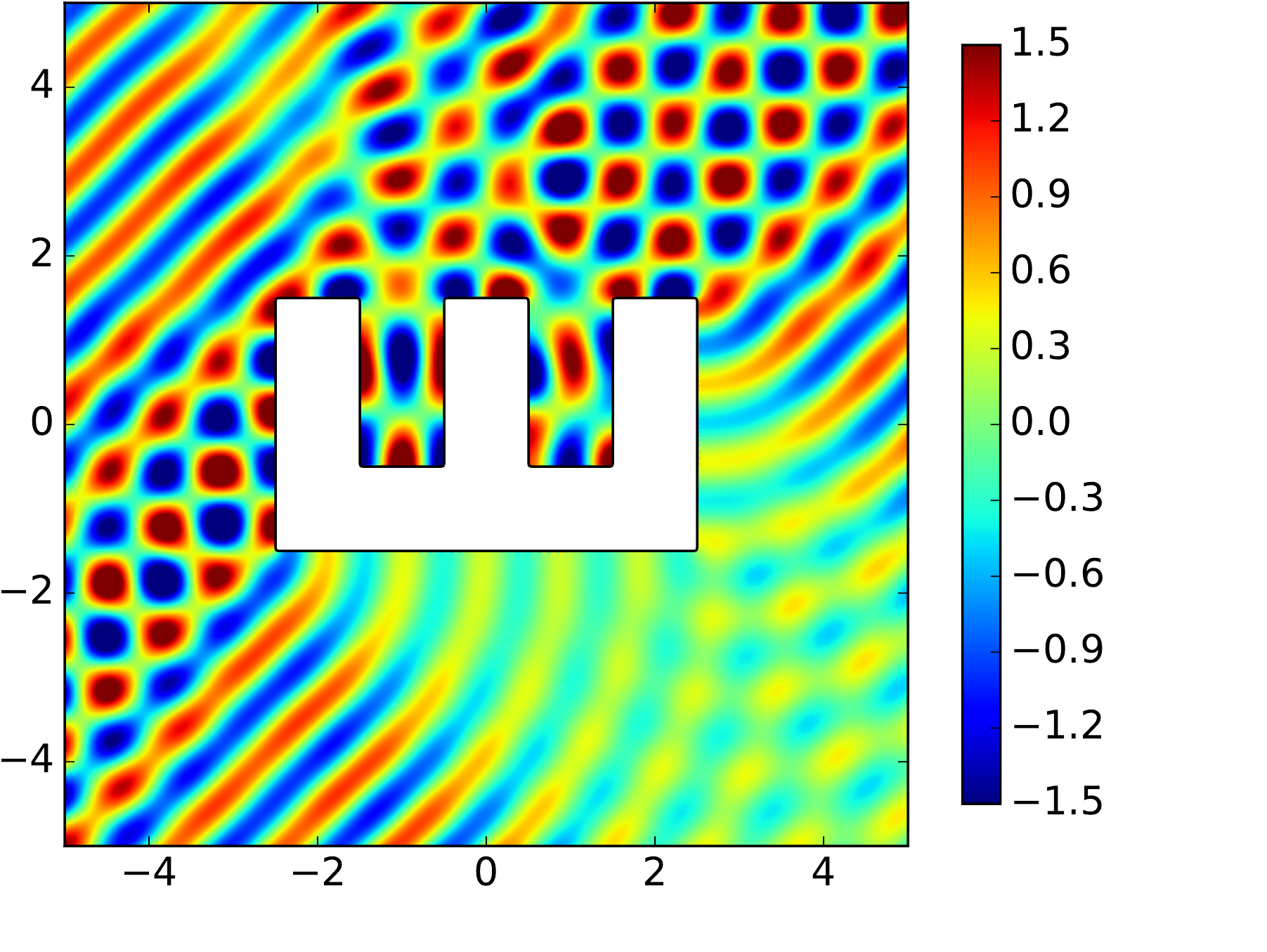}
    \caption{Total field for $k = 6.79 + 1.25i \times 10^{-6}$. The
      wavelength is approximately $\lambda \approx .93$, and the
      rounding parameter was $h=0.2$.}
  \end{subfigure}
  \hfill
  \begin{subtable}[b]{.45\linewidth}
    \centering
    \resizebox{\columnwidth}{!}{%
    \begin{tabular}{llll}\hline
      $h$ & $n$ & $RMSE$ error & Rel. $\ell_2$ error \\ \hline
      $0.4$ & $3872$ & $ 2.0 \times 10^{-3}$ & $ 4.6 \times 10^{-2}$ \\ 
      $0.2$ & $4576$ & $ 7.3 \times 10^{-4}$ & $ 1.7 \times 10^{-2}$ \\ 
      $0.1$ & $5088$ & $ 2.8 \times 10^{-4}$ & $ 6.7 \times 10^{-3}$ \\ 
      $0.05$ & $5632$ & $ 1.1 \times 10^{-4}$ & $ 2.7 \times 10^{-3}$ \\ 
      $0.025$ & $6016$ & $ 4.4 \times 10^{-5}$ & $ 1.1 \times 10^{-3}$ \\
      $0.0125$ & $6400$ & $ 1.8 \times 10^{-5}$ & $ 4.2 \times 10^{-4}$ \\
      $0.0$   & $13312$ & & \\
      \hline
    \end{tabular}  }
    \caption{Errors in the bi-static cross sections for $k = 6.79 +
      1.25i \times 10^{-6}$ at a radius of $1000 \approx 1081\lambda$.}
    \label{tab_neu_comb_6}
  \end{subtable} \\
  \vspace{\baselineskip}
  \begin{subfigure}[t]{.5\linewidth}
    \centering
    \includegraphics[width=1\linewidth]{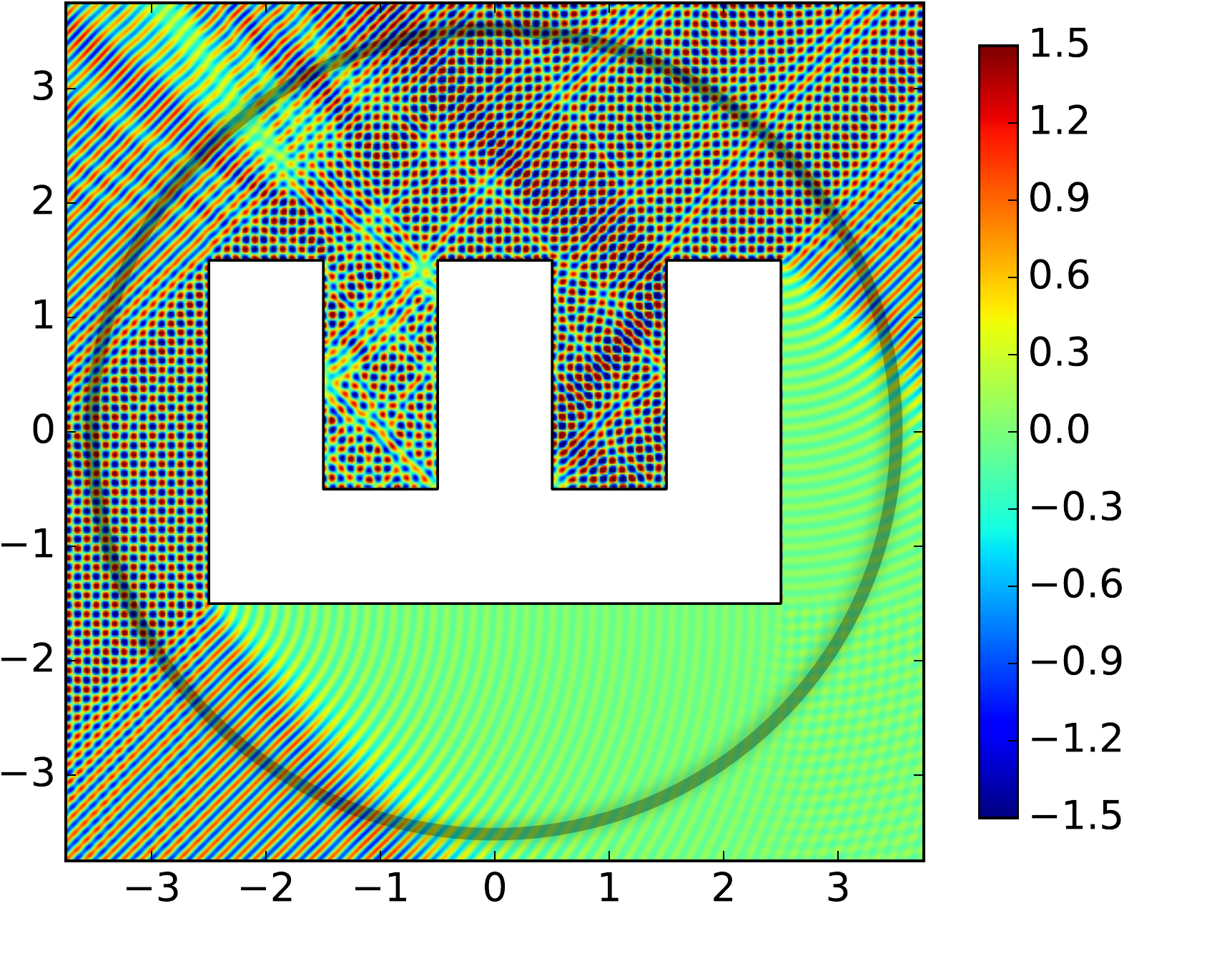}
    \caption{Total field for $k = 54.32 + i \times 10^{-5}$ along with
      testing curve for convergence of scattered field. The corners
      in this plot were rounded with a parameter $h=0.025$.}
  \end{subfigure}
  \hfill
  \begin{subtable}[b]{.45\linewidth}
    \centering
    \resizebox{\columnwidth}{!}{%
    \begin{tabular}{llll}\hline
      $h$ & $n$ & $RMSE$ error & Rel. $\ell_2$ error \\ \hline
      $0.4$ & $4608$ & $ 1.9 \times 10^{-1}$ & $ 2.5 \times 10^{-1}$ \\ 
      $0.2$ & $5312$ & $ 9.7 \times 10^{-2}$ & $ 1.3 \times 10^{-1}$ \\ 
      $0.1$ & $5824$ & $ 3.6 \times 10^{-2}$ & $ 4.7 \times 10^{-2}$ \\ 
      $0.05$ & $6752$ & $ 1.2 \times 10^{-2}$ & $ 1.6 \times 10^{-2}$ \\ 
      $0.025$ & $7296$ & $ 4.3 \times 10^{-3}$ & $ 5.5 \times 10^{-3}$ \\
      $0.0125$ & $7680$ & $ 1.6 \times 10^{-3}$ & $ 2.0 \times 10^{-3}$ \\
      $0.0$   & $14592$ & & \\
      \hline
    \end{tabular}  }
    \caption{Errors in the bi-static cross section for $k = 54.32 + i
      \times 10^{-5}$ in the near-field at a radius of $3.5 \approx
      30\lambda$.}
    \label{tab_neu_comb_near}
  \end{subtable}  
  \caption{Errors in the complex bi-static cross section (as compared
    to the true corner scattering problem) for the Neumann problem.
    The error converges approximately to first order in the rounding
    parameter $h$. In each case, the PDE was solved to roughly a
    precision of $10^{-9}$ in the $\mathcal L_\infty$ norm (as
    determined by testing against a known solution).}
  \label{fig_diffs_neumann}
\end{figure}
\clearpage
}

\afterpage{
  \begin{figure}[t]
    \centering
    \begin{subfigure}[t]{.45\linewidth}
      \centering
      \includegraphics[width=1\linewidth]{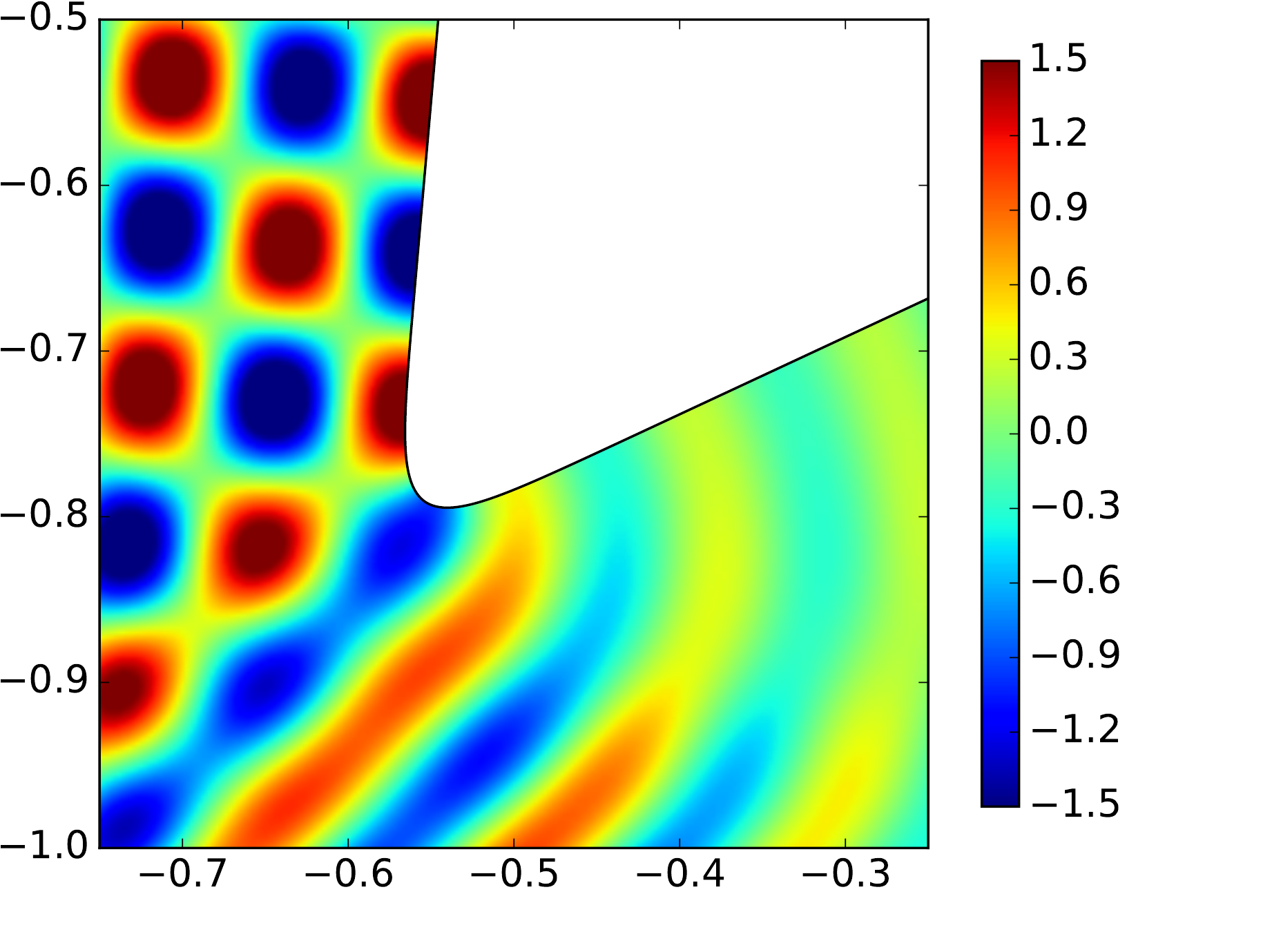}
      \caption{Rounding of $h = 0.4$, $k = 52.13 + i 10^{-7}$.}
    \end{subfigure} \hfill
    \begin{subfigure}[t]{.45\linewidth}
      \centering
      \includegraphics[width=1\linewidth]{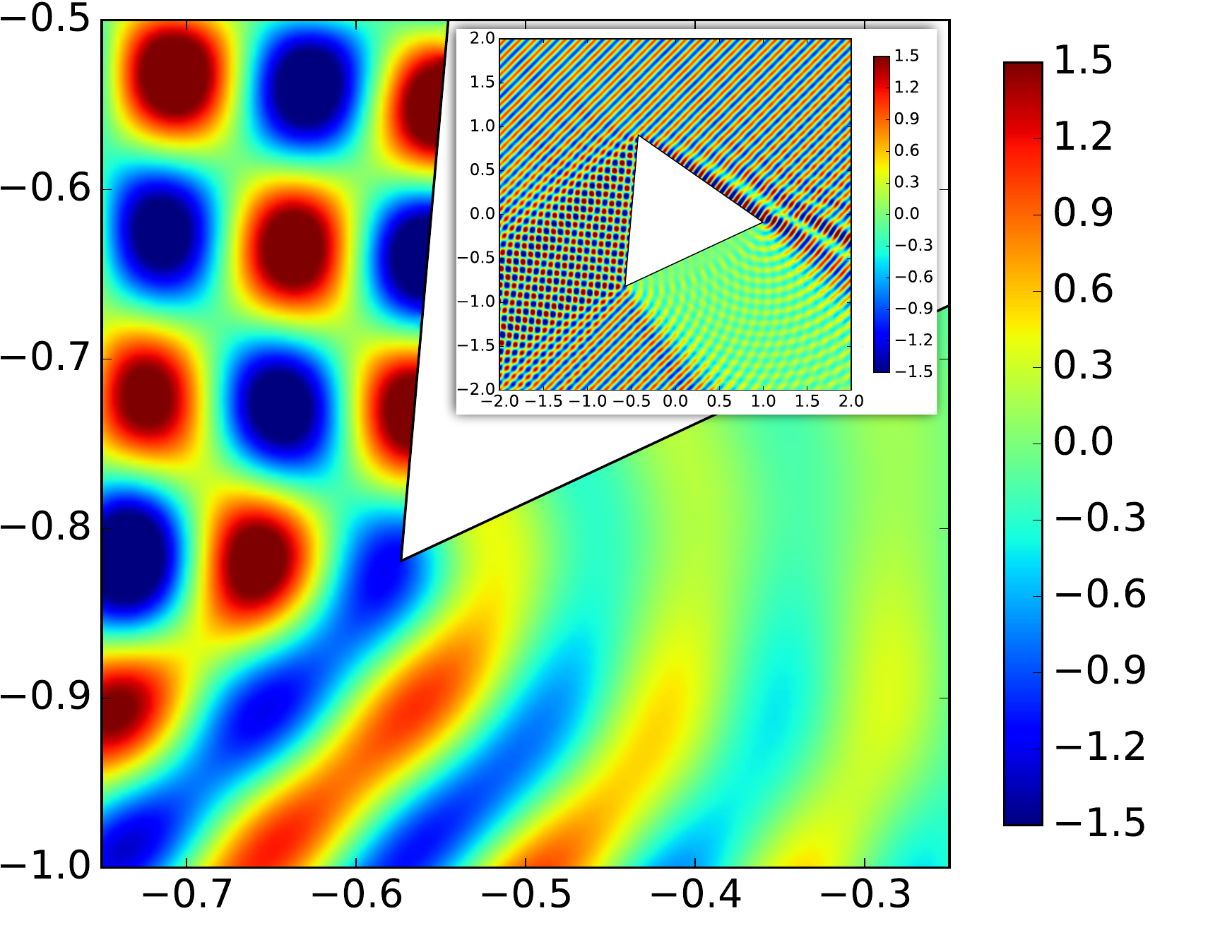}
      \caption{The corner problem, $k = 52.13 + i 10^{-7}$.}
      \label{fig_neu_triangle1}
    \end{subfigure} \\ \vspace{\baselineskip}
  \begin{subtable}[b]{.45\linewidth}
    \centering
    \resizebox{\columnwidth}{!}{%
    \begin{tabular}{llll}\hline
      $h$ & $n$ & $RMSE$ error & Rel. $\ell_2$ error \\ \hline
      $0.4$    & $1056$ & $ 8.0 \times 10^{-2}$ & $ 2.1 \times 10^{-1}$ \\ 
      $0.2$    & $1248$ & $ 4.2 \times 10^{-2}$ & $ 1.1  \times 10^{-1}$ \\ 
      $0.1$    & $1344$ & $ 2.1 \times 10^{-2}$ & $ 5.5 \times 10^{-2}$ \\ 
      $0.05$   & $1440$ & $ 9.8 \times 10^{-3}$ & $ 2.6 \times 10^{-2}$ \\ 
      $0.025$  & $1632$ & $ 4.4 \times 10^{-3}$ & $ 1.2 \times 10^{-2}$ \\
      $0.0125$ & $1728$ & $ 1.9 \times 10^{-3}$ & $ 5.1 \times 10^{-3}$ \\
      $0.0$    & $3456$ & & \\
      \hline
    \end{tabular}   }
    \caption{Errors in the bi-static cross section 
      at $r = 10$ for $k
      = 52.13 + i 10^{-7}$.}
    \label{tab_neu_triangle52}
  \end{subtable} \hfill
  \begin{subtable}[b]{.45\linewidth}
    \centering
    \resizebox{\columnwidth}{!}{%
    \begin{tabular}{llll}\hline
      $h$ & $n$ & $RMSE$ error & Rel. $\ell_2$ error \\ \hline
      $0.4$    & $1056$ & $ 2.4 \times 10^{-2}$ & $ 6.3 \times 10^{-2}$ \\ 
      $0.2$    & $1152$ & $ 1.0 \times 10^{-2}$ & $ 2.8 \times 10^{-2}$ \\ 
      $0.1$    & $1248$ & $ 4.5 \times 10^{-3}$ & $ 1.2 \times 10^{-2}$ \\ 
      $0.05$   & $1344$ & $ 2.0 \times 10^{-3}$ & $ 5.2 \times 10^{-3}$ \\ 
      $0.025$  & $1536$ & $ 8.6 \times 10^{-4}$ & $ 2.3 \times 10^{-3}$ \\
      $0.0125$ & $1632$ & $ 3.8 \times 10^{-4}$ & $ 9.9 \times 10^{-4}$ \\
      $0.0$    & $3360$ & & \\
      \hline
    \end{tabular}  }
    \caption{Errors in the bi-static cross section 
      at $r = 10$ for $k
      = 7.77 + i 10^{-6}$.}
    \label{tab_neu_triangle7}    
  \end{subtable}
  \caption{A depiction of sound-hard (Neumann) scattering for various
    roundings, along with convergence of the bi-static cross section
    in the moderate near-field.  In each case, the PDE was solved to
    roughly a precision of $10^{-9}$ in the $\mathcal L_\infty$ norm
    (as determined by testing against a known solution).}
  \label{fig_neu_triangle}
  \end{figure}

\begin{figure}[b]
  \centering
  \begin{subfigure}[t]{.45\linewidth}
    \centering
    \includegraphics[width=1\linewidth]{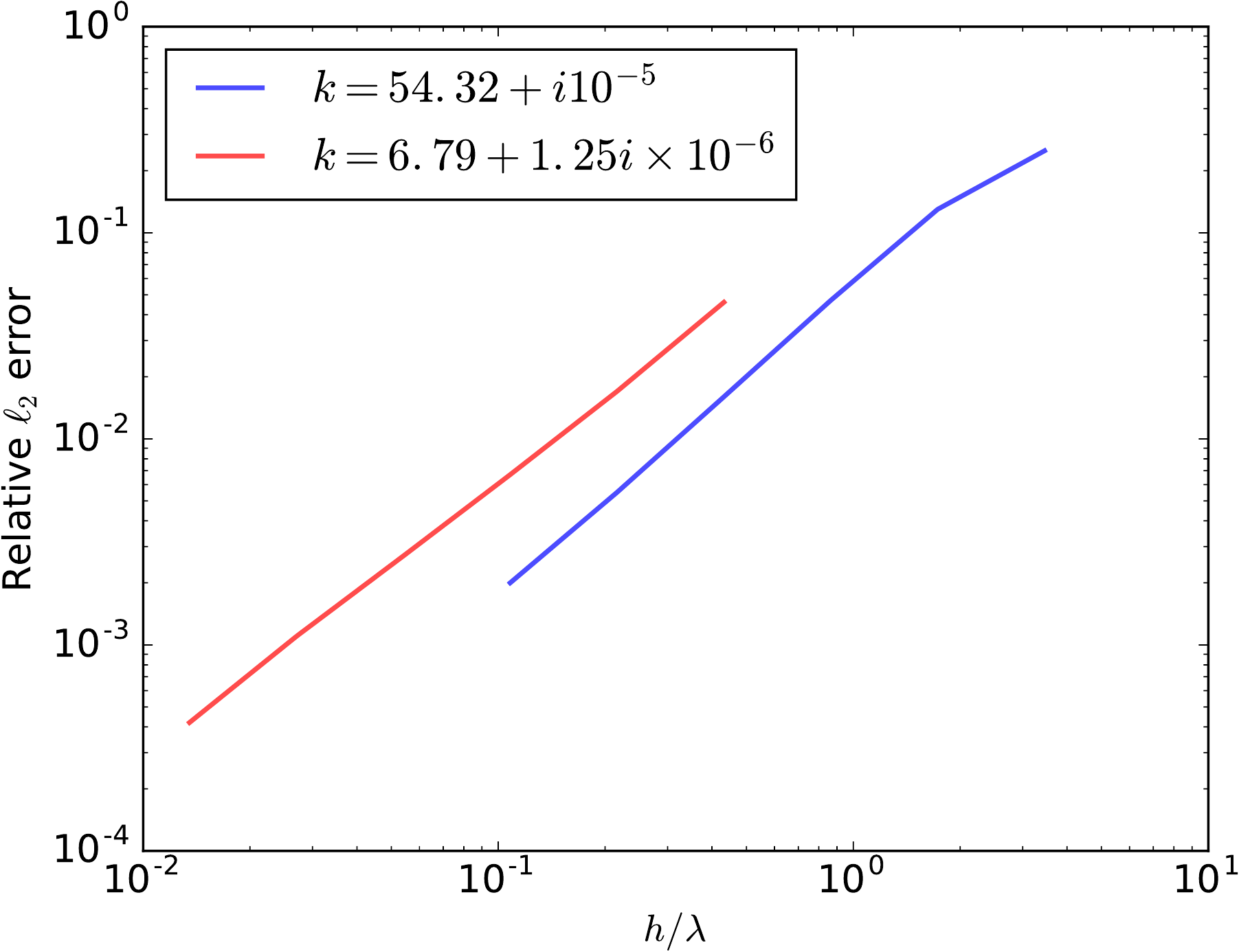}
    \caption{The empirical convergence for the comb shape with $k =
      54.32 + i \, 10^{-5}$ is $\mathcal O(h^{1.43})$ and that for $k
      = 6.79 + 1.25i \times 10^{-6}$ is $\mathcal O(h^{1.34})$.}
  \end{subfigure}
  \hfill
  \begin{subfigure}[t]{.45\linewidth}
    \centering
    \includegraphics[width=1\linewidth]{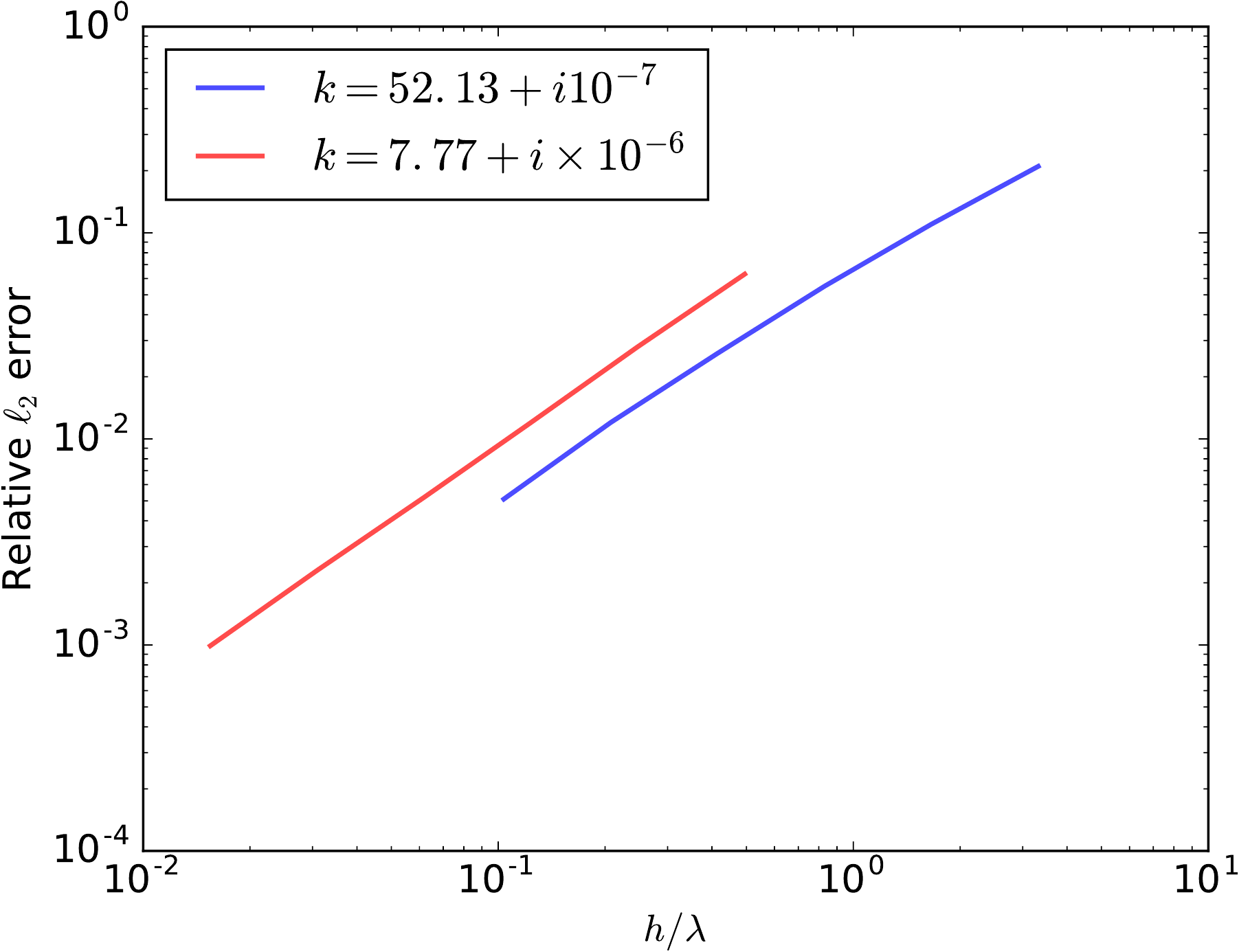}
    \caption{The empirical convergence for the triangle shape in
      Figure~\ref{fig_neu_triangle} with $k =
      52.13 + i \, 10^{-7}$ is $\mathcal O(h^{1.07})$ and that for $k
      = 7.77 + i\, 10^{-6}$ is $\mathcal O(h^{1.20})$.}
  \end{subfigure}
  \caption{Plot of the relative $\ell_2$ error of the scattered
    potential versus rounding size in terms of wavelength for the
    sound-hard (Neumann) problem. Both
    examples exhibit similar convergence.}
  \label{fig_neumann_regress}
\end{figure}
  \clearpage
}

\subsection{Extension to piecewise smooth boundaries} 

This technique can also be extended to piecewise smooth curvilinear
polygons. Since we need a variant of this idea to smooth polyhedra in
$\bbR^3,$ we pause to briefly describe it here. 
In short, in the neighborhood near a geometric singularity it is
possible to construct a diffeomorphism to a truncated cone. The corner
rounding can then be performed on the polygonal cone, and finally
composed with the inverse of the diffeomorphism to smooth the original
curvilinear polygon.

To this end, let $\cP$ be a region
in $\bbR$ whose boundary is composed of a finite collection of smoothly
embedded arcs, $\{\gamma_j:\: j=1,\dots, n\}$ meeting at points
\begin{equation}
  v_j=\gamma_j\cap\gamma_{j+1}
\end{equation}
and angles $\{0<\theta_j< 2\pi\}.$ We let $\gamma_{n+1}$ denote a
second copy of $\gamma_1.$ We are excluding the case of a cusp,
i.e. $\theta_j=2\pi.$

Once again the idea is to change only a small neighborhood each
vertex.  We define (in complex notation) the planar regions
\begin{equation}
\begin{aligned}
  W_j&=\{z: 0\leq\arg z\leq\theta_j\text{ and }|z|<1\}, &\qquad
  &\text{if }\theta_j<\pi,\\
  W_j&=\{z: 0\leq\arg z\leq 2\pi-\theta_j \text{ and }|z|<1\}, &
  &\text{if }\theta_j>\pi.
\end{aligned}
\end{equation}
Suppose that for each $j$ for which $\theta_j<\pi$ we can find a
diffeomorphism $\psi_j$ from $W_j$ to a neighborhood of $v_j$ in
$\cP,$ which carries:
\begin{equation}
  \begin{aligned}
    0 &\to v_j, \\
    \{z:\arg z=0\}\cap \partial W_j &\to \text{ a ray in } \gamma_j,\\
    \{z:\arg z=\theta_j\}\cap \partial W_j &\to \text{ a ray in } \gamma_{j+1}.
  \end{aligned}
\end{equation}
If $\theta_j>\pi,$ then $\psi_j$ is defined from a neighborhood of the
vertex in $W_j$ to a neighborhood of~$v_j$ in~$\overline{\cP^c},$ with
the boundary correspondence as before. Conformal mapping provides one
effective method to define such maps. Other, more elementary
techniques are also available. One such method, which works for
regions with convex boundaries, is described in Section~\ref{sect4}.

For each $h>0$ we define $W_j^h$ as the regions obtain by smoothing
the vertex of \mbox{$W_j$ at $0$} as described above.  For small $h,$ we have
$W_j^h\subset W_j,$ and the boundaries of $W_j^h$ and $W_j$ coincide
outside of a small neighborhood of $0.$ Thus, for small enough $r,$
the image $\psi_j(\partial W_j^h\setminus B_r(0))$ lies along the
boundary 
of~$\cP$ outside a small neighborhood of~$v_j.$ Hence the image
$\psi_j(\partial W_j^h)$ defines a smoothing of the vertex at $v_j.$ 
This procedure is done locally in a small neighborhood of each vertex,
allowing one to smooth the vertices while leaving as much of the
remainder of the boundary of $\cP$ fixed as desired.\\

\section{Polyhedra in three dimensions} \label{sec-3d}

In this section we describe several methods for smoothing
piecewise smooth boundaries of regions in $\bbR^3$. In
Section~\ref{3reg} we describe a special class of polyhedra,
\emph{3-regular Hamiltonian polyhedra}, whose boundaries can be smoothed
using the method described above {\em with a parameter}.  In fact, all
convex polyhedra, and many non-convex polyhedra can be smoothed this
way, but the results are often not-optimal.  In Section~\ref{ss4.2} we
show that by modifying a polyhedron in a small neighborhood of its
vertices one can obtain a 3-regular, Hamiltonian polyhedron. Hence it
can be smoothed using the method given in Section~\ref{3reg}.  This
leads to a smoothed boundary that agrees with the original polyhedron
outside a small neighborhood of the original edges and vertices.
Unfortunately, the smoothed polyhedron will also contain open subsets
of translated support planes of the vertices. This is both unsightly
and can produced a dramatically enhanced scattered wave in the
direction normal to the plane.  A more robust approach is described
in Section~\ref{sect3}.

\subsection{3-Regular Hamiltonian Polyhedra}\label{3reg}

There is a special collection of polyhedra in $\bbR^3$ whose edges and
vertices can be smoothed using only what might be called {\em the
two-dimensional method with parameter}. We first define this
class:
\begin{definition} Let $P$ be a polyhedron in $\bbR^3,$ and $G_P$ the
  graph defined by its edges. $P$ is 3-regular if every vertex is the
  intersection of three faces, or, equivalently, if $G_P$ is a 3-regular
  graph.  It is Hamiltonian if there is a finite collection of
  disjoint cycles $\{C_1,\dots,C_l\}\subset G_P$ so that every vertex
  belongs to exactly one of these cycles.
\end{definition}

It turns out that not every 3-regular polyhedron is Hamiltonian, and
when one is, the problem of finding these cycles is not generally
solvable in polynomial time. On the other hand, all 3-regular Platonic
solids (tetrahedron, cube, and dodecahedron) are Hamiltonian, as well
as many examples that arise in practice. As this class of polyhedra
can be smoothed by smoothing only edges, we take a moment to describe
the procedure.

Let $P$ be a 3-regular Hamiltonian polyhedron, with
$\cC=\{C_1,\dots,C_l\}$ a collection of disjoint cycles exhausting the
vertices. Let $\cE=\{e_1,\dots,e_m\}$ be the edges of $G_P$ that are
\emph{not} contained in any cycle. Because the graph is 3-regular, we
know that every vertex in $G_P$ lies on exactly one of these
edges. Moreover the edges in $\cE$ are disjoint.

To smooth $P$ we first smooth the corners that lie along the edges in
$\cE$ (using the two-dimensional method described earlier).  It is
easy to see that every edge $e_j$ lies in the intersection of two
planes $\{\pi_{j_1},\pi_{j_2}\}.$ Let $\pi_{j_3}$ be a plane
orthogonal to the line  $\ell_j=\pi_{j_1}\cap\pi_{j_2}$, and $\omega_j
=(\varphi,\theta) \in S^2$ be the direction of $\ell_j$. Finally let
$\gamma_j$ denote the component of
$(\pi_{j_1}\cup\pi_{j_2})\cap\pi_{j_3}$ so that a neighborhood of $e_j$
in $P$ lies inside the corner
\begin{equation}
  K_j=\{q+t\omega_j:\: q\in\gamma_j\text{ and }t\in\bbR\}.
\end{equation}
If $\gamma'_j$ is a smoothing of this curve, as defined in the
previous section, then
\begin{equation}
  K'_j=\{q+t\omega_j:\: q\in\gamma'_j\text{ and }t\in\bbR\}
\end{equation}
is a smoothing of the corner. If the rounding of $\gamma_j$ is done
close enough to the vertex, then we can smoothly replace a
neighborhood of $e_j$ in $P$ with its smoothed version in $K'_j$ by
simply intersecting the interior of the region bounded by $K'_j$ with
$P.$ Away from the smoothed edge, $K_j'$ is still a union of planar
regions which can be glued onto $P$, thereby replacing $e_j$ with a
smooth transition between these planar regions.

Since the edges in $\cE$ do not intersect, each of these smoothing
operations can be done independently of the others. Let $P'$ denote
the body in $\bbR^3$ obtained by smoothing all of these edges. Since
every vertex lies on one of the edges in $\cE$, the cycles on $P$ are
replaced by cycles $\cC'=\{C'_1,\dots,C'_l\}$ on $P'$ that are
\emph{smooth} non-intersecting curves. That is to say, every vertex
has been smoothed. The boundary of the body $P'$ is a comprised of
bounded smooth surfaces, which are mostly planar
regions. These surfaces are bounded by smooth, disjoint, closed
curves, along which these surfaces meet. All that remains is to smooth
these curves of intersection.

To that end we now define a diffeomorphism from a standard model onto
a neighborhood of $C_j'$. We smooth the standard model and use this
map to glue the result into $P'$. Let $c_j:[0,L_j]\to P'$ be an
arclength parameterization of $C_j'$. The unit vector field
$T_j(t)=\pa_t c_j$ is tangent to $C_j'$. Two smooth surfaces $S_{1j}$
and $S_{2j}$ meet, transversely, along this curve. Let
$N_{ij}(t)$ be the unit vector normal to $T_j(t)$ lying along
$S_{ij},$ $i=1,2$. Let $\pi_j(t)$ denote the plane through $c_j(t)$
spanned by $\{ N_{1j}(t),N_{2j}(t) \}$.

For $\epsilon>0$, let $U_{j\epsilon}$ denote the
$\epsilon$-neighborhood of $C_j'$. There is a radius $\epsilon>0$ so
that these planes $\{\pi_j(t):\: t\in [0,L_j]\}$ define a foliation of
$U_{j\epsilon}.$ This follows from the inverse function theorem and
the compactness of the curve. We define a map from
$V_{j\delta}=[0,L_j)\times (-\delta,\delta)\times(-\delta,\delta)$
  into a neighborhood of $C_j'$ by letting
\begin{equation}\label{eqn12}
  \Phi_j(t,s_1,s_2)=c_j(t)+s_1N_{1j}(t)+s_2N_{2j}(t).
\end{equation}
The differential of $\Phi_j$ at $(t,0,0)$ is given by
\begin{equation}
  d\Phi_j(t,0,0)=T_j(t)dt+N_{1j}(t)ds_1+N_{2j}(t)ds_2,
\end{equation}
which is clearly of rank three. From the inverse function theorem it
now follows easily that there is an $\delta>0$ so that
$\Phi_j\restrictedto_{V_{j\delta}}$ is a diffeomorphism onto its
image, which is a neighborhood of $C'_{j}$ foliated by the planes
$\{\pi_j(t)\}.$ We can continue $\Phi_j$ as a smooth $L_{j}$-periodic
function.

We first make the assumption that $P'\cap U_{j\epsilon}$ lies in the
image of the positive orthant in the $(s_1,s_2)$-variables under this
map. This is certainly the case if the interior angle along $C_j'$ is
everywhere less than $\pi$. Under this assumption it is easy to see
that for an $\eta>0$ there is a set of the form $V_{j\eta}=\bbR\times
[0,\eta]\times [0,\eta]$ on which $\Phi_j$ is a
{\em periodic}-diffeomorphism. Moreover $\Phi_j(V_{j\eta})\subset P'$
exhausts a neighborhood of $C'_j$ in $P'.$

We now smooth the corner $V_{j\eta}$ to obtain $V'_{j\eta}$, where the
smoothed edge lies in $V_{j\frac{\eta}{2}}.$ The image of the smoothed
corner under $\Phi_{j}$ defines a smoothing of $C'_{j}$. As these
curves are disjoint, each one can be smoothed independently. We let
$P''$ denote the resulting body in $\bbR^3.$ It is a smoothed
approximation to $P.$ If the interior angle is greater than $\pi$ at
every point, then we can smooth the corner by smoothing the exterior,
which satisfies the hypotheses above.

We demonstrate this approach to smoothing polyhedra by smoothing a
cubic torus $P$.  We suppose that $P$ is oriented parallel to the
standard coordinate axes. There are four cycles $\{C_1, \ldots,
C_4\}$, each containing four edges and parallel to the $xz$-plane.
These cycles bound the faces that have non-trivial topology.  For the
edges not belonging to cycles, $\cE$, we use the eight edges parallel
to the $y$-axis.  If the edges in $\cE$ are smoothed, then the cross
sections of $P'$ perpendicular to the $y$-axis are smoothed squares,
as shown in Figure~\ref{fig_torus}.

\begin{figure}[t]
  \centering
  \begin{subfigure}[t]{.45\linewidth}
    \centering
    \includegraphics[width=.95\linewidth]{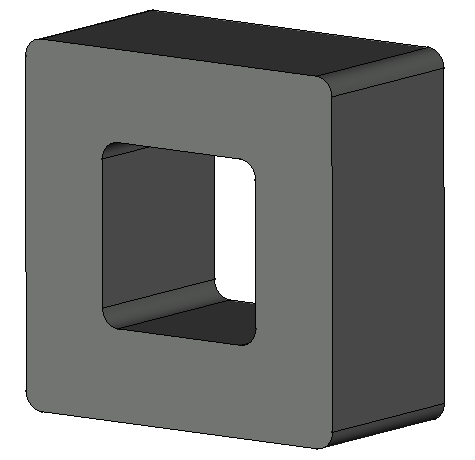}
    \caption{Rounded edges and cycles of the cubic torus.}
  \end{subfigure}
  \quad
  \begin{subfigure}[t]{.45\linewidth}
    \centering
    \includegraphics[width=.95\linewidth]{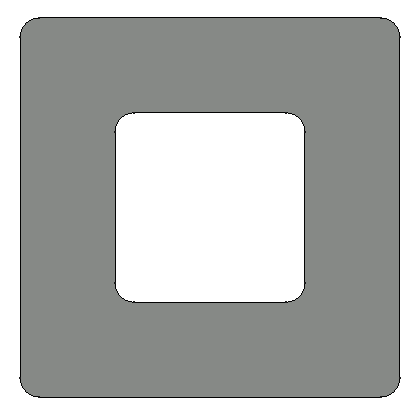}
    \caption{Rounded edges/cycles of a face.}
  \end{subfigure}
  \caption{Rounding Hamiltonian cycles of a cubic torus. Images were
    constructed merely for illustrative purposes only using FreeCAD.}
  \label{fig_torus}
\end{figure}

We now smooth the remaining edges using the representation in
equation~\eqref{eqn12} along with the smoothing of the right angle
used to smooth the edges in $\cE$. Figure~\ref{fig_torus2} shows two
views of the upper part of the final smoothed cubic torus.  We should
point out that these images were constructed using the software
FreeCAD using low-order fillet procedures, and are illustrative
only. Constructing the high-order computational geometry software to
carry out convolutional smoothing and subsequent high-order piecewise
triangulation for polyhedra in three
dimensions is an ongoing project.

\subsection{Smoothing the Vertices: I}
\label{ss4.2}

The method for smoothing edges described in the previous section can
be used to smooth an arbitrary convex polyhedron in a two-step
procedure. Let $\cP$ be a convex polyhedron with faces
$\cF=\{f_1,\dots,f_l\}$, edges $\cE=\{e_1,\dots,e_m\}$, and vertices
$\cV=\{v_1,\dots,v_n\}$.  At each vertex we choose an outward pointing
support vector, $\{\bnu_1,\dots,\bnu_n\}$. Suppose that the edges at
the $j^\text{th}$ vertex join to the vertices $\{v_{k_1},\dots,v_{k_p}\}.$ A
good choice for $\bnu_j$ is to take
\begin{equation}
  \bnu_j=\frac{1}{p}\sum_{q=1}^p\frac{v_j-v_{k_q}}{\|v_j-v_{k_q}\|},
\end{equation}
as it will preserve whatever symmetries the original polyhedron
possesses in the smoothed domain.

Given $\epsilon>0,$ we define a neighborhood $V_{\epsilon}$ of the
vertices by the condition
\begin{equation}
    X\in V_{\epsilon}\text{ if for some } j\text{ we have } \langle
    X-v_j,\bnu_j\rangle>-\epsilon.
\end{equation}
Note that $\cP\subset V_0^c,$ and, for small $\epsilon>0$, the
intersection $\partial V_{\epsilon}\cap \cP$ is a disjoint union of small
polygons lying near the vertices. See Figure~\ref{fig206}. It is easy to see
that the resultant polyhedron is 3-regular and Hamiltonian, with the disjoint cycles
being those introduced when cutting off the vertices.

To smooth the polyhedron we first smooth the edges. An edge $e_k$ lies
in the intersection of two faces $f_{i_k} \cap f'_{i_k}.$ Let $\pi_k$
be the plane, through the midpoint of the edge, which is perpendicular
to $e_k.$ Using the method described in Section~\ref{sect2d} we can
smooth the vertex $e_k\cap\pi_k$ of the polygon defined by the
intersection of $\pi_k$ with $\cP.$ By parallel translating this
smoothed vertex along the edge, we can replace a neighborhood of the
edge $e_k$ by a smooth surface joining the plane containing $f_{i_k}$
to the plane containing $f'_{i_k}.$ With $h>0$ the smoothing parameter
from Section~\ref{sect2d}, we let $\cP_h$ denote the polygon with all
its edges smoothed in this manner.

Of course, near enough to a vertex, the smoothings of different edges
intersect, but given $\epsilon>0$ we can choose a sufficiently small
$h>0$ so that, in the set $V_{\epsilon}^c,$ the modifications
corresponding to the different edges are disjoint. With such choices,
the intersection $\partial V_{\epsilon}\cap\cP_h$ is a \emph{disjoint}
union of polygons with smoothed vertices lying in the planes
\begin{equation}\label{eqn18.01}
  \langle X-v_j,\bnu_j\rangle=-\epsilon.
\end{equation}

\begin{figure}[t]
  \centering
  \begin{subfigure}[t]{.3\linewidth}
    \centering
    \includegraphics[width=.99\linewidth]{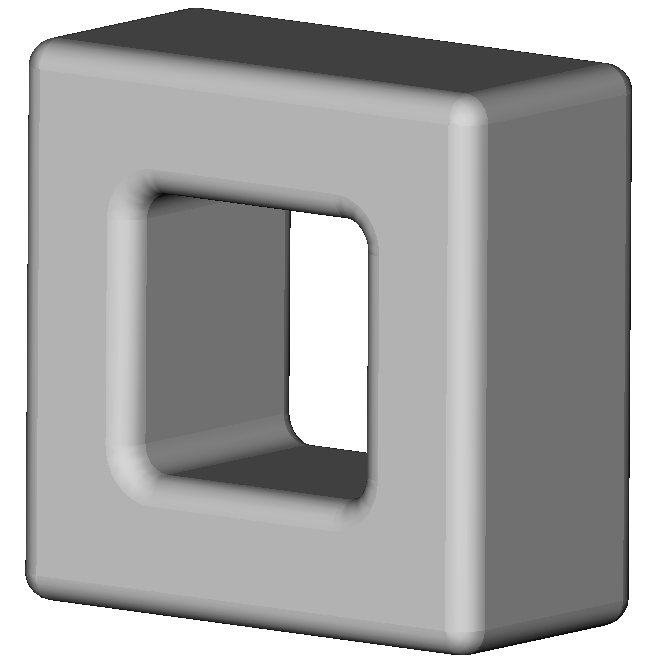}
    \caption{Final smooth surface.}
  \end{subfigure}
  \quad
  \begin{subfigure}[t]{.3\linewidth}
    \centering
    \includegraphics[width=.99\linewidth]{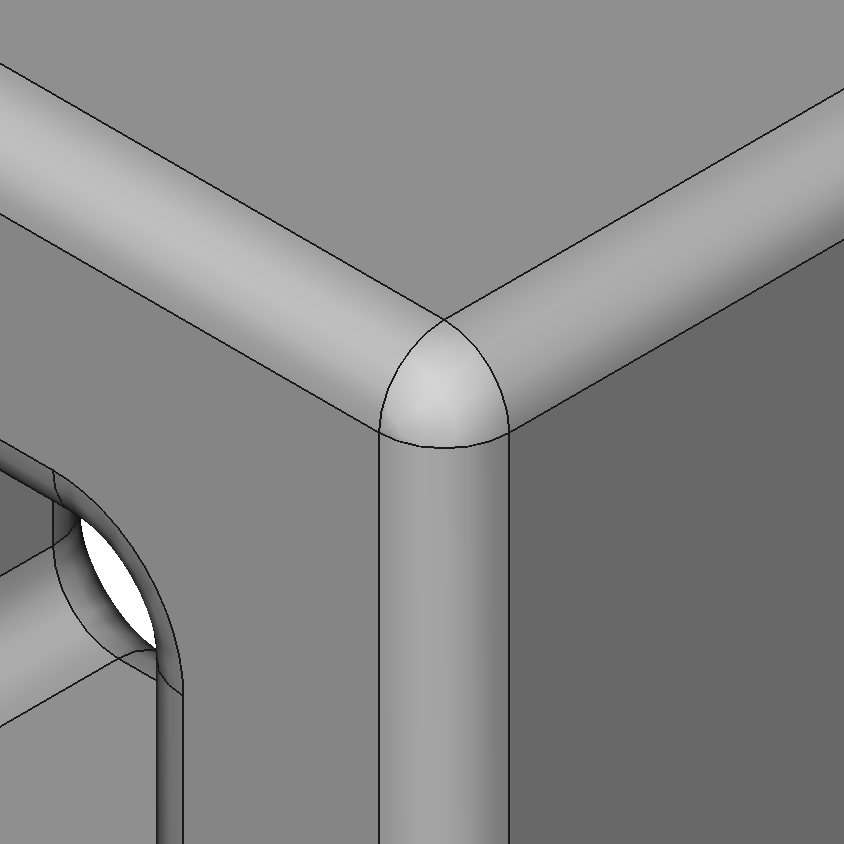}
    \caption{Smoothed exterior corner.}
  \end{subfigure}
  \quad
  \begin{subfigure}[t]{.3\linewidth}
    \centering
    \includegraphics[width=.99\linewidth]{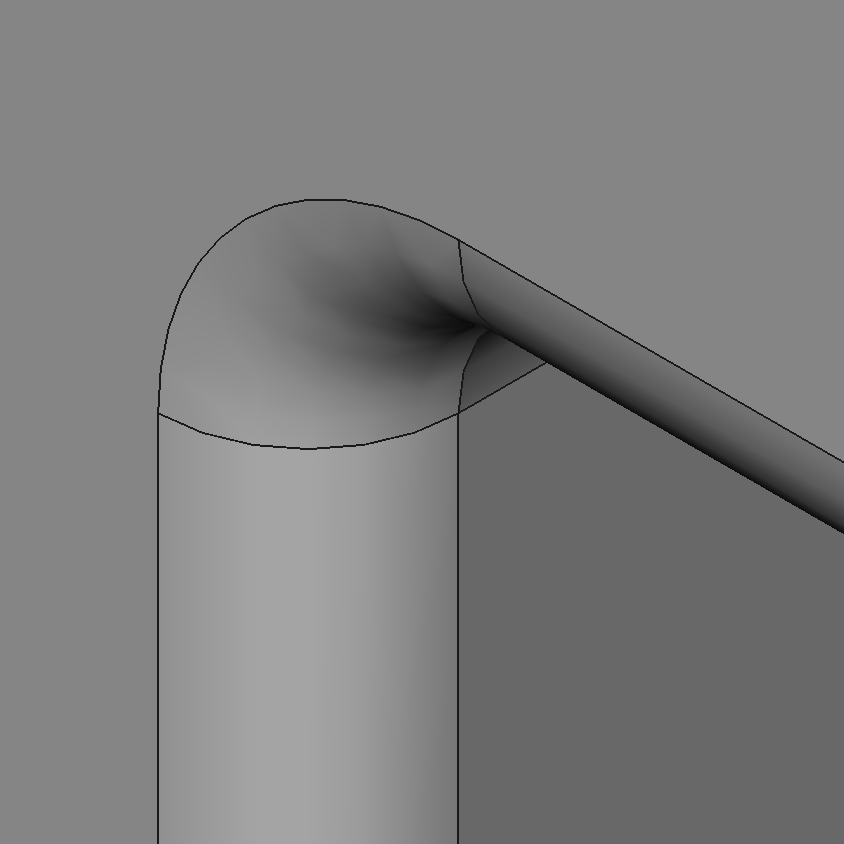}
    \caption{Smoothed interior corner.}
  \end{subfigure}
  \caption{Final rounding of Hamiltonian cycles of a cubic
    torus. Images were constructed merely for illustrative purposes
    only using FreeCAD.}
  \label{fig_torus2}
\end{figure}

Using the technique described in Section~\ref{3reg} the edges along
which these smoothed polygons meet $\partial\cP_h$ can be smoothed,
leading to an overall smoothing of the original polyhedron. While it is
clear that this can be done in an arbitrarily small neighborhood of
the singular locus of $\partial\cP,$ the $j^\text{th}$ vertex is
replaced by a smooth surface containing a open subset of the plane
defined in~\eqref{eqn18.01}. For applications to scattering theory
this might not be desirable, as it will produce a considerable
amplification of the scattered signal in this direction. In the next
section we describe a method for smoothing vertices that produces a
better result.

\begin{figure}[t]
  \centering
  \begin{subfigure}[t]{.45\linewidth}
    \centering
    \includegraphics[width=.95\linewidth]{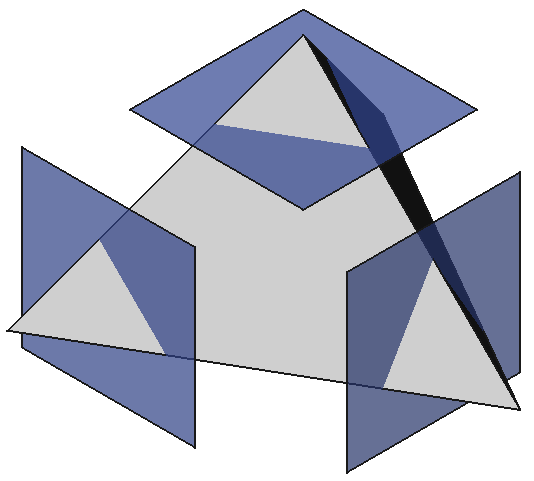}
    \caption{The intersections of $\cP$ with planes \\
      $\langle
      X-v_j,\bnu_j\rangle=-\epsilon$.}
    \label{fig206}
  \end{subfigure}
  \quad
  \begin{subfigure}[t]{.45\linewidth}
    \centering
    \includegraphics[width=.75\linewidth]{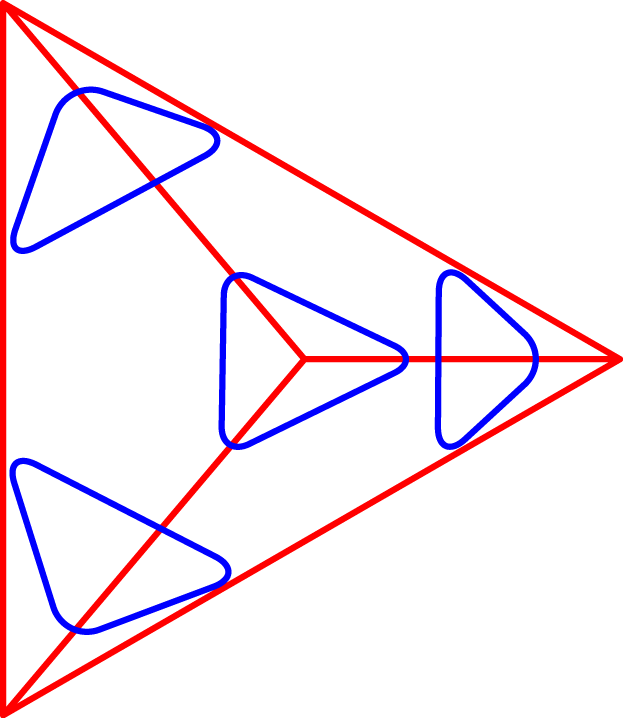}
    \caption{View along the $z$-axis of the intersections of the
      smoothed edges with planes \\$\langle
      X-v_j,\bnu_j\rangle=-\epsilon$.}
    \label{fig203}
  \end{subfigure}
  \caption{Corner rounding of a tetrahedron. Images are for
    illustrative purposes only, constructed using FreeCAD.}
\end{figure}

\subsection{Smoothing the Vertices, II}
\label{sect3}

The second method for smoothing vertices takes as its starting point
the domain $\cP_h$ constructed in the previous sections by smoothing
the edges (as depicted in Figure~\ref{fig203}). We assume that there
is a positive $\epsilon_0$ so that the intersections
\begin{equation}
 P^{h,j}_{\epsilon_0}= \cP_h\cap\{X: \langle X-v_j,\bnu_j\rangle=-\epsilon_0\}
\end{equation}
are disjoint smoothed polygons. With this assumption each vertex can
be smoothed without reference to any other vertex. We can therefore
fix a $j$ and describe the method for smoothing $\cP$ in a
neighborhood of $v_j.$

We let $0<\epsilon_j<\epsilon_0$ denote the infimum of the numbers so
that $P^{h,j}_{\epsilon}$ is a polygon with smoothed vertices. The
domain $P^{h,j}_{\epsilon_j}$ is a smoothed polygon, where the
smoothings of two (or more) of the edges meet without any flat
segment between them.

For each $\epsilon>\epsilon_j,$ we let $\Phi_{\epsilon}$ denote a
maximally smooth parameterization of $\partial P^{h,j}_{\epsilon}$ on
the unit circle. That is, $\Phi_{\epsilon}$ is a map from $S^1$ to 
$\partial P^{h,j}_{\epsilon}$. Therefore, we can represent it in terms of
a Fourier expansion:
\begin{equation}
  \Phi_{\epsilon}(\theta)=\sum_{n=-\infty}^{\infty}X_{\epsilon,n}
  e^{in\theta}-\epsilon\bnu_j .
\end{equation}
The infinite sum defines a map from the unit circle to $\partial
P^{h,j}_{\epsilon}$ translated to the plane $\langle
X,\bnu_j\rangle=0.$ We adjust $\epsilon_0$ so that
$\epsilon_0>4\epsilon_j$ for all $j.$

For each $\epsilon_j\leq\epsilon\leq \epsilon_0$ we can extend this
map as a diffeomorphism from the unit disk to the smoothed polygon
$P^{h,j}_{\epsilon}.$ For example, since the boundary of
$P^{h,j}_{\epsilon}$ is convex, it follows from a theorem of Choquet
that the harmonic extension has the desired properties:
\begin{equation}
  \tPhi_{\epsilon}(r,\theta)=\sum_{n=-\infty}^{\infty}X_{\epsilon,n}
  e^{in\theta}r^{|n|}-\epsilon\bnu_j.
\end{equation}
For additional details, 
see the next section and~\cite{Choquet1945}. The image of
$\tPhi_{\epsilon}$ lies in the 
plane~\mbox{$\langle X-v_j,\bnu_j\rangle= -\epsilon$}.

To use these maps to define a smoothing we need to choose two
auxiliary functions.  First we choose a number $\eta_1$ so that
$2\epsilon_j <\eta_1< \epsilon_0$. Next, choose a smooth, convex,
increasing function $\chi(s)$ defined in $[0,\epsilon_0]$ so that
for $s>\eta_1$, $\chi(s)=s$, $\chi(0)=\epsilon_j$,  and
\begin{equation}
  \chi^{[m]}(0)=0 \qquad \text{for } m=1,\dots,k.
\end{equation}
We also choose positive numbers $r_0<\epsilon_0,$ $\eta_2<r_0/2,$ and
an even convex function $\psi(r)$ defined in a neighborhood of $0.$ We
require
\begin{equation}
  \begin{aligned}
  \psi(r) &= r &\qquad &\text{for } r>r_0,\\
  \psi(0)&=\eta_2, & & \\
  \psi^{[m]}(0)&=0  & &\text{for } m=1,\dots,k.
  \end{aligned}
\end{equation}

The smoothing of the neighborhood is defined as the image of
$[0,\epsilon_0]\times S^1$ under the map
\begin{equation}\label{eqn30}
  \Psi(r,\theta): (r,\theta)\mapsto
  \tPhi_{\chi(r)}\left(\frac{r}{\psi(r)},\theta\right)-\chi(r)\bnu_j.
\end{equation}
For $r>\max\{\eta_1, r_0\}$ this map simplifies to
\begin{equation}
  (r,\theta)\mapsto \tPhi_{r}\left(1,\theta\right)-r\bnu_j.
\end{equation}
That is to say, its image lies in the already smoothed part of
$\partial \cP_h$ near to $v_j.$ Our assumptions assure, that as
function of $x=r\cos\theta$ and $y=r\sin\theta$, the map
$(x,y)\mapsto\Psi(x,y)$ is at least $m-1$ times differentiable in a
neighborhood of $(0,0)$ and $d\Psi(0,0)$ has rank $2$. Therefore, the
image of $D_{\epsilon_0}(0)$ under $\Psi$ is a smooth sub-manifold of
$\bbR^3$. The image lies in the set
\begin{equation}
  \langle X-v_j,\bnu_j\rangle\leq -\epsilon_j,
\end{equation}
where the vector $\bnu_j$ is the normal vector to the smoothed vertex
at the point $\Psi(0,0)$.

In describing this method for smoothing vertices, we have assumed that
the original polyhedron is convex, but this is not necessary for the
method to be applicable. It is merely required that each vertex $v_j$
has a local strict supporting plane. This means that there is a vector
$\bnu_j$ so that if $\langle v_j,\bnu_j\rangle=c_j$, then for some
$r>0$,
\begin{equation}
  \cP\cap B_{r}(v_j)\setminus v_j\subset \{X:\langle
X,\bnu_j\rangle< c_j\}.
\end{equation}
Here $B_r(v)=\{X\in\bbR^3: |X-v|<r\}.$ The existence of a strict local
support plane implies that for a range of $\epsilon>0$ the sets
\begin{equation}
 P^{j}_{\epsilon}= \cP\cap\{X: \langle X-v_j,\bnu_j\rangle=-\epsilon\}
\end{equation}
are polygons. With this assumption we proceed as before, first
smoothing the edges to produce $\cP_h.$ For some
$0<\epsilon_0<\epsilon<\epsilon_1$ the sets
\begin{equation}
 P^{h,j}_{\epsilon}= \cP_h\cap\{X: \langle X-v_j,\bnu_j\rangle=-\epsilon\}
\end{equation}
are smoothed polygons. The method described above can easily be
adapted to smooth the vertex in this case as well.  The results of
using this technique to smooth polyhedra are shown in
Figure~\ref{fig201}.

\begin{figure}[t]
  \centering
  \begin{subfigure}[t]{.3\linewidth}
    \centering
    \includegraphics[width=1\linewidth]{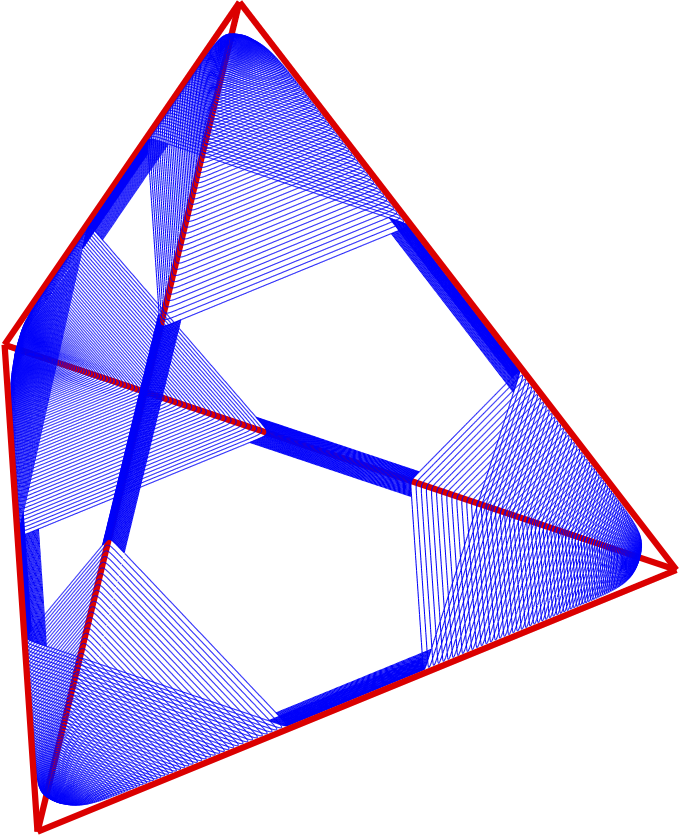}
    \caption{A smoothed tetrahedron.}
  \end{subfigure}
  \hfill
  \begin{subfigure}[t]{.3\linewidth}
    \centering
    \includegraphics[width=1\linewidth]{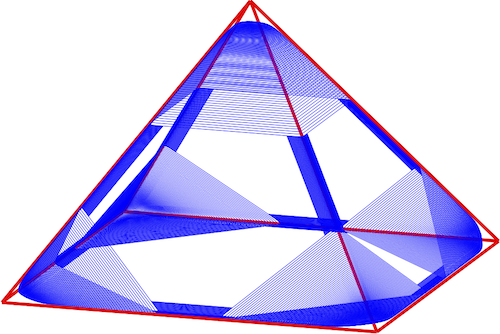}
    \caption{A smoothed pyramid.}
  \end{subfigure}
  \hfill
  \begin{subfigure}[t]{.3\linewidth}
    \centering
    \includegraphics[width=1\linewidth]{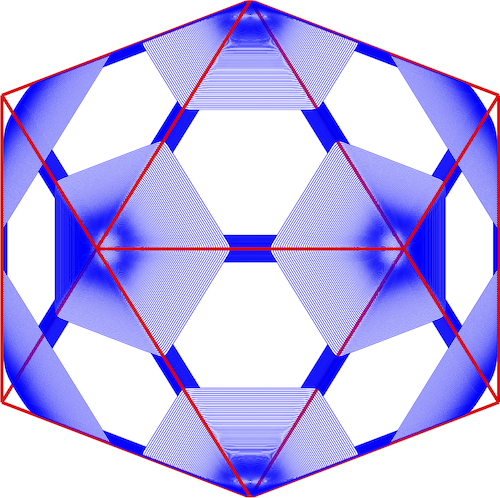}
    \caption{A smooth icosahedron.}
  \end{subfigure}
  \caption{Several smoothed polyhedra.}
  \label{fig201}
\end{figure}

\section{General Polyhedra} \label{sec-general}

Using this general scheme of first smoothing the edges, and then using
diffeomorphisms to smooth the vertices we now describe a method that
suffices to smooth arbitrary globally embedded polyhedra in $\bbR^3.$
Let $\cP$ be a polyhedron, by which we mean a bounded region in
$\bbR^3,$ whose boundary is a union of polygons lying in planes. We
let $\{v_j\}$ denote the vertices of $\cP.$ If the polyhedron has a
strict local support plane at every vertex, then the method describe
in Section~\ref{sect3} can be applied to produce a locally smoothed
polyhedron, by first smoothing the edges and then the vertices.  There
are polyhedra that do not have strict local support planes at every
vertex, e.g. the cubic torus does not have support planes at the {\em
  inner vertices}.

The method of Section~\ref{sect3} requires that near to the vertex
$v_j,$ the polyhedron is the cone over an intersection with a plane of
the form
\begin{equation}
  \cP\cap\{X: \langle X-v_j,\bnu_j\rangle=-\epsilon\}.
\end{equation}
Let $S_r(v)$ denote the sphere of radius $r$ centered at $v.$ A
slightly more complicated method results if we instead assume that for
each $j$ there is an $r_j>0$ so that
\begin{enumerate}
\item $\cP\cap S_{r_j}(v_j)$ is a connected region $R_j$ on
  $S_{r_j}(v_j)$ bounded by a simple closed curve, $\gamma_j$,
\item $\cP\cap B_{r_j}(v_j)$ is the cone over $R_j$ with vertex $v_j$.
\end{enumerate}
The curve $\gamma_j$ is a piecewise geodesic polygon on the sphere. A
polyhedron satisfying these conditions is globally embedded. In
general the region $R_j$ could have several connected components, a
case that we do not consider further.

For sufficiently small $h>0$, we let $\cP_h$ denote the result of
smoothing the edges of $\cP$ as described in Section~\ref{ss4.2}. If
$\cP$ is globally embedded, then, at each vertex there is a range of
radii $\rho_{0j}(h)<r<\rho_{1j}(h)$ so that the intersections
\begin{equation}
  \cP_h\cap S_{r}(v_j)
\end{equation}
are regions $R_j(r,h)$ bounded by simple closed curves,
$\gamma_j(r,h)$, that are smoothings of the curves $\cP\cap
S_{r}(v_j).$ As $h\to 0,$ it is clear that $\rho_{0j}(h)$ tends to $0$
and $\liminf_{h\to 0}\rho_{1j}(h)\geq r_j.$

For each $\rho_{0j}(h)<r<\rho_{1j}(h)$ we let
\begin{equation}
  \Phi_r:D_1(0)\longrightarrow R_j(r,h)
\end{equation}
be a diffeomorphism from the unit disk onto the region $R_j(r,h).$ The
maps $\{\Phi_r\}$ can, for example, be defined as the conformal maps
from $D_1(0)$ to the spherical domain $R_j(r,h),$ normalized so
that~$0$ is mapped to points lying on a carefully selected curve.
Using these maps we can define an analogue of the map
$\Psi(r,\theta),$ defined in~\eqref{eqn30}, so that the image of
$[0,\epsilon_0]\times S^1$ under this map is a smoothed version of a
neighborhood of the vertex $v_j,$ which is joined smoothly to $\cP_h.$
We leave the detailed construction of these maps to the ambitious
reader.

There are also approaches to smoothing that first smooth the vertices,
using the methods described above, and then interpolate these
smoothings along the edges. It is very difficult to preserve convexity
using this order of operations. That is why we have only described
methods that first smooth the edges, and then the vertices, using a
slicing approach along with families of diffeomorphisms.

This completes the description of our algorithms for smoothing 
polyhedra in $\bbR^3.$ Note that one can restrict the modifications
of the original polyhedron to lie in an arbitrarily specified
neighborhood of the 1-skeleton of the boundary of $P.$ One also
retains considerable control on the relationship between the Gauss map
of the smoothed polyhedron and that of the original, which is crucial
for the behavior of scattered waves. In the final section we provide
several practical methods for constructing diffeomorphisms.

\section{Methods to construct diffeomorphisms}\label{sect4}
We now describe several methods to define extensions of a map from
$S^1$ to $\Gamma,$ a Jordan curve in the plane, which are
diffeomorphisms from the unit disk $D_1(0)$ to the region
$D_{\Gamma}$ which is bounded by~$\Gamma$.

\subsection{Method 1}
Conformal mapping provides a method that can be computationally
expensive and numerically ill-conditioned (depending on the geometry)
\cite{kythe}, but guaranteed to work in considerable generality. In
particular $D_{\Gamma}$ can be a simply connected region in either a
plane or a round sphere.  Suppose that $f:D_1(0)\to D_{\Gamma}$ is a
conformal diffeomorphism. If $\Gamma$ is convex, then
\begin{equation}
  \Gamma_r=\{f(re^{i\theta}):\: \theta\in [0,2\pi]\}
\end{equation}
is a convex curve for every $r\in (0,1]$. If $\Gamma$ is star shaped
with respect to $0$ and $f$ is normalized so that $f(0)=0$ then the
curves $\{\Gamma_r:r\in (0,1]\}$ are star shaped. These results can be
  found in~\cite{pommerenke}.

\subsection{Method 2}

There is a simple method that is guaranteed to give a diffeomorphism
if $D_{\Gamma}$ lies in a plane and the initial curve $\Gamma$ is
convex. A theorem of T. Rado, Kneser, and G. Choquet states that if
$(u,v)$ defines a homeomorphism from the unit circle to $\Gamma$ which
bounds a convex region $D_{\Gamma}$, then the harmonic extension of
the coordinate functions $(U,V)$ defines a diffeomorphism from the
interior of $D_1$ to~$D_{\Gamma}$~\cite{Choquet1945}. This
theorem does not require $\Gamma$ to be strictly convex or smooth.

If the boundary map is given in terms of the Fourier series
\begin{equation}
  \theta\mapsto
  \left(\sum_{j=-\infty}^{\infty}a_j \, e^{ij\theta}, \,
  \sum_{j=-\infty}^{\infty}b_j \, e^{ij\theta}\right)
\end{equation}
it follows from Choquet's theorem that
\begin{equation}
  \Phi(\theta,r)=
  \left(\sum_{j=-\infty}^{\infty}a_j \, r^{|j|} e^{ij\theta}, \,
  \sum_{j=-\infty}^{\infty}b_j \, r^{|j|} e^{ij\theta}\right)
\end{equation}
defines a diffeomorphism from $D_1(0)$ onto $D_{\Gamma}.$

\subsection{Method 3}
If we specify a convex curve $\Gamma$ in terms of its Gauss map, that
is, as the image
\begin{equation}\label{eqn28}
  G(\theta)=g(\theta)\, \left(\cos\theta,\sin\theta\right)
  +g'(\theta)\, \left( -\sin\theta,\cos\theta \right),
\end{equation}
then we can proceed as above to get a diffeomorphism. If $g$ has the
Fourier representation
\begin{equation}\label{eqn29}
  g(\theta)=\sum_{n=-\infty}^{\infty}\beta_n \, e^{in\theta},
\end{equation}
then we can again apply Choquet's theorem to construct a harmonic
extension, which is guaranteed to give a diffeomorphism. The map
defined in~\eqref{eqn28} can be represented as
\begin{equation}
  e^{i\theta}\mapsto (g(\theta)+ig'(\theta))e^{i\theta}.
\end{equation}
Using the Fourier representation in~\eqref{eqn29}, we see that
\begin{equation}
  G(\theta,r)=\sum_{n=-\infty}^{\infty}\beta_{n-1}\, (2-n) \, r^{|n|}e^{in\theta}
\end{equation}
defines the harmonic extension of this map, and is therefore a
diffeomorphism from $D_1(0)$ onto $D_{\Gamma}.$

\section{Conclusions} \label{sec-conclusions}

In this paper we have presented several algorithms for modifying
polygons and polyhedra into fully regularized surfaces without
geometric singularities (vertices and edges). The original polygon (or
polyhedron) is modified in a controllable and arbitrarily small
neighborhood of its singular set. We have compared the solution to
acoustic scattering problems from the original singular boundary to
that obtained by smoothing the boundary at  sub-wavelength scales in
two dimensions. Both near- and far-field solutions converge at a rate
slightly faster than first-order in the rounding
parameter. Understanding this rate of convergence is an ongoing
research topic in our group.

Constructing numerical codes for performing rounding in two dimensions
is relatively straightforward. We presented results for the polygonal
case; software implementing the rounding of vertices joining
piecewise smooth curves is currently under development, requiring
merely the re-parameterization of the curve near the vertex as a graph
above a support line tangent to the vertex. These computations are
relatively fast, efficient, and accurate to near machine precision in
two dimensions.

We also introduced the analytical foundation for constructing
high-order roundings of polyhedra in three dimensions. Composing
various methods with diffeomorphisms near vertices allows for similar
regularizations to be computed as in the two-dimensional case. Building
more efficient software to perform these computations is a work in
progress.

Preliminary \textsc{Matlab} 
code which performs the vertex and edge smoothing for convex
polyhedra in three dimensions has been made available at:
\begin{center}
\url{http://gitlab.com/oneilm/rounding}
\end{center}

If only approximate scattering solutions are required to the true
problem involving geometries with corners and edges, the algorithms of
this paper offers a method to obtain these results with reduced
computational cost and  controlled accuracy. Furthermore,
the methods require nothing other than the usual quadratures for
weakly-singular functions on smooth curves or surfaces. Full
extensions of these smoothing algorithms to three dimensions may have
a wide array of applications in high-order CAD and CAE packages, as
many existing software solutions only allow for twice differentiable
roundings (fillets).

Lastly, we would like to note that the algorithms presented in this
paper for geometric regularization in three dimensions are only one
piece of a larger effort to develop high-order scattering codes for
arbitrary geometries. In three dimensions, all the numerical tools
that are required to solve boundary integral equations are more
expensive and more sophisticated than those in two dimensions. Merely
constructing high-order Nystr\"om-compatible quadratures for the
function $1/|\bx-\by|$ along triangular patches is a relatively recent
result~\cite{bremer_2012c,bremer_2013}. Coupling these schemes with
fast algorithms and high-order triangulations is under active
development.  Performing the analogous convergence studies for
rounding in three dimensions will  be reported at a later
date, after the requisite high-order accurate
computational PDE algorithms have been developed.

\bibliographystyle{siam}
\bibliography{../preprint}

\end{document}